\documentclass[12pt]{amsart}
\usepackage{amsfonts, amsmath, amsthm}
\usepackage{graphicx}
\usepackage{fourier,microtype}
\usepackage{geometry}
\usepackage{eepic}
\usepackage{epstopdf}

\newtheorem{thm}{Theorem}[section]

\newtheorem{cor}[thm]{Corollary}
\newtheorem{prop}[thm]{Proposition}
\newtheorem{lem}[thm]{Lemma}

\theoremstyle{definition}
\newtheorem{dfn}[thm]{Definition}
\newtheorem{notation}[thm]{Notation}

\newtheorem{que}[thm]{Question}

\theoremstyle{remark}
\newtheorem{rmk}[thm]{Remark}
\newtheorem{rmks}[thm]{Remarks}

\def\vol[#1]{\mbox{\rm Vol}(#1)}
\def\dfar{\ensuremath{d_{\mbox{Farey}}}}

\def\pf{\subset_{p.f.}} 

\def\lv[#1]{\mbox{\rm LinkVol}(#1)}
\def\lvd[#1#2]{\mbox{\rm LinkVol}_{#1}(#2)}
\def\lva[#1]{\mbox{\rm LinkVol}_{{X},d}(#1)}
\def\lvsd[#1#2]{\mbox{\rm LinkVol}_{s,{#1}}(#2)}

\def\nos{\ensuremath\infty}
\def\mnh{minimally non hyperbolic}

\def\kv[#1]{\mbox{\rm KnotVol}(#1)}
\def\kvd[#1#2]{\mbox{\rm KnotVol}_{#1}(#2)}
\def\kvs[#1]{\mbox{\rm KnotVol}_{s}(#1)}
\def\kvsd[#1#2]{\mbox{\rm KnotVol}_{s,{#1}}(#2)}

\def\depth[#1]{\mbox{\rm depth}(#1)}

\title[The Link Volume of Hyperbolic 3-Manifolds]{Cosmetic Surgery \\[5pt] 
and The Link Volume of Hyperbolic 3-Manifolds}

\date{\today}  
\address{Department of mathematical Sciences, University of Arkansas, 
Fayetteville, AR 72701} %
\address{Department of Information and Computer Sciences, Nara Women's
  University  Kitauoya Nishimachi, Nara 630-8506, Japan} %
\email{yoav@uark.edu} %
\email{yamasita@ics.nara-wu.ac.jp} %
\author{Yo'av Rieck} %
\author{Yasushi Yamashita} %


\begin{document}
\begin{abstract}
We prove that for any $V>0$, there exist a hyperbolic manifold $M_{V}$, so that
$\vol[M_{V}] < 2.03$ and $\lv[M_{V}] > V$.  

The proof requires study of cosmetic surgery on links (equivalently, fillings
of manifolds with boundary tori).  There is no bound on the number of components
of the link (or boundary components).  For statements, see the second part of the 
introduction.  Here are two examples of the results we obtain:
	\begin{enumerate}
	\item Let $K$ be a component of a link $L$ in $S^{3}$.  Then ``most'' slopes on $K$
	cannot be completed to a cosmetic surgery on $L$, unless $K$ becomes a component
	of a Hopf link.
	\item Let $X$ be a manifold and $\epsilon>0$.  Then all but finitely
	many hyperbolic manifolds obtained by filling $X$ admit a geodesic shorter than 
	$\epsilon$.
	\end{enumerate}
\end{abstract}

\maketitle

\setcounter{tocdepth}{1}
\tableofcontents

\section{Introduction}
\label{sec:intro}

In~\cite{rieckyamashita} we defined an invariant of closed orientable 3-manifolds that measures how
efficiently a given manifold $M$ can be represented as a cover of $S^{3}$ where the
branch set is a hyperbolic link (such a cover can be constructed using Hilden~\cite{hilden}
or Montesinos~\cite{montesinos}).  
We use the notation $M \stackrel{p}\to (S^{3},L)$ to denote a $p$-fold cover $M \to S^{3}$ branched
over $L$.  To the cover $M \stackrel{p}\to (S^{3},L)$
we associate the complexity $p \vol[S^{3} \setminus L]$.
The {\it link volume} of $M$ is defined to be
the infimum of the complexities of all possible covers, that is
$$\lv[M] = \inf\{\ p \vol[S^{3} \setminus L]\ |\ M \stackrel{p}\to (S^{3},L)\}.$$
We observed that for any hyperbolic manifold $M$, $\lv[M] > \vol[M]$, and conjectured that the link volume
cannot be bounded in terms of the hyperbolic volume.  In this paper we prove this conjecture.  
Our main result is (see the next two sections for definitions; by $M(\alpha_{i})$ we mean the manifold obtained 
by filling $M$ along slope $\alpha_{i}$ and  $\dfar$
denotes the distance in the Farey graph):
\begin{thm}
\label{thm:main}
Let $M$ be a hyperbolic manifold with one cusp.  Let $\{\alpha_{i}\}$
be a set of slopes of $\partial M$.

Then there exists $L>0$ so that $\lv[M(\alpha_{i})] < L$ for all $i$ if and only if
there exists $d>0$ so that $\dfar(\alpha_{i},\alpha_{i'})<d$ for all $i,i'$.
\end{thm}

\begin{rmk} One direction is known:
if there exists $d>0$ so that $\dfar(\alpha_{i},\alpha_{i'})<d$ for all $i,i'$,
then~\cite{rieckyamashita} implies that  there exists $L>0$ so that $\lv[M(\alpha_{i})] < L$ for all $i$.
\end{rmk}

Let $M$ denote the figure eight knot exterior; then $\vol[M] = 2.02988\dots$.
By Cao and Meyerhoff~\cite{caomeyerhoff}, $M$ has the smallest volume
among all cusped hyperbolic 3-manifolds.
Applying Theorem~\ref{thm:main} to manifolds obtained by filling $M$ we get
the following corollary:  

\begin{cor}
For every $V>0$, there exists a hyperbolic manifold $M_{V}$, so that 
$$\vol[M_{V}] < 2.02988\dots \hspace{10pt} \mbox{and} \hspace{10pt} \lv[M_{V}] > V.$$
\end{cor}

This corollary can be interpreted (negatively) as saying that representing 
manifolds as branched covers of $S^{3}$ is inefficient.  On the positive side,
it shows that the link volume is a truly new invariant.

Not much is known about the link volume of specific manifolds; Rieck and Remigio--Ju\'arez~\cite{remigiorieck}
calculated the link volume of certain prism manifolds (prism manifolds are small Seifert fibered spaces),
but the link volume is not known for {\it any} hyperbolic manifold.  We refer the reader to~\cite{rieckyamashita} 
for basic facts and open questions about the link volume.  In particular, the upper bound obtained in
\cite{rieckyamashita} is explicit, and linear in terms of distance in the Farey tessellation.  It would be
nice if a similar lower bound could be proved.
The following question is obtained by simply reversing the inequality
in the upper bound of \cite{rieckyamashita} (by $M(\alpha_{1},\dots,\alpha_{n})$ we
mean the manifold obtained by filling $M$ along slopes $\alpha_{1},\dots,\alpha_{n}$):

\begin{que}
Let $M$ be a compact orientable connected 3-manifold with toral boundary.
We will denote the components of $\partial M$ as $T_{1},\dots,T_{n}$.  For each $i$,
fix a slope $\beta_{i}$ of $T_{i}$.

Then do there exist $A,\ B>0$,
so that for any choice of slope $\alpha_{i}$ of $T_{i}$,
$$\lv[M(\alpha_{1},\dots,\alpha_{n})] > A + B(\Sigma_{i=1}^{n} \dfar(\alpha_{i},\beta_{i}))?$$
\end{que}

The work in this paper is an application of the Structure Theorem of~\cite{rieckyamashita}.
The Structure Theorem states that for any $V>0$, there is a finite set of ``parent systems''
$\{\phi_{i}:X_{i} \to E_{i}\}_{i=1}^{n}$, where $X_{i}$ and $E_{i}$ are hyperbolic manifolds and 
$\phi_{i}:X_{i} \to E_{i}$ is an unbranched cover, so that for any manifold $M$
with $\lv[M] < V$, there is some $i$ so that the following diagram commutes:
\medskip
\begin{center}
\begin{picture}(200,60)(0,0)
  \put(  0,  0){\makebox(0,0){$E_{i}$}}
  \put(  0,50){\makebox(0,0){$X_{i}$}}
  \put(  0, 40){\vector(0,-1){30}}   
  \put( 10,  0){\vector(1,0){70}}
  \put(100,  0){\makebox(0,0){$(S^3,L)$}}
  \put(100, 40){\vector(0,-1){30}}
  \put(10,28){\makebox(0,0){$/\phi_{i}$}}
  \put(108,28){\makebox(0,0){$/\phi$}}
  \put(108,28){\makebox(0,0)}
  \put(100,50){\makebox(0,0){$M$}}
  \put( 10,50){\vector(1,0){80}}
\end{picture}
\end{center}
\medskip
where here the horizontal arrows denote inclusions induced by fillings (that is, attaching solid tori) and 
$\phi:M \to S^{3}$ is a cover that realizes the link volume, that is, a cover for which
$\deg(\phi) \vol[S^{3} \setminus L] = \lv[M]$.
The reader may observe the similarity to the celebrated result of J\o rgensen and Thurston~\cite{thurston}
that states that for every $V>0$ there are finitely many ``parent manifolds'' $\{E_{i}\}_{i=1}^{n}$ so that
every hyperbolic manifold of volume less than $V$ can be obtained by filling $E_{i}$
for some $i$.
(For a detailed exposition see, for example,~\cite{kobayashirieck}.)  This is no coincidence;
the Structure Theorem is a consequence of J\o rgensen--Thurston and the manifolds $E_{i}$
appearing in it are the parent manifolds of 
J\o rgensen andThurston.

\bigskip\bigskip\bigskip\noindent
In order to obtain Theorem~\ref{thm:main} from the Structure Theorem above, 
we are forced to study several questions
about fillings.  Specifically, we study several questions about {\it cosmetic surgery},
that is, surgery on a link in a manifold $M$ that results in a manifold diffeomorphic to $M$.
Before describing our methods and the results obtained, we 
introduce the basic setup; detailed
description is given in Subsection~\ref{subsec:notaion}.

After two preliminary sections (\ref{sec:background} and~\ref{sec:BoundedSets}),
in Sections~\ref{sec:mnh} and~\ref{sec:T(X)} we construct our main tool, a rooted tree denoted
$T(X)$ which we associate with a manifold $X$.  The vertices of $T(X)$ are labeled by
labels that correspond to manifolds; $X$ itself corresponds to the root of $T(X)$.  
If $X$ is hyperbolic, its immediate descendants are certain non-hyperbolic
manifolds obtained by filling $X$; if $X$ is not prime, its immediate descendants are the factors of its 
prime decomposition; if $X$ is JSJ (that is, if $X$ is prime and the collection of tori in the
JSJ decomposition of $X$ is not empty),
its immediate descendants are the components of its torus decomposition.  If $X$ is Seifert
fibered or solv it has no descendants.  We prove
(Proposition~\ref{prop:T(X)isFinite}) that $T(X)$ is finite.  All the results described below
are proved by induction on $|T(X)|$, the number of vertices in $T(X)$.
The various applications of $T(X)$ are somewhat independent, and we made an effort
to make the following sections (especially 
Sections~\ref{sec:CosmeticSurgeryOnT2XI}--\ref{section:CosmeticSurgeryOnS3}
and~\ref{sec:FillingsThatDontFactorThroughM}) independently readable.
Throughout this paper a set of slopes of a torus is called {\it bounded}
if it is bounded in the Farey graph.

In Section~\ref{sec:CosmeticSurgeryOnT2XI} we study cosmetic surgery on a link $L \subset T^{2} \times [0,1]$.
Let $B$ be a bounded set of slopes of $T^2 \times \{1\}$. 
By a {\it multislope} $\alpha$ of $L$ we mean a vector whose components are slopes on the components
of $L$ or $\nos$ (see Subsection~\ref{subsec:notaion} for a precise definition).
We will denote the manifold obtained by surgery on $L$ with multislope $\alpha$ as $L(\alpha)$.
Let $\mathcal{A}$ be the multislopes of $L$ that yield cosmetic surgery, that is, 
$\mathcal{A} = \{\alpha\ |\ L(\alpha) \cong T^{2} \times [0,1]\}$.
Given $\alpha \in \mathcal{A}$, we may use the product
structure of $L(\alpha)$ to project the set $B$ and obtain a set of slopes of $T^{2} \times \{0\}$.
Since this set depends on $\alpha$ we will denote it as $B_{\alpha}$.  
We prove the $T^{2} \times [0,1]$ Cosmetic Surgery Theorem (\ref{thm:CosmeticSurgeryOnT2XI})
that says that
$$\bigcup_{\alpha \in \mathcal{A}} B_{\alpha}$$ 
is a bounded set of slopes of $T^{2} \times [0,1]$.

In Section~\ref{sec:CosmeticSurgeryOnSolidTorus} we study cosmetic surgery on a link
$L \subset D^{2} \times S^{1}$.
We prove the Solid Torus Cosmetic Surgery Theorem (\ref{thm:SlopesOnSolidTorus}), that 
says that the set of slopes of $\partial D^{2} \times S^{1}$
that bound a disk after cosmetic surgery on $L$ is bounded
(unless some component of $L$ is a core of the solid torus after surgery, in which case the 
claim is obviously false).   For use in later sections we also prove 
Proposition~\ref{pro:SolidTorusSurgery2}, which gives certain constraints
on multislopes of $L$ that yield a cosmetic surgery.

Sections~\ref{section:HyperSurgSlopes} and~\ref{section:HyperSurgRadInj} are
devoted to cosmetic surgery on hyperbolic manifolds.  Let $M$ be a hyperbolic
manifold, $L \subset M$ a link, $T$ a component of $\partial M$,
$B$ a bounded set of slopes of $T$, and
$\mathcal{X} = \{(\alpha,f_{\alpha})\}$ so that for every $(\alpha,f_{\alpha})$,
$\alpha$ is a multislope of $L$ and $f_{\alpha}$
is a diffeomorphism $f_{\alpha}:L(\alpha) \cong M$ that maps $T$ to itself.  
Then for every $(\alpha,f_{\alpha}) \in \mathcal{X}$ the image of $B$ under
$f_{\alpha}$ is a set of slopes of $T$ that we will denote as $B_{\alpha,f_{\alpha}}$.
In Section~\ref{section:HyperSurgSlopes} we prove 
Theorem~\ref{thm:CosmeticSurgeryOnM} that says that 
$$\bigcup_{(\alpha,f_{\alpha}) \in \mathcal{X}} B_{\alpha,f_{\alpha}}$$
is a bounded set of slopes of $T$.

In Section~\ref{section:HyperSurgRadInj} we consider multislopes $\alpha$ of a manifold $X$
that yield a hyperbolic manifold $X(\alpha)$, so that every geodesic in $X(\alpha)$ is
longer than $\epsilon$ (for some fixed $\epsilon>0$; here $\alpha$ is a multislope on
$\partial X$ and $X(\alpha)$ represents filling rather than surgery).  We prove two things: first, there
are only finitely many such manifolds $X(\alpha)$ (although there certainly may be infinitely 
many such multislopes $\alpha$).  Second, we prove that all but finitely many of these multislopes
factor through a non-hyperbolic filling.  In other words, there is a subset of the boundary components
so that the manifold obtained by filling only these components is not hyperbolic 
(we call this a non hyperbolic {\it partial filling}; this and
other useful terminology is introduced in Subsection~\ref{subsec:notaion}).  
Moreover, one of the non hyperbolic partial fillings corresponds to an edge out
of the root of $T(X)$; this is the key that allows us to use induction.

In Section~\ref{section:CosmeticSurgeryOnS3} we prove the $S^{3}$ Cosmetic Surgery 
Theorem~(\ref{thm:cosmeticSurgeryOnS3}): let $L \subset S^{3}$ be a link and $K$ a component
of $L$.  Let $\mathcal{A} = \{\alpha \ | \ L(\alpha) = S^{3}\}$, and  $\mathcal{A}' \subset \mathcal{A}$ 
be the multislopes for which the core of the solid torus attached to $\partial N(K)$ does
not form a Hopf link with the core of any other attached solid torus.  Given 
a multislope $\alpha$, we denote its value on $K$  as $\alpha|_{K}$.  We prove that
$$\{\alpha|_{K} \ | \ \alpha \in \mathcal{A}'\}$$
is bounded.

Finally, in Sections~\ref{sec:SetUpOfProof}--\ref{sec:BothFillingsFactor} we apply these
results to prove Theorem~\ref{thm:main}.  In particular, in Section~\ref{sec:FillingsThatDontFactorThroughM}
we prove Theorem~\ref{thm:fillingsThatDontFactor} which is of independent interest.  
In it we consider manifolds $X$ and $M$, where $M$ is a one cusped hyperbolic manifold.
We consider the set of multislopes $\mathcal{A}$ of $\partial X$ so that any $\alpha \in \mathcal{A}$
fulfills the following condition: 
the manifold obtained by filling all but one boundary component of $X$
is not $M$.  We describe this by saying that $\alpha$ does not admit a partial filling $\alpha'$
for which $X(\alpha') \cong M$.
In Theorem~\ref{thm:fillingsThatDontFactor} 
we show that the 
set of slopes $\beta$ so that $M(\beta) \cong X(\alpha)$ for some $\alpha \in \mathcal{A}$ 
is bounded.

%
%

\bigskip

\noindent{\bf Acknowledgements.}  We thank Matt Day, Chaim Goodman--Strauss,
Mike Hilden, Kazuhiro Ichihara, Tsuyoshi Kobayashi,
Marc Lackenby, Kimihiko Motegi, John Retcliffe, and Masakazu Teragaito for helpful conversations and
correspondence.  A very special thanks to Michael Yoshizawa whose patient
reading of an early version of this paper helped us greatly.
YR: this research took place during three visits to Nara Women's University.  I 
would like to thank the Math Department and the Department of Information and Computer Science 
for their hospitality during my long stays.

\section{Background}
\label{sec:background}

Throughout this paper, by {\it manifold} we mean a compact orientable connected 3-dimensional smooth 
manifold.  We only consider manifold with toral boundary, that is, manifolds
whose boundary consists of a (possibly empty) collection of tori.  A manifold is called {\it hyperbolic} if its
interior admits a complete finite volume Riemannian metric locally isometric to hyperbolic space $\mathbb H^{3}$;
we sometimes refer to the boundary components of a hyperbolic manifold as {\it cusps}.
We denote closed normal neighborhood, closure, and interior by $N(\cdot)$,
$\mbox{cl}$, and $\mbox{int}$, respectively. The geometric intersection number
between curves on a torus is denoted $\Delta(\cdot,\cdot)$.  Given a knot or a link $L \subset M$, 
we call $M \setminus \mbox{int}N(L)$ the {\it exterior} of $L$ and denote it as $E(L)$.
We {\it always} assume transversality.

\subsection{Notation}
\label{subsec:notaion}
The following notation will be used extensively throughout the paper. 
Let $X$ be a manifold, fix $n$ components of $\partial X$ denoted as $T_{1},\dots,T_{n}$, 
and denote their union as 
$\mathcal{T} = \cup_{i=1}^{n} T_{i}$ 
(note that $\mathcal{T} \subset \partial X$, but possibly $\mathcal{T} \neq \partial X$). 

\begin{enumerate}
\item By a {\it multislope} of $\mathcal{T}$, say $\alpha$, we mean a vector $\alpha = (\alpha_{1},\dots,\alpha_{n})$
so that for each $i$, $\alpha_{i}$ is either the homology class of a connected simple closed curve on $T_{i}$ or
$\alpha_{i} = \infty$.  By a multislope of a link $L$ we mean a multislope of $\partial N(L) \subset \partial E(L)$.  
\item By {\it filling} $X$ along $\alpha$ we mean the manifold $X(\alpha)$ that is obtained by the following 
operation:
	\begin{enumerate}
	\item To components $T_{i} \subset \mathcal{T}$ for which $\alpha_{i} \neq \infty$ we attach 
	a solid torus $V_{i}$ so that the meridian of  $V_{i}$ is identified with a connected simple closed 
	curve representing $\alpha_{i}$.
	\item Nothing is done to components $T_{i} \subset \mathcal{T}$ for which $\alpha_{i} = \nos$ and 
	components of $\partial X \setminus \mathcal{T}$. 
	\end{enumerate}
\item If $\alpha = (\alpha_{1},\dots,\alpha_{n})$ and $\alpha'= (\alpha'_{1},\dots,\alpha'_{n})$ are multislopes, 
we say that $\alpha'$ is a {\it partial filling} of $\alpha$, which we will
denote as $\alpha' \pf \alpha$, if for each $i$, $\alpha'_{i} \in \{\alpha_{i},\infty\}$.  If $\alpha' \pf \alpha$ and 
$\alpha' \neq \alpha$ we say that $\alpha'$ is a {\it strict} partial filling of $\alpha$. 
We will also use the notation $(\alpha_{1},\dots,\hat{\alpha}_{i},\dots,\alpha_{n})$ for the multislope
obtained from $(\alpha_{1},\dots,{\alpha}_{i},\dots,\alpha_{n})$ by replacing $\alpha_{i}$ with
$\nos$ (intuitively, toosing $\alpha_{i}$ out). 
\item Assume that $X$ is hyperbolic.  A multislope $\alpha$ is called {\it hyperbolic} if $X(\alpha)$ is hyperbolic;
$\alpha$ is called {\it totally hyperbolic} if every partial filling of $\alpha$ is hyperbolic.
\item Assume that $X$ is hyperbolic.  A multislope $\alpha$ is called {\it non-hyperbolic} 
if $X(\alpha)$ is not hyperbolic.  If $\alpha$ is non-hyperbolic, and every strict partial filling of $\alpha$ is hyperbolic, then $\alpha$ is called {\it \mnh}.  Minimally non hyperbolic fillings are
studies extensively in Section~\ref{sec:mnh}.
\item Let $\mathcal{T}' \subset \partial X$ be a union of components of $\partial X$ and $\alpha$
a multislope of $\mathcal{T}$.  Then $\alpha$ defines a multislope on $\mathcal{T}'$
by removing the components of  $\alpha$ that correspond to components of
$\mathcal{T} \setminus \mathcal{T}'$ and assigning the value $\nos$ to every 
component of $\mathcal{T}' \setminus \mathcal{T}$.  
This multislope is called the {\it restriction} of $\alpha$ to $\mathcal{T}'$ and
we will denote it as $\alpha|_{\mathcal{T}'}$.
In particular, we will denote the value of $\alpha$ on $T_{i}$ as $\alpha|_{T_{i}}$.  
$\alpha|_{\mathcal{T}'}$
is also called the multislopes {\it induced} by $\alpha$ on $\mathcal{T}'$.
\item Induced multislopes also appear in a more general setting: let $F$ be a collection of tori in $\mbox{int}(X)$.
Suppose that $X$ cut open along $F$ consists of $N_{1}$ and $N_{2}$, that is: $X = N_{1} \cup_{F} N_{2}$.
Here we are {\it not} assuming that $N_{1}$ or $N_{2}$ is connected.  Let $\alpha$ be a multislope of 
components of $\partial X$ (denoted $\mathcal T$) so that, for each component $Y$ of $N_{2}$,
$Y(\alpha|_{\partial Y}) \cong D^{2} \times S^{1}$.  
In other words, after filling, $N_{2}(\alpha|_{\partial N_{2}})$ consists of solid tori.
We will denote $(\mathcal{T} \cap \partial N_{1}) \cup F$ as $\mathcal{T}_{1} \subset \partial N_{1}$.
Then the multislope of $\mathcal{T}_{1}$ {\it induced} by $\alpha$ is the multislope defined by 
$\alpha|_{\mathcal{T} \cap \partial N_{1}}$
(on the components of $\mathcal{T} \cap \partial N_{1}$) and the homology classes of the meridians of 
$N_{2}(\alpha|_{\partial N_{2}})$  (on the components of $F$).
\item  If $L \subset M$ is a link then a {\it multislope} of $L$ is a multislope of 
$\partial N(L) \subset \partial E(L)$.  By {\it surgery} on $L$ with multislope $\alpha$,
which we will denote as $L(\alpha)$, we mean $E(L)(\alpha)$. 
\end{enumerate}

\subsection{JSJ-manifolds}

In this subsection we summarize the information we need about manifolds with non-trivial JSJ
decomposition.  JSJ decompositions were studied by Jaco and Shalen~\cite{jacoshalen} and,
independently, by Johannson~\cite{Johannson}.  We assume familiarity with this subject;
further details can be found in~\cite{JacoBook}.  We summarize what we need in the definition below;
note that since we restrict our attention to manifolds with boundary tori, we may assume that the JSJ
decomposition is along tori (and no annuli). 

\begin{dfn}[JSJ-manifold]
\label{dfn:jsj}
Let $X$ be a compact 3-manifold so that $\partial X$ consists of a (possibly empty) collection of tori.  We say that
$X$ is a {\it JSJ-manifold} if $X$ is irreducible and the JSJ-decomposition of $X$ consists
of a non-empty collection of tori which we will denote as $\mathcal{T}$.  In that case we also say that $X$
has a {\it non-trivial} JSJ decomposition.  The manifolds obtained by cutting $X$ along $\mathcal{T}$
are called the {\it components of the torus decomposition of $X$}.  The graph dual to the JSJ decomposition
of $X$ has one vertex for every component of the torus decomposition of $X$, and an edge for every
torus in $\mathcal{T}$; the edge corresponding to $T \in \mathcal{T}$ connects the (not necessarily distinct) 
vertices that correspond to the components of the torus decomposition of $X$ that have $T$ in their boundary.
\end{dfn}

\subsection{Topological preliminaries} 

We will need the following simple lemma which says that knot exteriors can't be 
``linked'' with each other: 

\begin{lem}
\label{lem:KnotExteriorsAreDisjoint}
Let $Y$ be a manifold and for $i=1,\dots,p$, let $E_{i}$ be a non-trivial knot exterior embedded in a ball
$D_{i} \subset Y$.  Suppose that for $i \neq i'$, $E_{i} \cap E_{i'} = \emptyset$.

Then we may choose the balls $D_{i}$ so that for $i \neq i'$, $D_{i} \cap D_{i'} = \emptyset$.
\end{lem}

\begin{proof}
We assume that $E_{i} \cap E_{i'} = \emptyset$ during all the isotopies considered in this proof (for $i \neq i'$).
Assume, for a contradiction, that the lemma is false and let $Y$ and
$E_{i}$ form a counterexample that minimizes $p$; note that $p \geq 2$.  Then 
there exist disjoint balls $D_{1},\dots,D_{p-1}$ so that
$E_{i} \subset D_{i}$.  
Assume first that there is no isotopy of $E_{p}$ so that
$E_{p} \cap (\cup_{i=1}^{p-1} D_{i}) = \emptyset$.  Let $m_{i} \subset D_{i}$ be a meridian disk for $E_{i}$,
$i=1,\dots,p-1$; note that $D_{i}$ is isotopic to $E_{i} \cup N(m_{i})$.  
Denote $M = \cup_{i=1}^{p-1} m_{i}$.  Minimize $|M \cap E_{p}|$;
since $E_{p}$ cannot be isotoped to be disjoint from $\cup_{i=1}^{p-1} D_{i}$,
$|M \cap E_{p}| > 0$.  Let $\delta \subset M$ be an innermost disk of $M \cap \partial E_{p}$. 
If $\partial \delta$ is inessential in $\partial E_{p}$, we use an innermost disk from
$\partial E_{p}$ to surger $M$ and reduce $|M \cap E_{p}|$.
This gives new (and not necessarily isotopic) 
meridian disks for $E_{1},\dots,E_{p-1}$; we still denote these disks by $m_{1},\dots,m_{p-1}$
and $\cup_{i=1}^{p-1}m_{i}$ by $M$.  We replace the balls $D_{i}$ with $E_{i} \cup N(m_{i})$,
which we will continue to denote by $D_{i}$.  We repeat this
process until one of the following holds:
	\begin{enumerate}
	\item $|M \cap E_{p}| > 0$ and $\partial\delta$ is essential in $\partial E_{p}$:  
	in that case, $\delta$ is a meridian disk for $E_{p}$.  
	Let $\hat D_{p} = E_{p} \cup N(\delta)$.  Then $\hat D_{p} \cap (\cup_{i=1}^{p-1} E_{i}) = \emptyset$.
	By isotopy, we move $M$ out of $\hat D_{p}$; thus we see that 
	$\{D_{i},\dots,D_{p-1},\hat D_{p}\}$ is a collection of disjointly embedded balls contradicting our assumption.
	\item $|M \cap E_{p}| = 0$: in this case $E_{p} \cap (\cup_{i=1}^{p-1} D_{i}) = \emptyset$.  Let $m_{p}$
	be a meridian disk for $E_{p}$; by isotopy we move
	$m_{p}$ out of $\cup_{i=1}^{p-1} D_{i}$.   Then 
	$\{D_{i},\dots,D_{p-1},E_{p} \cup N(m_{p})\}$ is a collection of 
	disjointly embedded balls contradicting our assumption.
	\end{enumerate}
\end{proof}

\begin{dfn}
\label{dfn:unlink}
Let $Y$ be a connected manifold.  Let $U \subset Y$ be a link.
We say the $L$ is an {\it unlink} if the components of $L$
bound disjointly embedded disks.  
\end{dfn}
In the following lemma we consider two links in a manifold $Y$, which we will denote as 
$\mathcal{L}$ and $\mathcal{U}$.  $\mathcal{U}$ is an unlink; hence we can perform 
$1/n$ surgery about any component $K$ of $\mathcal{U}$ without changing
the ambient manifold; here the framing is chosen so that the boundary of the disk bound by 
$K$ corresponds to $0/1$.  After this surgery say that the components of $\mathcal{L}$ 
are {\it twisted about $K$ $n$ times}.
This process may be repeated on the other component of $\mathcal{U}$.
We assume that $Y \setminus \mbox{int}N(\mathcal{L} \cup \mathcal{U})$ is irreducible,
but since $\mathcal{L}$ and $\mathcal{U}$ play different roles we phrase this condition differently,
see Condition~(2) below.
We are interested in the effect on $\mathcal{L}$ when $n$ is large:

\begin{lem}
\label{lem:TwistingToGetIrreducible}
Let $Y$ be a manifold and $\mathcal{L}$, $\mathcal{U}$ links in $Y$.  Assume that the following conditions hold:
	\begin{enumerate}
	\item $\mathcal{U}$ is an unlink.
	\item $\mathcal{L}$  is irreducible in the complement of $\mathcal{U}$.
	\end{enumerate}
Then the link $\mathcal{L}'$ obtained from $\mathcal{L}$ by twisting its components about each component of 
$\mathcal{U}$ sufficiently many times has an irreducible exterior.
\end{lem}

\begin{proof}
We induct on $|\mathcal{U}|$, the number of components in $\mathcal{U}$.  If $|\mathcal{U}| = 0$
there is nothing to prove.  Otherwise, let $K$ be a component of $\mathcal{U}$.  
We will denote $Y \setminus \mbox{int}N(\mathcal{L} \cup \mathcal{U})$ as $X$.  Let $\alpha$
be any slope of $\partial N(K)$ so that the manifold obtained by filling $\partial N(K)$ with slope
$\alpha$ is reducible.  Denote a reducing sphere that minimizes the intersection with the attached solid torus 
by $S$.  Then it is easy to see that $S_{E} = S \cap Y \setminus \mbox{int}N(\mathcal{L} \cup \mathcal{U})$
is an essential surface with $\partial S_{E} \subset \partial N(K)$ a non-empty collection of essential curves,
all defining the slope $\alpha$.  By Hatcher~\cite{hatcher} only finitely many slopes of $\partial N(K)$
bounds such a surface.  Twisting $n$ times about $K$ is equivalent to filling $\partial N(K)$ with slope
$\alpha = 1/n$ (for any choice of basis for $H_{1}(\partial N(K);\mathbb Z)$ for which the
boundary of the disk bound by $K$ corresponds to $0/1$).  Thus for all but
finitely many values of $n$ the manifold obtained is irreducible.  We pick one such $n$, and after twisting
$\mathcal{L}$ about $K$ $n$ times, we remove $K$ from $\mathcal{U}$.  Induction completes the proof.
\end{proof}

The following lemma was proved in~\cite{BHW} by Bleiler, Hodgson and Weeks.
It says that the action induced by the mapping class group of $M$ on the
slopes on $\partial M$ (which is assumed to be a single torus) is trivial.
We bring a new argument here.  Our argument holds for
many non-hyperbolic manifolds as well: all we require is that $M$
does not admit infinitely many fillings that result in diffeomorphic manifolds;
this is well known to hold for hyperbolic manifolds as well as all Seifert
fibered spaces except $D^{2} \times S^{1}$.  In Conclusion~(2) below
we use a basis for $H_1(\partial M)$ to identify the slopes of $\partial M$
with $\bar{\mathbb{Q}} = \mathbb{Q} \cup \{1/0\}$.

\begin{lem}
\label{lem:slopes}
Let $M$ be a hyperbolic manifold with $\partial M$ a single torus.  Let $\phi:M \to M$ 
be a orientation preserving diffeomorphism.  
Then for any simple closed curve $\gamma \subset \partial M$,
$\gamma$ and $\phi\circ \gamma$ represent the same slope.

Moreover, if $M$ admits orientation reversing diffeomorphisms then one of the following holds:
\begin{enumerate}
\item For any orientation reversing diffeomorphism $\gamma$ and 
$\phi\circ \gamma$ represent the same slope.
\item  There is a basis for $H_{1}(\partial M)$ so that
for any orientation reversing diffeomorphism $\phi$, 
if a curve $\gamma$ represents the slope $p/q$ then 
$\phi\circ \gamma$ represent the slope $- p/q$.
\end{enumerate}
\end{lem}

\begin{proof}
We use the notation introduced in the statement of the lemma.
By a well known homological argument $\ker i_{*} \cong \mathbb Z$, where here 
$i_{*}:H_{1}(\partial M;\mathbb Z) \to H_{1}(M;\mathbb Z)$ is the homomorphism
induce by the inclusion; moreover, any generator of $\ker i_{*}$ is primitive in
$H_{1}(\partial M;\mathbb Z) $.  Let $l$ be a simple curve so that $[l]$ is a generator of
$\ker i_{*}$ (we denote homology classes with $[ \ ]$).  Then $[\phi \circ l] = \pm [l]$.  

Let $m$ be a simple curve so that $[l]$ and $[m]$ generate $H_{1}(\partial M;\mathbb Z)$.
Since $\phi$ induces an isomorphism on $H_{1}(\partial M;\mathbb Z)$,
$[\phi \circ m]$ and $[\phi \circ l]$  generate $H_{1}(\partial M;\mathbb Z)$.
Thus $[\phi \circ m] = \pm [m] + n [l]$, for some $n \in \mathbb Z$.  
If $n \neq \pm 1$, the orbit of $[m]$ is infinite.  But then the slope
represented by $m$ has an infinite orbit, and hence $M$ admits infinitely many fillings
(namely, fillings along the slopes represented by $\phi^{k}\circ m$ for $k \in \mathbb Z$)
that produce diffeomorphic manifolds.  As $M$ is hyperbolic, this is impossible.
Thus $[\phi\circ m] = \pm m$.
We will use $[m]$ and $[l]$ as a basis for $H_{1}(\partial M)$.  
We conclude that $\phi_*$ is one of the following maps:
\begin{enumerate}
\item $\phi_{*}((p,q)) = (p,q)$.
\item $\phi_{*}((p,q)) = (-p,-q)$.
\item $\phi_{*}((p,q)) = (-p,q)$.
\item $\phi_{*}((p,q)) = (p,-q)$.
\end{enumerate}

Note that cases~(3) and ~(4) imply that $\phi$ is orientation reversing;
hence if $\phi$ is orientation preserving then 
$\phi_{*}:H_{1}(\partial M;\mathbb Z) \to H_{1}(\partial M;\mathbb Z)$
is either the identity or the antipodal map; since a slope is defined as the homology class
of an {\it unoriented} curve, both maps fix all slopes.  

All that remain is to show that if $\phi_{1}$ and $\phi_{2}$ are orientation reversing 
diffeomorphisms of $M$, then $\phi_{1 *}$ and $\phi_{2*}$ are both
as in Cases~(1) or~(2), or both as in Cases~(3) or~(4).   If one is 
as in Cases~(1) or~(2) and the other as in Cases~(3) or~(4), then
$\phi_{2}^{-1} \circ \phi_{1}$ is an orientation preserving diffeomorphism
and $(\phi_{2}^{-1} \circ \phi_{1})_{*}$ is as in Case~(3) or~(4), which is impossible.  
\end{proof}

\subsection{Cores of solid tori}
In this subsection we prove the following necessary condition for a curve in $T^{2} \times [0,1]$
to become a core of the solid torus obtained by filling:

\begin{lem}
\label{lem:CoresOfSolidTori}
Let $K \subset T^{2} \times [0,1]$ and
$c \subset T^{2} \times \{0\}$ be curves, and assume that $c$ is simple and
essential.    
Let $V$ be the solid torus obtained by filling $T^{2} \times \{0\}$ along $c$.  
Suppose that $[c]$ and $[K]$ do not generate $H_{1}(T^{2} \times [0,1];\mathbb Z)$.

Then $K$ is not isotopic to a core of $V$.
\end{lem}

\begin{proof}
We identify  $T^{2} \times \{1\}$ with $T^{2}$.   We will denote the projections 
of $c$ and $K$ to $T^{2}$ as $c'$ and $K'$ respectively.  
We obtain two classes in $H_{1}(T^{2};\mathbb Z)$, defined up-to sign, 
which we will denote as $\pm [c']$ and $\pm [K']$.  

Suppose that $K$ is isotopic to a core of $V$.  
Then the signed intersection of $K$ and the meridian disk of $V$ is $\pm 1$.  
Up-to isotopy, $c'$ is the boundary of the meridian disk of $V$.
Therefore, the signed intersection of $K'$ with $c'$ 
as curves in $T^{2}$ is $\pm 1$ (the sign may not agree with that of the intersection of $K$ and the meridian disk of $V$).
Any class of $H_{1}(T^{2};\mathbb{Z})$ is represented by $n$ parallel copies of a simple closed curve (possibly $n=0$).
Let $\gamma$ be a simple closed curve so that $\pm [n\gamma] = \pm [c']$.  Since $c'$ intersects
$K'$ algebraically once, $n = \pm 1$ and we may assume that $\gamma$ and $K'$ intersect once.
Thus $[\gamma]$ and $[K']$ generate  $H_{1}(T^{2};\mathbb Z)$.  
Since $[\gamma] = [c'] = [c]$ and $[K'] = [K]$, 
$[c]$ and $[K]$ generate $H_{1}(T^{2} \times [0,1];\mathbb Z)$.
\end{proof}

\subsection{Hyperbolic Dehn Surgery}
Let $M$ be a one cusped hyperbolic manifold.  It is well known that for all but finitely many slopes
on $\partial M$, $M(\beta)$ is hyperbolic and the core of the attached solid torus, which we will denote as
$\gamma$, is a geodesic.  As we vary $\beta$, the length of the geodesics $\gamma$ obtained fulfill
$$\lim_{l(\beta) \to \infty} l(\gamma) = 0$$
where here $l(\beta)$ is measured in the Euclidean metric induced on $\partial M$ after some truncation
of the cusp.  Moreover, after an appropriate choice of basepoints, 
as $l(\beta) \to \infty$, the manifolds $M(\beta)$ converge in the sense of
Gromov--Hausdorff to $M$.  With this we obtain the following lemma, which is well known to many experts,
but since we could not find a reference we sketch its proof here.

\begin{lem}
\label{lem:HyperbolicDehnSurgery}
With the notation as above, there exists $\epsilon>0$ and a finite set of slopes of $\partial M$,
which we will denote as $B_{f}$, so that for any slopes $\beta$ of $\partial M$, if $\beta \not \in B_{f}$
then the following three conditions hold:
\begin{enumerate}
\item $M(\beta)$ is hyperbolic.
\item $l(\gamma) < \epsilon$.
\item If $\delta \subset M(\beta)$ is a geodesic and $l(\delta) < \epsilon$, then $\delta = \gamma^{n}$ for some $n$.
\end{enumerate}
\end{lem}

\begin{proof}[Sketch of proof]
Fix $\mu>0$ a Margulis constant that is shorter than half the length of the shortest geodesic in $M$.
The {\it thick part} of $M$, which we will denote $M_{\geq \mu}$, consists of all the points
of $M$ that have radius of injectivity at least $\mu$; note that $M_{\geq \mu}$ is $M$ with its
cusp truncated.  The thick part of $M(\beta_{i})$, which we will denote as 
$M(\beta_{i})_{\geq \mu}$, is defined similarly.

By the discussion above,~(1) and ~(2) are established in the literature.  Assume, for a contradiction, that
there does not exist a finite set $B_{f}$ for which~(3)
holds.  Then there is a sequence $\beta_{i}$ of slopes of $\partial M$ with  $l(\beta_{i}) \to \infty$,
and geodesics $\delta_{i} \subset M(\beta_{i})$ so that $l(\delta_{i}) < 1/i$ and $\delta_{i}$ is
not a power of $\gamma_{i}$, the core of the attached solid torus.  
Let $V_{i}$ be the component of $M(\beta_{i}) \setminus M(\beta_{i})_{\geq \mu}$ that contains
$\gamma_{i}$.  We will denote $N_{\frac{1}{2i}}(M(\beta_{i}) \setminus V_{i})$ as $N_{i}$, where
here $N_{\frac{1}{2i}}$ denotes the $\frac{1}{2i}$ neighborhood.  By construction, $\delta_{i} \subset N_{i}$.

For an appropriate choice of basepoints  $x_{i} \in M(\beta_{i})$ and $x \in M$, $(M(\beta_{i}),x_{i})$
converges to $(M,x)$ is the Gromov--Hausdorff sense.
Then $f_{i} \circ \delta_{i}$
are closed curves in $M_{\geq \mu}$ with $l(f_{i} \circ \delta_{i}) \to 0$.  Thus, for sufficiently large $i$,
$l(f_{i} \circ \delta_{i}) < \mu$, and hence $f_{i} \circ \delta_{i}$ is null homotopic
or is homotopic into the cusp.  In the former case, let $D \subset M_{\geq \mu}$ be an immersed 
disk whose boundary is  $f_{i} \circ \delta_{i}$.   By isotopy we may move $D$ out of the cusp.
The image of $D$ under $f_{i}^{-1}$
shows that $\delta_{i}$ bounds an immersed disk, which is a contradiction.   In the latter case,
we may use an immersed annulus given by the trace of the homotopy of $f_{i} \circ \delta_{i}$
to the cusp to conclude that $\delta_{i}$ is homotopic into a neighborhood of the core
of the attached solid torus; this implies that the geodesic $\delta_{i}$ is itself a power of that core.
\end{proof}

\section{Bounded sets in the Farey graph}
\label{sec:BoundedSets}

\bigskip

\begin{figure}[h!]
\includegraphics[height=3.5in]{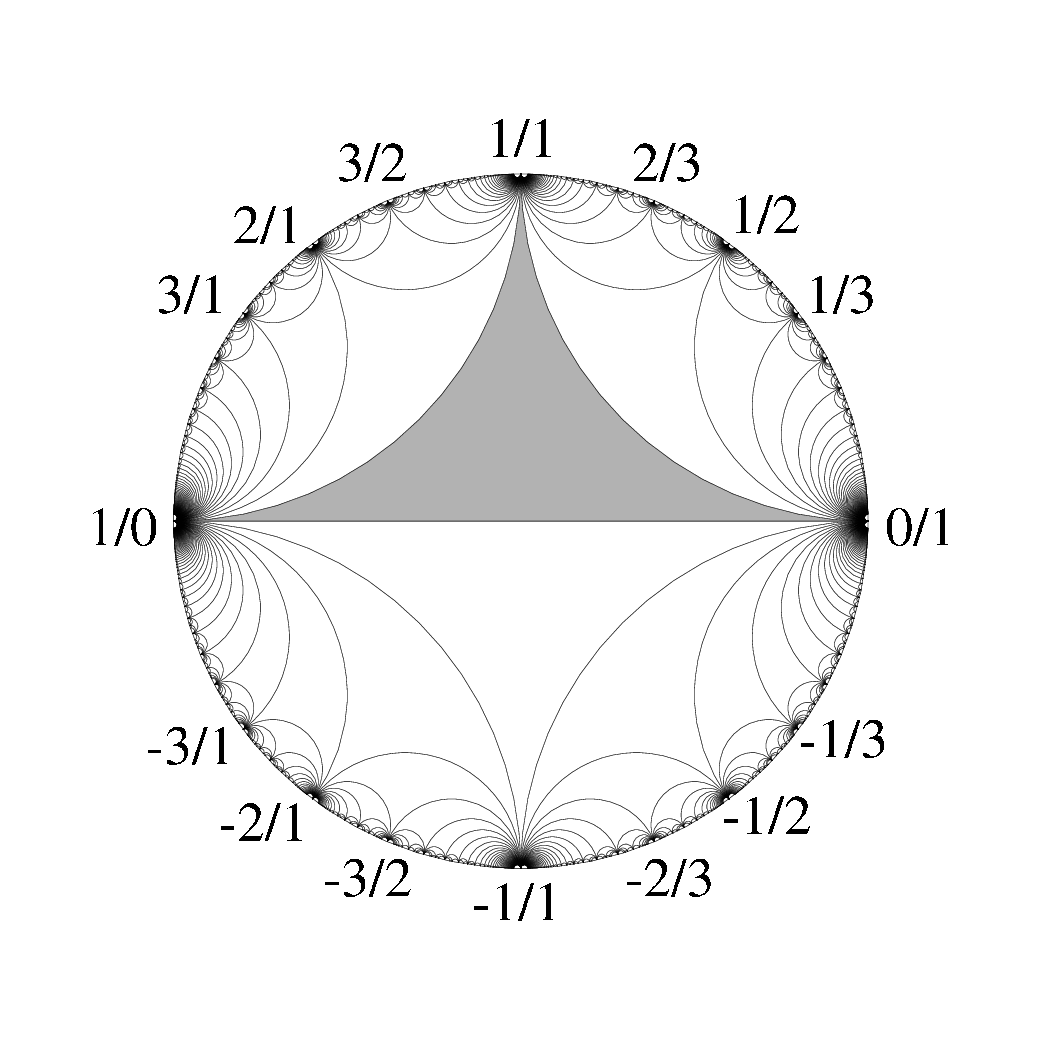}
\caption{The Farey tessellation}
\label{fig:Farey}
\end{figure}

\noindent
The Farey graph is a connected graph whose vertices encodes the slopes of the torus;
we begin this section with a brief description of this well known construction,
see Figure~\ref{fig:Farey}.
By viewing each edge as a length one interval the Farey graph induces a metric on the slopes 
which is instrumental for our study:  throughout this paper we make extensive use of sets of 
slopes that are bounded in this metric.    In~\cite{rieckyamashita} we noted that 
that any cover between tori induces a bijection on their slopes and argued that
branched covers between manifolds can be completed after filling if and only if the 
slopes filled correspond under this bijection; in Subsection~\ref{subsection:FareyAndCovering}
we recall these facts and prove that this bijection is a bilipschitz map, and hence the image
of a bounded set is bounded (in either direction).  In Subsection~\ref{subsection:boundedsets}
we prove that bounded sets are closed under certain operations, notably Dehn twists and 
adding slopes of bounded intersection (see Proposition~\ref{prop:PropertiesOfBoundedSets} 
for a precise statement).  Moreover, we show that bounded sets form the smallest non empty 
collection of sets that is closed under these operations;
hence, from our point of view, they are the smallest collection of sets we may use.

\bigskip
\noindent
We now describe the Farey graph; it is best seen as embedded in 
$\mathbb H^{2} \cup S^{1}_{\infty}$ as the 1-skeleton of the Farey tessellation of
the hyperbolic plane, although
the metric we will use (as described above) is {\it not} induced by the hyperbolic metric.  
For this construction see Figure~\ref{fig:Farey}.  Pick an ideal triangle in $\mathbb H^{2}$
and label its vertices as $0/1$, $1/0$, and $1/1$.  The Farey tessellation is constructed
recursively: after reflecting a triangle by an edge whose endpoints are labeled 
$p/q$ and $r/s$, we obtain a new triangle and
label the new vertex by $(p \pm r)/(q \pm s)$; the sign
depends on the direction of the reflection.  This addition rule
is induced by the addition in $\mathbb Z \times \mathbb Z$.  
At the end of the day we obtain a tessellation of $\mathbb H^{2}$ by ideal triangles and 
it is a well known consequence of Euclid's Algorithm that every element of 
$\bar{\mathbb{Q}} = \mathbb{Q} \cup \{1/0\}$ appears as
the label for exactly one vertex.  The {\it Farey graph} is the graph
given by the 1-skeleton of the Farey tessellation.
After choosing a basis for $H_{1}(T^{2};\mathbb Z)$ we 
identify the slopes of $T^{2}$ with $\bar{\mathbb{Q}}$.
Thus we have a bijection between the slopes of $T^{2}$ and the vertices of the
Farey tessellation.  It is easy to see that the claims below
do not depend on the choice of  basis, as a change of basis induces an isomorphism
of the Farey graphs.

Let $x$ and $y$ be vertices that correspond to slopes $\alpha_{1}$
and $\alpha_{2}$; it is well known that $x$ and $y$ are connected by an edge in the Farey
tessellation if and only if $\alpha_{1}$ and $\alpha_{2}$ can be represented by
curves that intersect exactly once.
The {\it distance}
in the Farey tessellation is the minimal number of edges traversed to get
from one vertex to another.  
We define the distance between slopes $\alpha_{1}$ and $\alpha_{2}$ to be the distance
between the corresponding vertices of the Farey tessellation and denote it as
$$\dfar(\alpha_{1},\alpha_{2}).$$
Using this mertic we can define:

\begin{dfn}
\label{dfn:bounded}
A set of slopes is called {\it bounded} if it is has a bounded diameter using the 
metric induced by $\dfar$.  
\end{dfn}

It is well known that the Farey graph has infinite diameter, hence:

\begin{lem}
If $B$ is a bounded set of slopes and then there are infinitely many slopes not in $B$.
\end{lem}

\subsection{Coverings and slopes} 
\label{subsection:FareyAndCovering}
For this subsection fix tori $T$, $T'$ and $\phi:T \to T'$ a cover. 
In~\cite{rieckyamashita} we showed that $\phi$ induces a bijection between the
slopes of $T$ and those of $T'$; we refer the reader to that paper for a detailed discussion. 
The correspondence is defined as follows: let $\alpha$ be a slope of $T$ and $\gamma$ an essential connected
simple closed curve on $T$ representing $\alpha$.  Then $\phi \circ \gamma$ is an essential connected
(not necessarily simple) closed curve on $T'$, and hence there exist a positive integer $m$ and
$\beta'$, a connected simple closed curve on $T'$,
so that $(\beta')^{m}$ is freely homotpic to $\phi \circ \gamma$.  We define the 
slope represented by $\beta'$ to be the slope that corresponds to $\alpha$.  Conversely, let $\alpha'$
be a slope of $T'$ represented by an essential simple closed curve $\gamma'$.  Then $\phi^{-1}(\gamma')$
is an essential (not necessarily connected) simple closed curve.  We define  
the slope represented by a component of $\phi^{-1}(\gamma')$ to be the slope that corresponds to $\alpha'$.  
It is not hard to see that these correspondences are inverses of each other.  

We refer to slopes that correspond under this bijection as {\it corresponding slopes}.
The branched covers between manifolds
that we consider in this paper induce covers on boundary components.
In~\cite{rieckyamashita} we showed that a branched cover $X \to E$ (where $X$ and $E$ are manifolds
with toral boundary and the branch set is closed, that is, disjoint from $\partial E$)
extends to a branched cover after filling if and only if the slopes filled are corresponding slopes. 
{\it This simple fact will often be used without reference throughout this paper.}

\begin{rmk}
In the example below it will be convenient to take an alternate view of the correspondence between the slopes of
$T$ and those of $T'$, where $\phi: T \to T$ is a cover.
By lifting $\phi:T \to T'$ appropriately to the universal covers
we obtain an matrix in $SL(2,\mathbb Q)$. Any such matrix gives a correspondence between  
lines of rational slope in the universal covers, which are naturally identified with slopes on the tori; 
the reader can verify that this correspondence is equivalent to the correspomndence described above.
\end{rmk}

Let $T$ and $T'$ be tori and $\phi:T \to T'$ a cover as above.
We will often consider bounded sets of $T$ or $T'$; our goal in this subsection 
is to understand their behavior under the correspondence induced by $\phi$.
We begin with a simple example: consider $T$ and $T'$ as
$\mathbb{R}^{2}/\sim$, where $\sim$ is given by $(x,y)\sim (x+n,y+m)$ for $n,m \in \mathbb Z$. 
Let $\phi:T \to T'$ be the double cover induced by the map $\tilde\phi:\mathbb{R}^{2} \to \mathbb{R}^{2}$
given by $\tilde\phi((x,y)) = (x,2y)$.  The four segments in $\mathbb{R}^{2}$ connecting
$(0,0)$ to $(1,0)$, $(2,1)$, $(1,1)$, and $(0,1)$ map to four curves on $T$, defining slopes
that can be naturally denoted as $0/1$, $1/2$, $1/1$, and $1/0$, see Figure~\ref{fig:Farey2}.  
\begin{figure}
\includegraphics[height=2.5in]{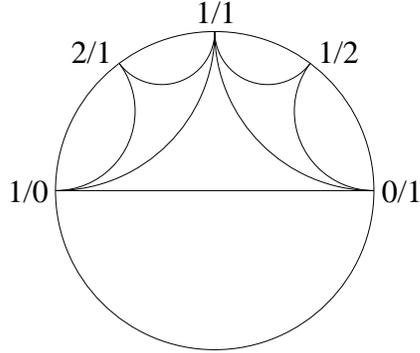}
\caption{The two quadrilaterals in the Farey tessellation}
\label{fig:Farey2}
\end{figure}
The images these segments
under $\tilde\phi$ are the segments connecting $(0,0)$ to $(1,0)$, $(2,2)$, $(1,2)$, and $(0,2)$,  
and the corresponding slopes are $0/1$, $1/1$, $2/1$, and $1/0$.  
Thus the correspondence maps a quadrilateral in the Farey graph of $T$ 
to a quadrilateral in the Farey graph of $T'$; each quadrilateral has exactly one diagonal edge.
In the first, the diagonal edge connects $0/1$ to $1/1$, and in the second it 
connects $1/0$ to $1/1$.
These edges do {\it not} correspond, showing that the correspondence
between the slopes of $T$ and those of $T'$ (viewed as a bijection between the vertices of
the Farey graphs) does {\it not} induce an isomorphism of the Farey
graphs.  Equivalently, the bijection between the slopes of $T$ and $T'$ (considered as metric
spaces with the metric $\dfar$) is {\it not} an isometry.  It is easy it see that in this example
some distances decrease while others increase.
Since we always consider Farey graphs as metric spaces with the metric $\dfar$ we need
the following lemma:

\begin{lem}
\label{lem:CorrespondingSlopesFeray}
Let $T$, $T'$ be tori and $\phi:T \to T'$ be a cover.
Then the correspondence between slopes of $T$ and $T'$ is bilipschitz 
and hence it sends bounded sets of $T$ to
bounded sets of $T'$ and vice versa.
\end{lem}

\begin{proof}
For convenience we endow $T'$ with a Euclidean metric and $T$ with the pullback metric.
Since we are only interested in the homology classes represented by curves, throughout the proof of 
this lemma we assume as we may that the essential curves of $T$ and $T'$ considered are geodesic.

We first show that the map from the slopes of $T$ to those of $T'$ is Lipschitz.  
Fix a positive integer $d$ and two slopes of $T$ of distance exactly $d$, which we will
denote as $\alpha_{0}$ and $\alpha_{d}$.  By definition of $\dfar$,
there exist slopes $\alpha_{1},\dots,\alpha_{d-1}$ so that
$\dfar(\alpha_{i-1},\alpha_{i}) = 1$ (for $i=1,\dots,d$).  
Let $\gamma_{0},\dots,\gamma_{d}$ be geodesics on $T$ representing 
$\alpha_{0},\dots,\alpha_{d}$ (respectively).  Since geodesics on a torus minimize intersection, 
$\dfar(\alpha_{i-1},\alpha_{i}) = 1$ is equivalent to 
$|\gamma_{i-1}\cap \gamma_{i}|=1$.  
Let $\gamma_{i}'$ be   a connected simple closed geodesic on $T'$
and $m_{i}$ a positive integer so that $\phi(\gamma_{i})$ is obtained by 
traversing $\gamma_{i}'$ exactly $m_{i}$ times.  
By definition, $\gamma_{i}'$ is a geodesic representing $\alpha_{i}'$,
where here $\alpha_{i}'$ is the slope of $T'$ that corresponds to $\alpha_{i}$.
We will denote $|\gamma'_{i-1} \cap \gamma'_{i}|$ as $c_{i}$.
Let $\Gamma_{i}$ denote $\phi^{-1}(\phi(\gamma'_{i}))$ and denote the number of components of 
$\Gamma_{i}$ as $n_{i}$.  
Each component of $\Gamma_{i}$ is a geodesic parallel to $\gamma_{i}$ and is an $m_{i}$ fold
cover of $\beta_{i}$; it follows that $n_{i} m_{i} = \deg(\phi)$, and in particular 
$$n_{i} \leq \deg(\phi).$$  
Since $\Gamma_{i-1} \cap \Gamma_{i} = \phi^{-1}(\gamma_{i-1} \cap \gamma_{i})$, 
$|\Gamma_{i-1} \cap \Gamma_{i}| = c_{i} \deg(\phi)$.
Since $|\gamma_{i-1} \cap \gamma_{i}|=1$, every component of $\Gamma_{i-1}$ intersects
every component of $\Gamma_{i}$ exactly once; if follows that 
$|\Gamma_{i-1} \cap \Gamma_{i}| = n_{i-1} n_{i}$, the number of pairs of
curves from $\Gamma_{i-1}$ and $\Gamma_{i}$.  Combining these facts we see
that
$$c_{i} \deg(\phi) = |\Gamma_{i-1} \cap \Gamma_{i}| = n_{i-1}n_{i} \leq \deg(\phi)^{2},$$
showing that $c_{i} \leq \deg(\phi)$.

Thus $\Delta(\gamma_{i-1}',\gamma_{i}') \leq |\gamma_{i-1}' \cap \gamma_{i}'| \leq \deg(\phi)$.  
It then follows from Euclid's algorithm
that $\dfar(\alpha_{i-1}',\alpha_{i}') \leq 2 \log_{2}{(\deg(\phi))}$.
(For the relation between Euclid's Algorithm and the Farey tessellation see, for example,
\cite{hardywright}.)
We get:
\begin{eqnarray*}
\dfar(\alpha'_{0},\alpha'_{d}) &\leq& \sum_{i=1}^{d} \dfar(\alpha'_{i-1},\alpha'_{i}) \\
	&\leq& 2 \log_{2}(\deg(\phi)) d \\
	&=&  2 \log_{2}(\deg(\phi)) \dfar(\alpha_{0},\alpha_{d}).
\end{eqnarray*}
%
As $d$, $\alpha_{0},$
and $\alpha_{d}$ were arbitrary we conclude that:
$$\dfar(\alpha_{0}',\alpha_{d}') \leq 2 \log_{2}{(\deg(\phi))} \dfar(\alpha_{0},\alpha_{d}),$$
showing that the map from the slopes of $T$ to the slopes of $T'$ is Lipschitz with
constant $2\log_{2}(\deg(\phi))$.

\medskip

The converse is essentially identical and we only sketch it here.  
Similar to the argument above, fix $d$
and let $\alpha_{0}',\alpha_{d}'$ be slopes of $T'$ so that $\dfar(\alpha'_{0},\alpha'_{d})=d$.
Let $\alpha_{0}',\dots,\alpha_{d}'$
be a sequence of slopes that realizes the distance, that is, $\dfar(\alpha_{i-1}',\alpha_{i}') = 1$.
Let $\gamma_{i}'$ be a geodesic on $T'$ that represents $\alpha_{i}'$; 
then $|\gamma_{i-1}' \cap \gamma_{i}'|=1$.
Let $\gamma_{i}$ be a component of $\phi^{-1}(\gamma_{i})$.  By definition, $\gamma_{i}$
represents $\alpha_{i}$, the slope of $T$ that corresponds to $\alpha_{i}'$.
Since $\gamma_{i-1} \cap \gamma_{i} \subset \phi^{-1}(\gamma'_{i-1} \cap \gamma_{i}')$ we have:
$$|\gamma_{i-1} \cap \gamma_{i}| \leq |\phi^{-1}(\gamma'_{i-1} \cap \gamma_{i}')|
= \deg(\phi) |\gamma'_{i-1} \cap \gamma_{i}'| = \deg(\phi).$$
As above this implies that $\dfar(\alpha_{i-1},\alpha_{i}) \leq 2 \log_{2}{(\deg(\phi))}$,
and hence
$$\dfar(\alpha_{0},\alpha_{d}) \leq 2 \log_{2}{(\deg(\phi))} \dfar(\alpha_{0}',\alpha'_{d}),$$
showing that the map from the slopes of $T'$ to the slopes of $T$ is Lipschitz with
constant $2\log_{2}(\deg(\phi))$.
\end{proof}

\bigskip

\noindent

\subsection{On bounded sets}
\label{subsection:boundedsets}
In this subsection we study properties of bounded sets, that is, sets of slopes that are bounded
in the metric $\dfar$ (as defined in~\ref{dfn:bounded}).
Our main goal is to show (Proposition~\ref{prop:PropertiesOfBoundedSets}) that bounded sets are closed
under certain operations, the most important of which is Dehn twists, which we explain below.  
We then note (Lemma~\ref{lem:BoundedIsNecessary}) 
that any non-empty collection of sets of slopes 
that is closed under the operations discussed in Proposition~\ref{prop:PropertiesOfBoundedSets} 
contains all bounded sets; therefore boundedness 
is the {\it weakest} condition that suffices for our work.

Let $\alpha$ be a slope on $T$ represented by a connected simple closed curve $\gamma$.
By definition $\alpha$ determines $\gamma$ up-to isotopy.  Hence the Dehn twist about $\gamma$
is completely determined by $\alpha$.   The Dehn twist about $\gamma$ induces
a map on the slopes of $T$, which we will denote as $D_{\alpha}$.   Although the proper way
to refer to $D_{\alpha}$ is ``the map induced on the slopes of $T$ by Dehn twisting
about a simple connected curve representing $\alpha$'' we will, for simplicity's sake,
refer to it as {\it the Dehn twist about $\alpha$}.

Since any Dehn twist is a homeomorphism of $T$, it sends
any pair of curves that intersect once to a pair of curves that intersect once;
hence, for any slope $\alpha$, $D_{\alpha}$ induces a graph isomorphism 
of the Farey graph, and in particular
$D_{\alpha}$ induces an isometry on the slopes
of $T$ with the metric given by $\dfar$.
We are now ready to state:

\begin{prop}
\label{prop:PropertiesOfBoundedSets}
Let $T^{2}$ be a torus.  The following conditions hold for sets of slopes on $T^{2}$: 
	\begin{enumerate}
	\item A subset of a bounded set is bounded.
	\item Finite unions of bounded sets are bounded.
	\item For any slope $\alpha$ of $T^{2}$ and any bounded set of slopes $B$, the set
			$$\cup_{n \in \mathbb Z}\{D_{\alpha}^{n}(\beta) \ | \ \beta \in B \}$$ 
	is bounded (where $D_{\alpha}$ is as above).
	\item For any integer $c \geq 0$ and any bounded set $B$, the set of slopes
		$$\{\alpha \ | \ \exists \beta \in B, \Delta(\alpha,\beta) \leq c \}$$
	is bounded (where here we use $\Delta(\alpha,\beta)$ to represent the geometric intersection
	between simple closed curves that represent $\alpha$ and $\beta$).
	\end{enumerate}
\end{prop}

\begin{proof}
\begin{enumerate}
\item Obvious from the definition of a metric space.
\item Obvious from the definition of a metric space.
\item  Since $D_{\alpha}$ is induced by a homeomorphism
of $T^{2}$, it is clearly an isometry of the Farey graph; 
moreover, $\alpha$ is fixed under $D_{\alpha}$.  
Let $d>0$ be the diameter of $B$; fix $\beta \in B$.  
By the triangle inequality, for any $n_{1},n_{2}\in\mathbb Z$ 
and any $\beta_{1},\beta_{2} \in B$ we have: 
\begin{eqnarray*}
\dfar(D_{\alpha}^{n_{1}}(\beta_{1}),D_{\alpha}^{n_{2}}(\beta_{2})) 
&\leq& \dfar(D_{\alpha}^{n_{1}}(\beta_{1}),D_{\alpha}^{n_{1}}(\beta)) \\
 && +\dfar(D_{\alpha}^{n_{1}}(\beta),\alpha) \\
 & & + \dfar(\alpha,D_{\alpha}^{n_{2}}(\beta)) \\
 && + \dfar(D_{\alpha}^{n_{2}}(\beta),D_{\alpha}^{n_{2}}(\beta_{2})).
\end{eqnarray*}
Since $\alpha$ is fixed under the isometry $D^{n_{1}}_{\alpha}$, for the second term
we have 
$$\dfar(D_{\alpha}^{n_{1}}(\beta),\alpha) = 
\dfar(D_{\alpha}^{n_{1}}(\beta),D^{n_{1}}_{\alpha}(\alpha)) =
\dfar(\beta,\alpha).$$
Similarly, the third term is bounded above by $\dfar(\beta,\alpha)$.
For the first term we have
$$ \dfar(D_{\alpha}^{n_{1}}(\beta_{1}),D_{\alpha}^{n_{1}}(\beta)) \leq
 \dfar(\beta_{1},\beta) \leq d.$$ 
Similarly, the fourth term is bounded by $d$.
Combining these bounds we see that set 
$\cup_{n \in \mathbb Z}\{D_{\alpha}^{n}(\beta) \ | \ \beta \in B \}$
is bounded with diameter at most 
$$2d+2\dfar(\alpha,\beta).$$

\item  We first prove the claim when $B$ has only one element which we will 
denote as $\beta$.
By an appropriate choice of basis for $H_{1}(T^{2};\mathbb Z)$ we may identify the slopes of $T$
with $\mathbb Q \cup \{1/0\}$ so that $\beta$ corresponds to $1/0$.  It is well know that for any slope
$p/q$, $\dfar(1/0,p/q)$ is exactly the length of the shortest continued fraction expansion of $p/q$
(see, for example, \cite{series}).
On the other hand, $\Delta(1/0,p/q) = |\det  (\begin{smallmatrix} 1&p\\ 0&q \end{smallmatrix})| = |q|$.  
Thus the slopes under consideration correspond to $p/q$ with $|q| \leq c$.  
By a well known application of Euclid's Algorithm, such a number has a continued
fraction expansion of length at most $2 \log_{2}(|q|)$.  Thus every slope in $B$
has distance at most $2 \log_{2}(|q|)$ from $\beta$, showing that $B$ is bounded with diameter
at most $4 \log_{2}(c)$.

For the general case, let $B$ be a bounded set of slopes and $c \geq 0$ an integer.  
We will denote the diameter 
of $B$ as $d$.  Let $\alpha_{1}$ and $\alpha_{2}$ be slopes in
$\{\alpha \ | \ \exists \beta \in B, \Delta(\alpha,\beta) \leq c \}$.  Let $\beta_{1}$
and $\beta_{2}$ be slopes of $B$ so that $\Delta(\alpha_{1},\beta_{1}) \leq c$ and
$\Delta(\alpha_{2},\beta_{2}) \leq c$.  By the argument above,
$\dfar(\alpha_{1},\beta_{1}), \dfar(\alpha_{2},\beta_{2}) \leq 4 \log_{2}(c)$.  
Since $\beta_{1},\beta_{2} \in B$,
$\dfar(\beta_{1},\beta_{2}) \leq d$.  By the triangle inequality 
$\{\alpha \ | \ \exists \beta \in B, \Delta(\alpha,\beta) \leq c \}$
in bounded with diameter at most $4 \log_{2}(c) + d$.

\end{enumerate}
\end{proof}

\bigskip\noindent
The reader may wonder if ``bounded sets'' is the right concept to use.   First we give an example that will explain
why using unbounded sets may be tricky.  Using the upper half plane model of $\mathbb H^{2}$, we construct the
Farey tessellation starting with the triangle 0, 1, and $1/0$ (the point at infinity).   
Then the slopes are naturally identified with $\mathbb Q \cup 1/0$; this identification 
is used throughout this example.  
Let $q_{i}'$ be an enumeration of $\mathbb{Q} \cap [0,1)$, $i=1,2,\dots$.  Let $D_{1/0}$ denote
the map on the slopes induced by Dehn twist about $1/0$; it is not hard to see that for any slope $\alpha \neq 1/0$, 
$D_{1/0}(\alpha) = \alpha+1$.  Let $q_{i} = D^{i}(q_{i}')$.  Finally, let $S = \{q_{i}\}$.  
Now $S$ is a fairly ``thin'' set of slopes; it has only one member in every interval $[n,n+1), n=1,2,\dots$, and so 
its only accumulation point is $1/0$.  However, if we allow twisting about $1/0$, we see the following: for any slope 
$\alpha \neq 1/0$,
there is a unique $j \in \mathbb Z$, so that $\alpha + j \in [0,1)$.  Hence, 
there is a unique $i$ so that $\alpha+j = q_{i}'$.  Thus
$\alpha = D_{1/0}^{-j}(q_{i}') = D_{1/0}^{-(i+j)}(q_{i})$.  In other words, 
$\{D_{1/0}^{n}(q_{i}) | q_{i} \in S, n \in \mathbb Z\}$ is the set of {\it all} slopes but $1/0$.
As we shall see below repeatedly, when considering cosmetic surgery 
one must allow for Dehn twists; thus, the analogue
of Proposition~\ref{prop:PropertiesOfBoundedSets} for the set $S$ is 
very much false, and such a set
cannot be used in our work.  

On the other hand we have the following lemma, that tells us that all 
bounded sets {\it must} be permitted; in that sense,
our results cannot be improved:

\begin{lem}
\label{lem:BoundedIsNecessary}
Let $S$ denote the set of all slopes and suppose $Z \subset \mathcal{P}(S)$ fulfills 
the following condition:
	\begin{enumerate}
	\item For some slopes $\alpha$, $\{\alpha\} \in Z$.
	\item $Z$ is closed under conditions~(1)--(4) of Proposition~\ref{prop:PropertiesOfBoundedSets}.
	\end{enumerate}
Then $Z$ contains all the bounded sets.
\end{lem}

\begin{proof}
By assumption there exists a slope $\alpha$ so that $\{\alpha\} \in Z$.  Let $\alpha'$ be any other slope.
By condition~(4), $\{\alpha'' | \Delta(\alpha,\alpha'') \leq \Delta(\alpha,\alpha')\} \in Z$. 
By~(1), $\{\alpha'\} \in Z$.  Hence $Z$ contains all the singletons.  

For a slope $\alpha_{0}$ and a positive integer $d$, we will denote the set
of all slopes of distance at most $d$ from $\alpha_{0}$ as 
$D_{d}(\alpha_{0})$.   We claim that for any $\alpha_{0}$ and any $d$,
$D_{d}(\alpha_{0}) \in Z$.  The proof is inductive.  For $d=0$,
$D_{0}(\alpha_{0}) = \{\alpha_{0}\}$ and hence $D_{0}(\alpha_{0}) \in Z$. 
Assume that $d>0$ and that $D_{d-1}(\alpha_{0}) \in Z$. Then
for any $\alpha_{d} \in D_{d}(\alpha_{0}) \setminus D_{d-1}(\alpha_{0})$, there is
$\alpha_{d-1}(\alpha_{0}) \in D_{d-1}(\alpha_{0})$, so that $\dfar(\alpha_{d},\alpha_{d-1}) = 1$.
Hence $\Delta(\alpha_{d},\alpha_{d-1}) = 1$.  Thus
$$D_{d}(\alpha_{0}) \subset \{\alpha | (\exists \beta \in D_{d-1}) \Delta(\alpha,\beta) \leq 1 \}.$$
By (4), the latter is in $Z$.  By~(1), $D_{d} \in Z$.

Let $B$ be a non empty bounded set.  We will denote the diameter of $B$ by $d$.
Then for any slope $\beta \in B$, 
$$B \subset D_{d}(\beta).$$
Thus by~(1), $B \in Z$.
\end{proof}

We conclude this section by revisiting Lemma~\ref{lem:slopes}:

\begin{lem}
\label{lem:MCGisometryOfSlopes}
Let $M$ be a hyperbolic manifold, $\partial M$ a single torus.
There exists a (possibly trivial) involution $i$ on the set of slopes of $\partial M$,
that induces an isometry on the Farey graph,
so that if $\phi:M \to M$ is a diffeomorphism and $\alpha$ a slopes of $\partial M$,
the image of $\alpha$ under $\phi$ is either $\alpha$ or $i(\alpha)$.
\end{lem}

\begin{proof}
This follows immediately from Lemma~\ref{lem:slopes} and the fact that
$(p,q) \mapsto (-p,q)$ induces an isometry on the Farey graph (this isometry is the
reflection by the edge connecting $1/0$ to $0/1$, and the endpoints of this edge are its  
fixed slopes).
\end{proof}

\section{Minimally non hyperbolic fillings}
\label{sec:mnh}

\noindent
From this section on, the notation introduced in Subsection~\ref{subsec:notaion} will be used regularly;
we assume the reader is comfortable with it.
In this section we study {\it \mnh\ fillings}, a concept which is designed for studying exceptional 
surgeries on links with many boundary components, or equivalently, exceptional fillings
on manifold with many boundary components.
There are two reasons for studying \mnh\ fillings: first, every non hyperbolic filling admits a \mnh\ partial filling; 
second, in Proposition~\ref{pro:mnhIsFinite} we
show that any manifold admits only {\it finitely many} \mnh\ fillings.  
This is essential for finiteness of $T(X)$, the tree that we will construct in the next section.

\bigskip\bigskip\noindent
Let us begin with a simple example.  Let $X$ be the exterior of the Whitehead link endowed with the 
natural meridian and longitude on each boundary component.  
For $j=1,2\dots$, let $(\alpha_{1}^{j},\alpha_{2}^{j}) = (1/0,p^{j}/q^{j})$
for some $p_{j},q_{j}$.  Then for 
each $j$, $X(\alpha_{1}^{j},\alpha_{2}^{j})$ is a lens space and hence $(\alpha_{1}^{j},\alpha_{2}^{j})$ 
is a non hyperbolic multislope.  There is no mystery here: all the multislopes $(\alpha_{1}^{j},\alpha_{2}^{j})$ 
have a common partial filling, namely $(1/0,\nos)$,
and $X(1/0, \nos) \cong D^{2} \times S^{1}$ is non-hyperbolic.  
In this situation, it is better to consider the single multislope $(1/0,\infty)$ and not the 
infinite set of multislopes $\{(\alpha_{1}^{j},\alpha_{2}^{j}) \}_{j=1}^{\infty}$.  
This leads us to the following definition
which is central to our work:

\begin{dfn}[\mnh\ filling]
Let $X$ be a hyperbolic manifold, $\mathcal{T} =T_1 \cup \cdots \cup T_{n}$ components of $\partial X$,
and $\alpha$ a multislope of $\mathcal{T}$.  
We say that $\alpha$ is {\it \mnh} if 
$X(\alpha)$ is non-hyperbolic  and any strict partial 
filling $\alpha' \pf \alpha$ is hyperbolic.
\end{dfn}

So in the example of the Whitehead link exterior discussed above, the slopes $(1/0,p^{j}/q^{j})$ are not \mnh\
and  $(1/0, \nos )$ is.  

\bigskip\noindent
We prove:

\begin{prop}
\label{pro:mnhIsFinite}
Let $X$ be a hyperbolic manifold and $\mathcal{T} = T_{1} \cup \cdots \cup T_{n}$ 
components of $\partial X$.  Then there are only finitely
many \mnh\ fillings on $\mathcal{T}$.
\end{prop}

\begin{proof}
We assume as we may that there are infinitely many multislopes on 
$\mathcal{T}$ yielding non hyperbolic manifolds.
Let $\alpha^{j} = (\alpha_{1}^{j},\dots,\alpha_{n}^{j})$ be an infinite set of distinct 
non hyperbolic fillings  ($j=1,2,\dots$).  We will prove the theorem by showing that 
one of these is not \mnh.  After subsequencing $n$ times we may assume that 
for each $1 \leq i \leq n$, $\alpha_{i}^{j}$ (the restrictions of
$\alpha^{j}$ to $T_{i}$) fulfill one of the following conditions:
	\begin{enumerate}
	\item For $j \neq j'$, $\alpha^{j}_{i} \neq \alpha^{j'}_{i}$; we assume in addition 
	that for all $j$, $\alpha_{{i}}^{j} \neq \nos$.
	\item For any $j,j'$, $\alpha^{j}_{{i}} = \alpha^{j'}_{i}$ (possibly $\alpha^{j}_{i} = \nos$).
	\end{enumerate}
To avoid overly complicated notation we do not rename $\alpha^{j} = (\alpha_{1}^{j},\dots,\alpha_{n}^{j})$.
After reordering the boundary components if necessary, we may assume that 
$\alpha_{i}^{j}$ are distinct for $1 \leq i \leq k$, and $\alpha_{i}^{j}$ is constant for $k+1 \leq i \leq n$
(for some $0 \leq k \leq n+1$).  Note that $k=0$ is impossible since $\{\alpha^{j}\}$ 
is infinite.  We drop the superscript from $\alpha_{i}$ for $i>k$.

Let $\widehat X = X(\nos ,\dots,\nos,\alpha_{k+1},\dots,\alpha_{n})$.  Then the manifolds 
$\widehat X(\alpha_{1}^{j},\dots,\alpha_{k}^{j}) = X(\alpha_{1}^{j},\dots,\alpha_{n}^{j})$
are all non hyperbolic by assumption.  We claim that $\widehat X$ is not hyperbolic.
Assume for a contradiction it is.  Since $\alpha_{i}^{j} \to \nos$ for all $i$,
by Thurston's Dehn Surgery Theorem, for large enough $j$, 
$\widehat X(\alpha_1^{j},\dots,\alpha_{k}^{j})$ is hyperbolic as well.
This contradicting our assumptions and shows that
$\widehat X$ is not hyperbolic.  Recall that $\alpha_i^{j} \neq \nos$ for $1 \leq i \leq k$ and $k\geq 1$;
hence $(\nos,\dots,\nos,\alpha^{j}_{k+1},\dots,\alpha^{j}_{n})$
is a strict partial filling of $(\alpha^{j}_{1},\dots,\alpha^{j}_{n})$, showing that the latter is not \mnh.
The proposition follows.
\end{proof}

\section{$T(X)$}
\label{sec:T(X)}
\bigskip

\noindent
In this section we construct the tree $T(X)$, which is the main tool for our work on cosmetic
surgery.  The construction relies heavily on the concept of {\it \mnh\ fillings} from the previous section.
After constructing $T(X)$ we prove (Proposition~\ref{prop:T(X)isFinite}) that it is
finite.  We then explain (Proposition~\ref{prop:ObtainedByFilling}) 
why $T(X)$ can be used to study fillings.  Since $T(X)$ was designed
for studying exceptional surgeries on hyperbolic manifolds, it is perhaps a little
surprising that it can also be used to study hyperbolic fillings; this
is explained and proved in Proposition~\ref{prop:UsingT(X)forHyperbolicFilling}.

\bigskip\bigskip\noindent
Let $X$ be a compact orientable manifold whose boundary consists of tori
and $\mathcal{T} = T_{1} \cup \dots \cup T_{n}$ 
components of $\partial X$.
We wish to associate to $(X,\mathcal{T})$ a finite rooted tree, 
denoted $T(X,\mathcal{T})$ (or $T(X)$ when no confusion can arise), that 
encodes exceptional fillings on $X$.   

Before constructing $T(X)$ we comment about its structure. The vertices of $T(X)$
correspond to manifolds with $X$ as the root.  We direct every edge away from the root. 
Branches, which are always assumed to follow this direction, encode the filling process, and
therefore we may have distinct vertices that correspond to diffeomorphic manifolds.
It follows from the construction below that the vertices are arranged along levels.  The levels are grouped
into block of the form $3m$, $3m+1$, and $3m+2$ and obey the following rules:
	\begin{enumerate}
	\item Geometric manifolds (that is, hyperbolic manifolds, Seifert manifolds, and sol manifolds)
	are arranged on levels $3m$ (for $m\in \mathbb Z_{\geq 0}$).
	\item Reducible manifolds are arranged on levels $3m+1$ (for $m\in \mathbb Z_{\geq 0}$).
	\item JSJ manifolds (recall Definition~\ref{dfn:jsj})  are arranged on levels $3m+2$ (for $m\in \mathbb Z_{\geq 0}$).
	\item Every edge in $T(X)$ is directed from the initial vertex to the terminal vertex (say from $u$ 
	to $v$) so that if $u$ is at level $3m$, $3m+1$, or $3m+2$, then $v$ is at level $3m+1$, $3m+2$, 
	or $3m+3$; moreover, the level of $v$ is strictly greater than that of $u$.  We call  $u$ the {\it predecessor} of 
	$v$ and $v$ is the {\it direct descendant} of $u$.
	\end{enumerate}

\bigskip\noindent
We are now ready to construct $T(X)$.  The root of $T(X)$ is a vertex labeled $X$.  
Assume first that $X$ is geometric.  Then $X$ is placed in
level $0$.  If $X$ is Seifert fibered or a sol manifold, 
then the corresponding vertex is a leaf: there
are no edges out of $X$.  If $X$ is hyperbolic, we place one edge $e$ out of $X$ for each 
\mnh\ filling on $X$ (say $\alpha$), and the terminal vertex of $e$ is labeled $X(\alpha)$.
Since $\alpha$ is a non-hyperbolic filling, $X(\alpha)$ is one of the following:
	\begin{enumerate}
	\item Reducible: then $X(\alpha)$ is placed at level $1$.
	\item JSJ: then $X(\alpha)$ is placed at level $2$. 
	\item Seifert fibered or sol manifold: then $X(\alpha)$ is placed at level $3$.
	\end{enumerate}

Next suppose that $X$ is not prime.  Then the corresponding vertex is placed at
level $1$.  Let $X_{1},\dots,X_{k}$ be the factors of the prime decomposition of $X$.
We place $k$ edges out of
$X$ with terminal vertices labeled $X_{1},\dots,X_{k}$.  Each $X_{i}$ is either
JSJ, hyperbolic, Seifert fibered, or sol.
Accordingly, the corresponding vertex is placed at level $2$ (if it is JSJ) or $3$ (in all other cases).
Note that if $i \neq i'$ then $X_{i}$ and $X_{i'}$ correspond to distinct vertices
even if $X_{i} \cong X_{i'}$.

Finally, let $X$ be a JSJ manifold.  The corresponding vertex is placed at level
$2$.  Let $X_{1},\dots,X_{k}$
be the components of the torus decomposition of $X$.  We place $k$ edges out of
$X$, with terminal vertices labeled $X_{1},\dots,X_{k}$.  Each $X_{i}$ is 
hyperbolic or Seifert fibered.  Accordingly, it is placed at level $3$.
As above,  if $i \neq i'$ then $X_{i}$ and $X_{i'}$ correspond to distinct vertices
even if $X_{i} \cong X_{i'}$.  

The construction is recursive, and if $X_{1}$ is a
direct descendant of $X$, then we place $T(X_{1})$ with the root at the vertex
labeled $X_{1}$; since $X_{1}$ is a direct descendant of $X$ its level is
$1$, $2$ or $3$.  If the level is $3$ we shift all the levels in $T(X_{1})$ by $+3$.

\bigskip\noindent
Let us discuss an example.  Let $X_{3}$ be a hyperbolic manifold with $m+1$ boundary components.  
Let $X_{2}$ be double of $X_{3}$ along
$m$ boundary components.  Hence $X_{2}$ is a toroidal manifold with $2$ boundary components.  
Let $X_{0}$ be the manifold obtained from $X_{2}$
by drilling a hyperbolic knot (which is known to exist,
essentially by Myers~\cite{Myers}).  
Hence $X_{0}$ is a hyperbolic manifold with $3$ boundary components.
Now in $T(X_{0})$ there is an edge connecting $X_{0}$ to $X_{2}$, 
corresponding to a \mnh\ filling (since we fill only one boundary component, the filling
must be minimal).  Next we see two edges from $X_{2}$ to two copies of $X_{3}$.  
Thus as we move down $T(X_{0})$ we start with a single hyperbolic manifold with 
three boundary components, and later encounter 
two hyperbolic manifolds, each with $m+1$ boundary components for an arbitrary $m$
(and, perhaps, other manifolds as well---this may not be all the edges between the levels $0$
and $3$).  
As the tree of $T(X_{0})$ contains two copies of $T(X_{3})$ it can be quite big.
We leave it as an exercise to the reader to construct other complicated examples; for instance, 
given an integer $m$, construct a hyperbolic manifold $X$ with one 
boundary component, so that $T(X)$ admits a directed path of length $m$.  

\bigskip\noindent
Our goal is to use $|T(X)|$, the number of vertices in $T(X)$, 
as a basis for induction.  For that we need:

\begin{prop}
\label{prop:T(X)isFinite}
$T(X)$ is finite.
\end{prop}

\begin{proof}
Let the {\it degree} of a vertex $v$ be the number of direct descendants of $v$ (that is, the number 
of edges {\it out} of $v$).
It is easy to see that the degree of every vertex $v$ is finite:
	\begin{enumerate}
	\item If $v$ corresponds to a hyperbolic manifold $X$, the degree is finite because $X$
	admits only finitely many \mnh\ fillings (Proposition~\ref{pro:mnhIsFinite}).
	\item If $v$ corresponds to a Seifert manifold or a sol manifold the degree is zero by construction. 
	\item If $v$ corresponds to a reducible manifold the degree is the number of factors in its prime decomposition
	and is therefore finite.
	\item If $v$ corresponds to a JSJ manifold the degree is finite by the finiteness of the torus decomposition~\cite{jacoshalen}
	and~\cite{Johannson}.
	\end{enumerate}

The problem is avoiding an infinite branch.   Assume there is such a branch.
Now by construction every edge starting on level $3m$, $3m+1$, or $3m+2$ ends at a level $3m+1$, $m+2$, or
$3m+3$.  The vertices at level $3m+3$ that do not correspond to hyperbolic manifolds are leaves; hence
the branch must admit a vertex corresponding to a hyperbolic manifold on every level $3m$. 
Let $X_{3m}$ denote this hyperbolic manifold.

We will use the {\it Gromov Norm}; for definition see \cite{GromovBoundedCohomology}.
The Gromov norm has the following properties, proved in~\cite{GromovBoundedCohomology}
and~\cite[Theorem~1]{soma}.  Here, $X$ is a compact orientable manifold so that $\partial X$ consist of tori.
	\begin{enumerate}
	\item If the Gromov norm is non-zero, then it strictly decreases under any filling.  
	\item The Gromov norm of $X$ equals the sum of the Gromov norms of the components of the 
	prime decomposition of $X$.
	\item The Gromov norm is additive under decomposition along essential tori.  
	\item The Gromov norm is additive under disjoint union.
	\item The Gromov norm is non-negative.
	\end{enumerate} 
Since $X_{3m+3}$ is obtained from $X_{3m}$ by filling, then (possibly) reducing along
essential spheres and discarding components, and (possibly) decomposing along essential tori
and discarding components, we see that the Gromov norm of $X_{3m+3}$ is strictly smaller than
that of $X_{3m}$.  It is well known that the hyperbolic volume is
a constant multiple of the Gromov norm, and so we see that $X_{0},X_{3},X_{6},\dots$ 
forms a sequence of hyperbolic manifolds with 
$\vol[X_{0}] > \vol[X_{3}]> \vol[X_{6}] > \cdots$.  But this cannot be, as hyperbolic volumes are well ordered. 
\end{proof}

\bigskip\bigskip\noindent
In the following sections, we will use $T(X)$ inductively.  The problem is the we are {\it only} dealing with
filling, while non-prime and JSJ manifolds are treated differently on $T(X)$
(reduction along spheres and essential tori, respectively). 
Let $\alpha = (\alpha_{1},\dots,\alpha_{n})$ be a 
multislope.  If there exist a \mnh\ filling $\alpha'$ so that $\alpha' \pf \alpha$
(possibly $\alpha' = \alpha$) we say that {\it $\alpha$ admits a  
\mnh\ partial filling}.  Note that if $\alpha$ is a non hyperbolic
multislope, which by definition means that $X$ is hyperbolic and $X(\alpha)$ is not, then
$\alpha$ admits a \mnh\ partial filling.

We will need the following proposition; it is useful, for example, when $\alpha$ is a hyperbolic
multislope that admits a \mnh\ partial filling (note that such multislopes do exist;
constructing examples is quite easy).

\begin{prop}
\label{prop:ObtainedByFilling}
Let $X$ be a hyperbolic manifold and 
$\alpha$ a multislope of $\partial X$.  Suppose that $\alpha$ admits a \mnh\ partial filling $\alpha'$
so that either $X(\alpha')$ is not prime or it is JSJ.  Suppose further that
$X(\alpha)$ is irreducible and a toroidal.

Then $X(\alpha)$ is obtained by filling some descendant of $X(\alpha')$ on $T(X)$. 
\end{prop}

\begin{proof}
Since $\alpha' \pf \alpha$, it is clear that
$X(\alpha)$ is obtained from $X(\alpha')$ by filling.  

Assume first that $X(\alpha')$ is reducible.  Then the descendants of $X(\alpha')$ 
are, by construction of $T(X)$, the factors of the prime decomposition of $X(\alpha')$.
Since $X(\alpha)$ is a prime manifold that is obtained from $X(\alpha')$ by
filling, it is easy to see that in that filling all the components in the prime decomposition
of $X(\alpha')$ become balls except at most one.  The proposition follows in this case.

Next assume that $X(\alpha')$ is JSJ.   By assumption, $X(\alpha)$ is prime 
and a toroidal.  Therefore every torus $T$ in $X(\alpha)$ fulfills at least one of the following
three conditions:
	\begin{enumerate}
	\item $T$ is boundary parallel.
	\item $T$ bounds a solid torus.
	\item $T$ bounds a knot exterior contained in a ball.
	\end{enumerate}
Since $X(\alpha)$ is obtained from $X(\alpha')$ by Dehn filling, $X(\alpha') \subset X(\alpha)$.	
Let $T \subset X(\alpha')$ be a torus.  Considering $T$ as a torus in $X(\alpha)$ allows us 
to endow it with a co-orientation as follows:
	\begin{enumerate}
	\item If $T$ is boundary parallel, we co-orient it towards the boundary.
	\item If $T$ bounds a solid torus, we co-orient it towards the solid torus.
	\item If $T$ bounds a knot exterior contained in a ball, we co-orient it towards the knot exterior.
	\end{enumerate}
Note that some tori may get both co-orientations (for example, a boundary parallel torus in a solid torus
or an unknotted torus in a ball).
In that event we pick an orientation arbitrarily.

Let $\mathcal{T}$ be the tori of the JSJ decomposition of $X(\alpha')$.  
Let $\Gamma$ be the {\it dual graph} to $\mathcal{T}$,
which we defined to be the graph that has one vertex for each component 
of the torus decomposition of $X(\alpha')$ and one edge
for each torus $T \in \mathcal{T}$, connecting the vertices that correspond 
to the components adjacent to $T$.  
Since $X$ is connected so is $X(\alpha')$;
hence $\Gamma$ is connected.  Since $X(\alpha)$ is a toroidal and irreducible
it contains no non-separating tori, and is follows that neither does $X(\alpha')$; hence $\Gamma$ contains 
no cycles.  We conclude that $\Gamma$ is a tree.  We endow each 
edge of $\Gamma$ with an orientation consistent with the co-orientation of the corresponding torus of $\mathcal{T}$.
Using induction, it is easy to see that $\Gamma$ admits a sink (a vertex connected only to edges that point away from it):
since $\mathcal{T} \neq \emptyset$, $\Gamma$ contains an edge.  
Thus $\Gamma$ admits a leaf (a vertex connected to only one other vertex), say $v$.  If the edge connected to
$v$ points away from $v$, $v$ is a sink.  Otherwise, removing $v$ and the edge attached to it we obtain a tree
with fewer vertices than $\Gamma$.  
If the tree obtained contains only one
vertex (and hence no edges), that vertex is a sink of $\Gamma$.  Otheriwse,
by induction the tree obtained admits a sink; it is clear that the same vertex is a sink for 
$\Gamma$ as well.

Let $X'$ be a component of the JSJ decomposition of $X(\alpha')$ that corresponds to a sink.  
By construction of $T(X)$, $X'$ is a direct descendant of $(\alpha')$.  
We claim that $X(\alpha)$ is obtained from $X'$ by filling (this allows for the
possibility that $X' \cong X(\alpha)$ and no boundary component is filled).  
To see this, let  $T$ be a component 
of $\partial X'$.  We will denote the component of $X(\alpha')$ 
cut open along $T$ that is disjoint from $X'$ as $Y$.  
As above, considering $T \subset X(\alpha)$ we see three cases:
	\begin{enumerate}
	\item $T$ is parallel to a component of $\partial X(\alpha)$:  equivalently,
	$Y(\alpha|_{\partial Y}) \cong T^{2} \times [0,1]$.  We remove $Y$ and the solid tori
	attached to it.  The manifold obtained is not changed.
	\item $T$ bounds a solid torus outside $X'$: equivalently, 
	$Y(\alpha|_{\partial Y}) \cong D^{2} \times S^{1}$.
	Again we remove $Y$, and in the process of obtaining 
	$X(\alpha)$ from $X'$ we consider attaching a solid torus to $T$ along the 
	slope defined by $Y(\alpha|_{\partial Y})$.
	\item $Y$ is a knot exterior contained in a ball $D$: if $\partial Y$ is compressible then
	$Y \cong D^{2} \times S^{1}$; this was treated in case~(2), 
	and we assume as we may that this does not happen.
	Thus $T$ is essential in $Y(\alpha|_{\partial Y})$, and since $X(\alpha)$ is a toroidal $T$ must compress
	in $X(\alpha)$ away from $Y$.  Let $D$ be a compressing disk.  
	It is now easy to see that replacing $Y$ with a solid torus so that the meridian of
	the solid torus intersects $\partial D$ exactly once does not change the ambient manifold.
	This can be seen as ``unknotting'' $T$ in $D$, see Figure~\ref{fig:unknotting}.
\begin{figure}
\includegraphics[width=3in]{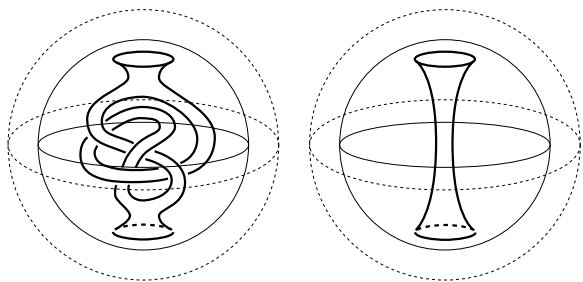}
\caption{Unknotting in a ball}
\label{fig:unknotting}
\end{figure}
	\end{enumerate}
Repeating this process on all the components of $\partial X'$ we see that $X(\alpha)$ is 
obtained from $X'$ by attaching solid tori.  

This completes the proof of Proposition~\ref{prop:ObtainedByFilling}.
\end{proof}

\bigskip\bigskip\noindent
So here is where we stand: if we fill to obtain a prime a-toroidal manifold, Proposition~\ref{prop:ObtainedByFilling}
allows us to travel down $T(X)$ from a vertex at level $3m+1$ or $3m+2$.  
Assume, in addition, that 
$X(\alpha)$ is not hyperbolic.   A manifold corresponding to a vertex labeled $3m$ 
is Seifert fibered, sol, or a hyperbolic.  
The first two cases require direct analysis; in the final
case we are guaranteed to have a \mnh\ filling that allows us to keep going down $T(X)$.  

The problem occurs when we want to study hyperbolic fillings.  Obviously, one cannot 
expect {\it every} filling to admit a \mnh\ partial filling;
this is simply false.  Somewhat surprisingly we have the following proposition, that allows us to go down $T(X)$
in certain circumstances; recall that a multislope that does not admit a \mnh\ partial filling is called 
totally hyperbolic:

\begin{prop}
\label{prop:UsingT(X)forHyperbolicFilling}
Let $X$ be a hyperbolic manifold and  $\epsilon > 0$.  Let $\mathcal{A}$ be the set
of multislopes of $\partial X$ so that every $\alpha \in \mathcal{A}$ we have that
$X(\alpha)$ is hyperbolic, and every geodesic in $X(\alpha)$ is longer then $\epsilon$.

Then there are only finitely many totally hyperbolic mutlislopes in $\mathcal{A}$ 
\end{prop}

\begin{rmk}
Since $\partial X$ may have arbitrarily many components, it is easy to construct examples
of hyperbolic manifolds with infinitely many multislopes $\alpha$ so that $X(\alpha)$
is hyperbolic but does not admit a geodesic shorter than $\epsilon$,
for a fixed $\epsilon>0$.  This proposition allows us to study those:
there is a {\it finite set} of totally hyperbolic mulitslopes, and every other multislope in
$\mathcal{A}$ admits a \mnh\ partial filling; these correspond to moving down $T(X)$.
\end{rmk}

\begin{proof}
Denote the components of $\partial X$ by $T_{1},\dots,T_{n}$.  We assume as we may that $\mathcal{A}$
is infinite.  Let $\alpha^{j} = (\alpha_{1}^{j},\dots,\alpha_{n}^{j})$ be an infinite
set of multislopes in $\mathcal{A}$; we will prove the theorem by showing that
some $\alpha^{j}$ admits a non hyperbolic partial filling.
The remainder of the proof is very similar to the proof of Proposition~\ref{pro:mnhIsFinite}
and we only paraphrase it here.

After subsequencing we may  assume that for each $i$ one of the following holds:
	\begin{enumerate}
	\item For $j \neq j'$, $\alpha^{j}_{i} \neq \alpha^{j'}_{i}$; we assume in addition 
	that for all $j$, $\alpha_{{i}}^{j} \neq \nos$.
	\item For any $j,j'$, $\alpha^{j}_{{i}} = \alpha^{j'}_{i}$ (possibly $\alpha^{j}_{i} = \nos$).
	\end{enumerate}
After renumbering we may assume that the constant slopes are $\alpha_{i}^{j}$ 
for $i > k$  (for some $0 \leq k \leq n$).  Since 
$\mathcal{A}$ is infinite, $k \geq 1$.   For $i > k$, we drop the superscript from $\alpha_{i}$.

Let $\widehat X = X(\nos,\dots,\nos,\alpha_{k+1},\dots,\alpha_{n})$.  Then 
$X(\alpha^{j}) = \widehat X(\alpha_{1}^{j},\dots,\alpha_{k}^{j})$.  If $\widehat X$ were
hyperbolic, then by Thurston's Dehn Surgery Theorem 
for large enough $j$, we would obtain a hyperbolic manifold with $k$ 
geodesics of length less than $\epsilon$; as $k \geq 1$ this violates our assumption.  
Thus $\widehat X$ is non-hyperbolic, and 
$(\nos,\dots,\nos,\alpha_{k},\dots,\alpha_{n})  \pf (\alpha_{1}^{j},\dots,\alpha_{n}^{j})$ 
is a non hyperbolic partial filling.
The proposition follows.
\end{proof}

We end this section with the following lemma:

\begin{lem}
\label{lem:TreeOfJSJ}

Let $X$ be a JSJ manifold and $X_{0}$ a connected manifold that is obtained 
as the union of a strict subset of the components of
the torus decomposition of $X$.  

Then $|T(X_{0})| < |T(X)|$.
\end{lem}

\begin{proof}
If $X_{0}$ is a component of the torus decomposition of $X$ then it corresponds to
a direct descendant of the root of $T(X)$, and clearly $T(X_{0}) \subsetneqq (X)$.
The lemma follows in this case.

If $X_{0}$ is the union of more than one component of the torus decomposition of
$X$, we embed $T(X_{0})$ into $T(X)$ by placing the root of $T(X_{0})$ at the
root of $T(X)$ and using the edges the correspond to the components of the torus
decomposition of $X$ that appear in $X_{0}$.  By assumption this is a strict subset
of the components of the torus decomposition of $X$, and hence no all the edges are used.
After this embedding we can view $T(X_{0})$ and a subtree of $T(X)$, and we
see that  $T(X_{0}) \subsetneqq (X)$.  The lemma follows.
\end{proof}

\section{Cosmetic surgery on $T^2 \times I$}
\label{sec:CosmeticSurgeryOnT2XI}

\bigskip
\noindent
There are two types of theorems proved using $T(X)$.  The first asks ``how much can a manifold get twisted 
when performing cosmetic surgery'' and the second asks ``how many fillings can
result in a manifold fulfilling such-and-such condition''.  
In the next three section we prove theorems of the first type. 
Recall that a cosmetic surgery on $L \subset M$ is a surgery with multislope $\alpha$ 
so that $L(\alpha) \cong M$.  
Below we consider cosmetic surgery on links in $T^{2} \times [0,1]$.
Note that $T^{2} \times [0,1]$ gives a natural projection from $T^{2} \times \{1\}$ to $T^{2} \times \{0\}$;
however, after cosmetic surgery, this identification may change.  Hence the image of a specific slope in
$T^{2} \times \{1\}$ may give an infinite set after cosmetic surgeries (and this does in fact happen).  The theorem
below controls this set, and more generally, the image of a bounded set:

\bigskip\bigskip
\noindent

\begin{thm}
\label{thm:CosmeticSurgeryOnT2XI}

Let $B$ a bounded set of slopes of $T^{2} \times \{1\}$, $L$ be a link in 
$T^{2} \times [0,1]$, and $\mathcal{A} = \{\alpha \ | \ L(\alpha) \cong T^{2} \times [0,1]\}$.  
For $\alpha \in \mathcal{A}$, 
let $B_{\alpha}$ be the set of slopes of $T^{2} \times \{0\}$ that are obtained by projecting $B$
via the natural projection.  

Then $\cup_{\alpha \in \mathcal{A}} B_{\alpha}$ is bounded.
\end{thm}

\begin{proof}
We will denote $T^{2} \times [0,1] \setminus N(L)$ as $X$ and $\partial N(L)$ as $\mathcal{T}$.
Note that we may regard $\mathcal{A}$ as a set of multislopes on $\mathcal{T} \subset \partial X$.

The proof is an induction on $|T(X)|$.

\bigskip

\noindent {\bf Assume that $X$ is not prime.}  Let $X_{1}$ be the factor of the prime decomposition
of $X$ that contains $T^{2} \times \{0\}$ and  $T^{2} \times \{1\}$ (note that both are contained
in the same factor).  Then any $\alpha \in \mathcal{A}$ induces $\alpha|_{\partial X_{1}}$,
and $X_{1}(\alpha|_{\partial X_{1}}) \cong T^{2} \times S^{1}$.  Moreover, all the other factors of 
the prime decomposition of $X$ become balls after filling and therefore do not effect the
identification of $T^{2} \times \{1\}$ with  $T^{2} \times \{0\}$.  Denote the image of $B$
under the natural projection using the product structure of 
$X_{1}(\alpha|_{\partial X_{1}}) \cong T^{2} \times [0,1]$
by $B_{\alpha|_{\partial X_{1}}}$.
Thus for every $\alpha \in \mathcal{A}$,
$B_{\alpha} = B_{\alpha|_{\partial X_{1}}}$; therefore
 $\cup_{\alpha \in \mathcal{A}} B_{\alpha}  = \cup_{\alpha \in \mathcal{A}} B_{\alpha|_{\partial X_{1}}}$.
Since $|T(X_{1})| < |T(X)|$, $B_{\alpha|_{\partial X_{1}}}$ is bounded  by induction.
The proposition follows in this case.

We assume from now on that $X$ is prime.

\bigskip

\noindent {\bf Assume that $X$ is Seifert fibered.}
Fix a Seifert fiberation of $X$ and a multislope $\alpha \in \mathcal{A}$.
If, for some $T_{i} \subset \mathcal{T}$, $\alpha|_{T_{i}}$ is a fiber 
then $X(\alpha)$ contains a 
sphere that separates $T^{2} \times \{0\}$ from $T^{2} \times \{1\}$, which is impossible since
$X(\alpha) \cong T^{2} \times [0,1]$.  Hence the fiberaion of $X$ extends to a finration
of $T^{2} \times [0,1]$.  We conclude that $X$ is obtained from $T^{2} \times [0,1]$ by
removing fibers.

Note that up to diffeomorphism the only Seifert fiberation of $T^{2} \times [0,1]$ is annulus cross $S^{1}$.
To see this, simply note that if the base orbifold of a Seifert fibered manifold
has positive genus or positive number of orbifold points then it admits a filling
that is not a lens space.  Thus $X$ is obtained from an annulus cross $S^{1}$
by removing a set of $n$ curves that has the form 
$\{p_{1},\dots,p_{n}\} \times S^{1}$, showing that $X$ 
an $n$-times punctured annulus cross $S^{1}$.
Since $X(\alpha)$ results in a Seifert fibered manifold with no exceptional fiber,
the core of the solid torus attached to $T_{i}$ is not an exceptional fiber, and 
hence has the form $p_{i}/1$ in the 
Seifert notation.  Suppose $n>1$.  Following Seifert's original work \cite{seifert}, by performing $k$ twists about an annulus 
connecting fibers with Seifert invariants $\frac {p_{1}}{q_{1}}$ and 
$\frac {p_{i}}{q_{i}}$ the invariants change as follows (for an arbitrary $k \in \mathbb Z$):
$$(p_{1}/q_{1}) \mapsto (p_{1} + k q_{1}/q_{1}) \mbox{ and } (p_{i}/q_{i}) \mapsto (p_{i} - k q_{i}/q_{i}).$$
As in our case $q_{i} = 1$, by choosing $k$ appropriately, we may assume that the filling of $T_{i}$ is
of the form $0/1$, and ignore it (for $i>1$).  Thus we have reduced the problem to the
case $n=1$.   Let $A \subset X$ be an embedded vertical annulus connecting $T_{1}$ with a curve
on $T^{2} \times \{0\}$ that we will denote as $\gamma$. 
By twisting about $A$, we see that the effect of the cosmetic surgery 
is the same as $D_{\gamma}^{n}$, an $n$ power of a Dehn twist about $\gamma$, for some $n \in \mathbb Z$.  
Hence the image of $B$ after all possible cosmetic surgeries on $L$ is:
$$\{D_{\gamma}^{n}(\beta)\ |\ \beta \in B, n \in \mathbb{Z}\}.$$
By Proposition~\ref{prop:PropertiesOfBoundedSets}~(3), this set is bounded.

We assume from now on that $X$ is prime and not Seifert fibered.

\bigskip
\bigskip

\noindent  {\bf Assume that $X$ is a JSJ manifold.}  We will denote the tori of the JSJ
decomposition as $\mathcal{F}$.  Let $\Gamma$ be the graph dual to the torus decomposition
as defined in~\ref{dfn:jsj}.  Since $T^{2} \times [0,1]$
is connected and admits no non-separating tori, $\Gamma$ is a tree.
Note that $\Gamma$ has two (not necessarily distinct) vertices, denoted $v_{0}$
and $v_{1}$, that correspond to the components of the torus decomposition of $X$
that contain $T^{2} \times \{0\}$ and $T^{2} \times \{1\}$.  There are two cases
to consider:

\bigskip
\noindent{\bf Case One:}  $v_{0} \neq v_{1}$.  Let $e$ be an edge on the shortest path 
connecting $v_{0}$ and $v_{1}$ and $F \in \mathcal{F}$ the torus corresponding to $e$.  
In any filling, $F$ separates $T^{2} \times \{0\}$
from $T^{2} \times \{1\}$.  Thus for every $\alpha \in \mathcal{A}$,
we have that $F \subset X(\alpha) \cong T^{2} \times [0,1]$  
is isotopic to $T^{2} \times \{1/2\}$.   Let $X_{0}$ and $X_{1}$ be the components of 
$X$ cut open along $F$ so that $T^{2} \times \{i\} \subset X_{i}$ ($i=1,2$).  
We see that the fillings induced by $\alpha \in \mathcal{A}$
fulfill:
$$X_{i}(\alpha|_{\partial X_{i}}) \cong T^{2} \times [0,1].$$
For $i=0,1$, let
$$\mathcal{A}_{i} = \{\alpha|_{\mathcal{F} \cap \partial X_{i}}\ | \ \alpha \in \mathcal{A}\}.$$
Then for any $\alpha \in \mathcal{A}$, $X(\alpha) = X_{0}(\alpha_{0}) \cup X_{1}(\alpha_{1})$,
where here $\alpha_{i} = \alpha|_{\mathcal{F} \cap \partial X_{i}}$.  
By Lemma~\ref{lem:TreeOfJSJ}, $|T(X_{0})|, |T(X_{1})| < |T(X)|$.
Applying induction to $X_{1}$ we conclude that
$$\cup_{\alpha_{1} \in \mathcal{A}_{1}} B_{\alpha_{1}}$$
is a bounded set of slopes of $F$; we will denote it as $B^{F}$.  Next we apply induction to $X_{0}$
and conclude that
$$\cup_{\alpha_{0} \in \mathcal{A}_{0}} B^{F}_{\alpha_{0}}$$
is  bounded.  It is clear from the discussion above that 
$\cup_{\alpha \in \mathcal{A}} B_{\alpha} \subset \cup_{\alpha_{0} \in \mathcal{A}_{0}} B^{F}_{\alpha_{0}}$;
the theorem follows in case one.

\bigskip
\noindent{\bf Case Two:}  $v_{0} = v_{1}$, that is, that both $T^{2} \times \{0\}$ and $T^{2} \times \{1\}$
are contained in the same component of the torus decomposition of $X$, say $X_{0}$.
We will denote 
the components of $\mbox{cl}(X \setminus X_{0})$ as $X_1,\dots,X_{k}$.
Since $T^{2} \times [0,1]$ does not admit a non separating torus, if
$X_{i} \cap X_{0}$ consists of more than one component (for some $i$) then
$\mathcal{A} = \emptyset$ and there is nothing to prove.  We assume as we may
that for every $i$, $X_{i} \cap X_{0}$ is a single torus which  we will denote as $T'_{i}$. 
Note that $T'_{i}$ is a component of the JSJ decomposition of $X$ and 
under the assumptions of case two it cannot be boundary parallel in 
$X(\alpha) \cong T^{2} \times [0,1]$.
Hence every $\alpha \in \mathcal{A}$ induces
a filling on every $X_{i}$ fulfilling exactly one of the following conditions:
\begin{enumerate}
\item $X_{i}(\alpha|_{\partial X_{i}}) \cong D^{2} \times S^{1}$.
\item $X_{i}(\alpha|_{\partial X_{i}}) \cong E(K_{i})$, where $K_{i} \subset S^{3}$ is a non
trivial knot and $X_{i}(\alpha|_{\partial X_{i}}) \subset D_{i}$ for some ball $D_{i} \subset X(\alpha)$. 
\end{enumerate}
By Lemma~\ref{lem:KnotExteriorsAreDisjoint} we assume as we may that the balls
$D_i \subset X(\alpha)$ are disjointly embedded.
Given $\alpha \in \mathcal{A}$ we construct a multislope $\alpha_{0}$ of $\mathcal{T}_{0}$ as follows:
\begin{enumerate}
\item If $X_{i}(\alpha|_{\partial X_{i}}) \cong D^{2} \times S^{1}$, then $\alpha_{0}|_{T_{i}}$
is the meridian of the solid torus $X_{i}(\alpha|_{\partial X_{i}})$.
\item If $X_{i}(\alpha|_{\partial X_{i}}) \cong E(K_{i})$ we pick a slope that intersects the
meridian of $E(K_{i})$ exactly once (recall Figure~\ref{fig:unknotting}).
\end{enumerate}
Then $X_{0}(\alpha_{0}) \cong X(\alpha)$ and by construction the product structures of   
$X_{0}(\alpha_{0})$ and $X(\alpha)$ induce the same projection from the slopes
of $T^{2} \times \{1\}$ to those of of $T^{2} \times \{0\}$.

Let $\mathcal{A}_{0}$ be the set of all multislopes of $X_{0}$ so that 
$X_{0}(\alpha_{0}) \cong T^{2} \times [0,1]$.  For each $\alpha_{0}\in \mathcal{A}_{0}$
we will denote the image of $B$ under the projection induced by the product structure as 
$ B_{\alpha_{0}}$.   If $\alpha_{0}$ was constructed as above for some $\alpha \in \mathcal{A}$,
then $B_{\alpha} = B_{\alpha_{0}}$.  Thus we see:
$$\cup_{\alpha \in \mathcal{A}} B_{\alpha} \subset \cup_{\alpha_{0} \in \mathcal{A}_{0}} B_{\alpha_{0}}.$$
By Lemma~\ref{lem:TreeOfJSJ}, $|T(X_{0})| < |T(X)|$; hence by induction
$\cup_{\alpha_{0} \in \mathcal{A}_{0}} B_{\alpha_{0}}$ is bounded.
The theorem follows in case two.
 
We assume from now on the $X$ is prime, not Seifert
fibered, and not a JSJ manifold.

\bigskip

\noindent  {\bf Assume that $X$ is hyperbolic.}   As $T^{2} \times [0,1]$ is non-hyperbolic,
any multislope $\alpha$ with $X(\alpha) \cong T^{2} \times [0,1]$ admits a \mnh\ partial
filling.  By Proposition~\ref{pro:mnhIsFinite} $X$ admits only finitely many
\mnh\ multislopes, say denoted by $\alpha_{1},\dots,\alpha_{k}$.  For $1 \leq i \leq k$,
we will denote $X(\alpha_{i})$ as $X_{i}$.
Let $\mathcal{A}_{i} = \{\alpha|_{\partial{X}_{i}} \ | \ \alpha \in \mathcal{A}, \ \alpha_{i} \pf \alpha\}$,
that is, $\mathcal{A}_{i}$ consists of the multislopes induced on $X_{i}$ by multislopes $\mathcal{A}$
that admits $\alpha_{i}$ as a partial filling.
For $\alpha_{i} \in\mathcal{A}_{i}$, we will denote the image of $B$ under the 
projection induced by the product structure of $X_{i}(\alpha_{i})$ as $B_{\alpha_{i}}$. 
Since $X_{i}$ is a direct descendant of the root of $T(X)$, $|T(X_{i})| < |T(X)|$.
By induction $\cup_{\alpha_{i} \in \mathcal{A}_{i}} B_{\alpha_{i}}$ is bounded.  
Since every $\alpha \in \mathcal{A}$ admits a \mnh\ partial filling, for every such $\alpha$
there exists $i$ so that
$$X(\alpha) = X_{i}(\alpha_{i}),$$
where here  $\alpha_{i} = \alpha|_{\partial X_{i}} \in \mathcal{A}_{i}$.  Hence:
$$\cup_{\alpha \in \mathcal{A}} B_{\alpha} = \cup_{i=1}^{k} 
\cup_{\alpha_{i} \in \mathcal{A}_{i}} B_{\alpha_{i}}.$$
Since bounded sets are closed under finite union, $\cup_{\alpha \in \mathcal{A}} B_{\alpha}$
is bounded.

This concludes the proof of Theorem~\ref{thm:CosmeticSurgeryOnT2XI}.
\end{proof}

\section{Cosmetic surgery on solid torus}
\label{sec:CosmeticSurgeryOnSolidTorus}

\bigskip
\noindent
Let $V$ be a solid torus and $L \subset V$ a link.  In this section we address the following question: 
how many slopes one $\partial V$ become the boundary of a meridian disk after cosmetic surgery on $L$?  
Before stating the main theorem of this section, we consider some examples.  If some component of 
$L$ is the core of $V$, then (trivially) the answer is {\it every slope.}  Next, let $L' = K'_{1} \cup K'_{2}$ 
be a two component link, where $K_{1}'$ is a knot that is not a torus knot and admits a non-trivial cosmetic 
surgery (see Gabai~\cite{gabai} and Berge~\cite{berge}) and $K_{2}'$ is a core.  Assume further that $K'_{2}$ 
was isotoped to be in a very ``complicated'' configuration with respect to $K_{1}'$ (we allow $K_{2}'$ to pass 
through $K_{1}'$ during this isotopy, changing $L'$ but not $K_{1}'$ or $K'_{2}$).  
Let $L = K_{1} \cup K_{2}$ be the image of $K_{1}'$ and $K_{2}'$
in the solid torus obtained by cosmetic surgery on $K_{1}'$.  Then, due to the isotopy discussed above, we 
expect that $K_{2}$ is not a core of the solid torus and certainly $K_{1}$ is not.  However, $L$ admits 
cosmetic surgeries that realize every slope on the boundary of the solid torus as the boundary of the meridian disk.  
The trouble is that although $K_{2}$ is not a core
of the solid torus it becomes a core after cosmetic surgery.

\bigskip\noindent
With this in mind we state the main result of this section; note that condition~(2)
of the theorem is equivalent to requiring that non of the cores of the attached solid tori
is a core of $L(\alpha)$:

\begin{thm}
\label{thm:SlopesOnSolidTorus}
Let $V$ be a solid torus and $L \subset V$ a link.  Consider
the set of multislopes $\mathcal{A}$ of $L$ so that for any 
$\alpha \in \mathcal{A}$ the following two conditions hold:
	\begin{enumerate}
	\item $L(\alpha) \cong D^{2} \times S^{1}$.
	\item For any $\alpha' \pf \alpha$, $L(\alpha') \not\cong T^{2} \times [0,1]$. 
	\end{enumerate}
Then the set of slopes on $\partial V$ that bounds a disk in $\{L(\alpha)\ |\ \alpha \in S\}$ is bounded.
\end{thm}

\begin{proof}
Let $X = V \setminus N(L)$ and denote the components of $\partial X$ by 
$T,T_{1},\dots,T_{n}$, with $T=\partial X$.  We will denote $L(\alpha)$ as
$X(\alpha)$; this is consistent with the notation of subsection~\ref{subsec:notaion}
and should cause no confusion.  We induct on $|T(X)|$.

\bigskip
\noindent  {\bf Assume $X$ is not prime.}  Let $S \subset X$ be an essential separating sphere that
realizes the decomposition $X = X' \# X''$, so that $X'$ is prime and $T \subset X'$ 
(we are not assuming that $X''$ is prime).
For any $\alpha \in \mathcal{A}$, $X(\alpha) \cong D^{2} \times S^{1}$ which is prime.  
Therefore after filling one side of $S$ becomes a ball $D$ 
and the other becomes a solid torus whose boundary is $T$.  We conclude that $\alpha$ induces $\alpha|_{\partial X'}$
and the following conditions hold:
	\begin{enumerate}
	\item $X'(\alpha|_{\partial X'}) \cong D^{2} \times S^{1}$ and its meridian is the same slope of $T$ 
	as the meridian of $X(\alpha)$.
	\item There does not exist $\alpha' \pf \alpha|_{\partial X'}$ so that $X'(\alpha') \cong T^{2} \times [0,1]$. 
	Otherwise, there would be $\alpha'' \pf \alpha$ so that $X(\alpha'') \cong T^{2} \times [0,1]$, 
	contradicting our assumptions.
	\item $|T(X')| < |T(X)|$: this is immediate from the construction of $T(X)$, 
	since $X'$ is a direct descendant of the root $X$.
	\end{enumerate}
Applying the inductive hypothesis, we conclude that the set of meridians of $X'(\alpha|_{\partial X'})$ 
for $\alpha \in \mathcal{A}$ is bounded; the theorem follows in this case.

We assume from now on that $X$ is prime.

\bigskip
\noindent {\bf Assume $X$ is Seifert fibered.}  Fix a Seifert fiberation on $X$.  
First assume that, for some $\alpha \in \mathcal{A}$ and some $i$, $\alpha|_{T_{i}}$
is the fiber; we will denote the meridian disk of the solid torus attached to $T_{i}$ 
as $D_{i}$.  
Then the result of gluing $D_{i}$ to a vertical annulus connecting $\partial D_{i}$
to the regular fiber on $T$ is a meridian disk for $X(\alpha)$.  Hence the meridian
of $X(\alpha)$ is the regular fiber.  From now on we consider the subset of $\mathcal{A}$
consisting of multislopes for which
$\alpha|_{T_{i}}$ is not a fiber in the fiberation of $X$ (for all $i$).  To avoid overly complicated
notation we do not rename $\mathcal{A}$.

Thus the fiberation of $X$ extends to a fiberation of $X(\alpha)$.
Since $X(\alpha) \cong D^{2} \times S^{1}$ 
its base orbifold is $D^{2}$ with at most one orbifold point.  
We see that the base orbifold for $X$ is a punctured disk with at most one orbifold point.  
If there is no orbifold point, then there is an index $i$ so that for any $i' \neq i$, 
$\alpha|_{T_{i'}}$ intersects the fiber exactly once (possibly, $\alpha|_{T_{i}}$
intersects the fiber once as well).
Define $\alpha' \pf \alpha$ by setting $\alpha'|_{T_{i}} = \nos$
and $\alpha'|_{T_{i'}} = \alpha|_{T_{i'}}$ for all $i' \neq i$.
Then $X(\alpha')$ is a Seifert fiber space over an annulus with no exceptional fibers
and hence $X(\alpha') \cong T^{2} \times [0,1]$, contradicting condition~(2) of the theorem.  
We assume as we may that the base orbifold of $X$ has exactly one orbifold point.
Denote the multiplicity of the critical fiber by $p \geq 2$.  Then the meridian of
$X(\alpha)$ intersects a fiber on $\partial X$ exactly $p$ times; this gives a
bounded set by Proposition~\ref{prop:PropertiesOfBoundedSets}~(4).

We assume from now on that $X$ is prime and not Seifert fibered.

\bigskip
\noindent  {\bf Assume $X$ is a JSJ manifold.}   Let $X_{0}$ be the component of the JSJ decomposition
of $X$ that contains $T$ and denote the components of $\partial X_{0} \setminus T$ as $F_{j}$, $j=1,\dots,{m}$.  We will denote the
closures of the components of $X \setminus X_{0}$ as $X_{j}$.  
To avoid the situation where $X_{j} = \emptyset$, if $F_{j} \subset \partial X$ we push it slightly 
into the interior of $X_{0}$ so that $X_{j} \cong T^{2} \times [0,1]$ in that case.  
Since $D^{2} \times S^{1}$ admits no non-separating tori, we assume as we may 
that $X_{j} \cap X_{j'} = \emptyset$ for $j \neq j'$.
By reordering the indices of $X_{j}$ if necessary we may assume that $F_{j} \subset X_{j}$.
Finally, given $\alpha \in \mathcal{A}$ and $1 \leq j \leq m$, we will denote the components of 
$X(\alpha)$ cut open along $F_{j}$ as $X(\alpha)_{j}^{+}$ and $X(\alpha)_{j}^{-}$, with 
$\partial X(\alpha)_{j}^{+} = T \cup F_{j}$
and $\partial X(\alpha)_{j}^{-} = F_{j}$.   Consider the following subsets 
$\mathcal{A}_{j} \subset \mathcal{A}$ (for $j=0,\dots,m$):
	\begin{enumerate}
	\item $\mathcal{A}_{0}$ consist of all the multislopes $\alpha \in \mathcal{A}$ so that for all $j$, 
	$X(\alpha)_{j}^{+} \not\cong T^{2} \times [0,1]$.
	\item For $1 \leq j \leq m$, $\mathcal{A}_{j}$ consist of all the multislopes $\alpha \in \mathcal{A}$ 
	so that $X(\alpha)_{j}^{+} \cong T^{2} \times [0,1]$.
	\end{enumerate}
It is immediate from the definitions that
$$\mathcal{A} = \cup_{j=0}^{m} \mathcal{A}_{j}.$$

We first consider multislopes $\alpha \in \mathcal{A}_{0}$.  
By definition of $\mathcal{A}_{0}$, no
$F_{j}$ is boundary parallel in $X(\alpha) \cong D^{2} \times S^{1}$; thus 
every torus $F_{j}$ bounds a solid torus
or a non-trivial knot exterior $E(K_{j})$.  For each $F_{j}$ that bounds a solid torus, 
let $\hat\alpha|_{F_{j}}$ be the slope of $F_{j}$ defined by the 
meridian of that solid torus.  For each torus $F_{j}$ that bounds a non trivial
knot exterior $E(K_{j})$ we do the following:
by Lemma~\ref{lem:KnotExteriorsAreDisjoint} we may assume that $E(K_{j}) \subset D_{j}$
for disjointly embedded balls $D_{j}$.
We replace every $E(K_{j})$ with a solid torus (which we will denote as $V_{j}$) so that the meridian of $V_{j}$
intersects that meridian of $E(K_{j})$ exactly once.  
This does no change the ambient manifold and, since the changes
are contained in balls, the slope of $T$ that is the meridian of the solid 
torus $X(\alpha)$ is not changed.  
Let $\hat\alpha|_{F_{j}}$ be the slope of $F_{j}$ defined by the meridian of 
$V_{j}$.  Thus we have defined a slopes $\hat\alpha|_{F_{j}}$ for every $1 \leq j\leq m$;
together they induce a multislope of $\{F_{j}\}_{j=1}^{m}$ which we will denote by $\hat\alpha$.
We claim that $\hat\alpha$ fulfills the following two conditions:
	\begin{enumerate}
	\item $X_{0}(\hat\alpha) \cong D^{2} \times S^{1}$: this is immediate from the construction.
	\item there is no partial filling $\hat\alpha' \pf \hat\alpha$ so that 
	$X_{0}(\hat\alpha') \not\cong T^{2} \times [0,1]$:
	assume, for a contradiction, that such a partial filling $\hat\alpha'$ exists.
	Since $T^{2} \times [0,1]$ has two boundary components, $\hat\alpha'$ is
	obtained from $\hat\alpha$ by setting the value of $\hat\alpha|_{F_{i}}$ to $\nos$
	on exactly one torus $F_{j}$.  By the defining condition for $\mathcal{A}_{0}$
	this is impossible for values of $j$ for which $X(\alpha)_{j}^{-} \cong D^{2} \times S^{1}$, 
	and for other values of $j$ this is impossible because $F_{j}$ is contained in
	the ball $D_{j} \subset X(\alpha)$.
	\end{enumerate}
Thus $\hat\alpha$ satisfies the assumptions of Theorem~\ref{thm:SlopesOnSolidTorus}.  
By construction $X_{0}$ corrsponds to a 
direct descendant of the root of $T(X)$, and therefore $|T(X_{0})| < |T(X)|$.
By induction, the set of meridians of the solid tori $X_{0}(\hat\alpha)$ (as 
$\hat\alpha$ varies over all possible multislope that correspond to multislope 
$\alpha \in \mathcal{A}_{0}$) form a bounded set of slopes of $T$, which we will
denote as $B_{0}$.   By construction,
the set of meridians of $X(\alpha)$ for $\alpha \in \mathcal{A}_{0}$ is $B_{0}$.

Next fix $1 \leq j \leq n$ and consider $\alpha \in \mathcal{A}_{j}$.  By the defining 
condition for $\mathcal{A}_{j}$, 
$X(\alpha)_{j}^{+} \cong T^{2} \times S^{1}$.  Hence 
$X(\alpha)_{j}^{-} \cong X(\alpha) \cong D^{2} \times S^{1}$.
Note that $X(\alpha)_{j}^{-}$ is obtained by filling a component of $X$ cut open along $F_{j}$ 
denoted above as $X_{j}$; the induced filling is given by $\alpha|_{\partial X_{j}}$.  We
claim that the following conditions hold:
	\begin{enumerate}
	\item $X_{j}(\alpha|_{\partial X_{j}}) \cong D^{2} \times S^{1}$:  this is immediate from the construction.
	\item There is no partial filling $\alpha|_{\partial X_{j}}' \pf \alpha|_{\partial X_{j}}$ with
	$X_{j}(\alpha|_{\partial X_{j}}') \cong T^{2} \times [0,1]$: otherwise, there would be a corresponding 
	partial filling $\alpha' \pf \alpha$ so that 
	$$X(\alpha') \cong X(\alpha)_{j}^{+} \cup_{F_{j}} 
	X_{j}(\alpha|_{\partial X_{j}}') \cong T^{2} \times [0,1],$$
	contradicting the assumptions of the theorem.  
	\item $|T(X_{j})| < |T(X)|$: this follows from Lemma~\ref{lem:TreeOfJSJ}.
	\end{enumerate}
By induction, the set of 
slopes of meridians of $X_{j}(\alpha|_{\partial X_{j}})$ is bounded; we will denote it as $B_{j}'$.
By Theorem~\ref{thm:CosmeticSurgeryOnT2XI} 
the set of slopes of $T$ obtained by projecting $B_{j}'$ 
after all possible cosmetic surgeries on $L \cap X(\alpha)_{j}^{+}$ 
is bounded; we will it as $B_{j}$.  Clearly, the set of meridians on $X(\alpha)$ 
(for $\alpha \in \mathcal{A}_{j}$) is contained in $B_{j}$.

We have obtained $m+1$ bounded sets, namely, $B_{0},\dots,B_{m}$, so that the meridians of $X(\alpha)$ 
(for $\alpha \in \mathcal{A}$) are contained in $\cup_{j=0}^{m} B_{j}$.  The theorem follows in this case.

We assume from now on that $X$ is prime, not Seifert fibered, and not a JSJ manifold.

\bigskip
\noindent  {\bf Assume $X$ is hyperbolic.}  By Proposition~\ref{pro:mnhIsFinite}
$X$ admits only finitely many \mnh\ fillings, which we will denote as
$\alpha_{1},\dots,\alpha_{k}$.  Fix $1 \leq j \leq k$.  If there is no $\alpha \in \mathcal{A}$
for which $\alpha_{j} \pf \alpha$, we set $B_{j} = \emptyset$; otherwise,
any $\alpha \in \mathcal{A}$ for which $\alpha_{j} \pf \alpha$ induces the 
multislope $\alpha|_{\partial X(\alpha_{j})}$ on $\partial X(\alpha_{j})$.
We claim that $X(\alpha_{j})$ and $\alpha|_{\partial X(\alpha_{j})}$ fulfill the following four conditions:
\begin{enumerate}
\item $X(\alpha_{j})(\alpha|_{\partial X(\alpha_{j})}) 
\cong D^{2} \times S^{1}$: this is immediate, 
as $X(\alpha_{j})(\alpha|_{\partial X(\alpha_{j})}) = X(\alpha)$ (we emphasize that
this is equality, not up to diffeomorphism).
\item $\partial X(\alpha_{j})(\alpha|_{\partial X(\alpha_{j})}) = T$: 
this is immediate, as above.
\item $\alpha|_{\partial X(\alpha_{j})}$ does not admit a partial filling
 $\alpha' \pf \alpha|_{\partial X(\alpha_{j})}$ so that
$X(\alpha_{j})(\alpha') \cong T^{2} \times [0,1]$:  otherwise the corresponding partial filling
of $\alpha$ would yield $T^{2} \times [0,1]$, violating the second assumption of the theorem.
\item The meridians of the solid tori $X(\alpha_{j})(\alpha|_{\partial X(\alpha_{j})})$ 
and $X(\alpha)$ define the same slope of $T$: this is immediate,
as~(1).
\end{enumerate}
We will denote as $\mathcal{A}_{j} \subset \mathcal{A}$ the set
$$\mathcal{A}_{j} = \{ \alpha|_{X(\alpha_{j})} \ | \ \alpha \in \mathcal{A}, \\ \alpha_{j} \pf \alpha\}.$$ 
By points~(1)-(3) above, $X(\alpha_{j})$ and $\mathcal{A}_{j}$ 
fulfill the assumptions of the theorem.  
By construction $T(X(\alpha_{j}))$ corresponds to a direct descendant of the root of $T(X)$; 
therefore $|T(X(\alpha_{j}))| < |T(X)|$.   By induction, the meridians of
$$\{X(\alpha_{j})(\alpha|_{\partial X(\alpha_{j})}) \ | \ \alpha \in \mathcal{A}_{j}\}$$
form a bounded set of slopes of $T$, 
which we will denote as $B_{j}$.  Since $D^{2} \times S^{1}$ is not hyperbolic, 
for every $\alpha \in \mathcal{A}$, there is $1 \leq j \leq k$, so that 
$\alpha_{j} \pf \alpha$; by point~(4) above the meridian
of $X(\alpha)$ is in $B_{j}$.  
We see that
the meridians of $\{X(\alpha) | \alpha \in \mathcal{A}\}$ are 
$$\cup_{j=1}^{k} B_{j}.$$
The theorem follows, as the finite union of bounded sets is bounded.
\end{proof}

\bigskip\noindent  
Next we prove a proposition about fillings that yield $D^{2} \times S^{1}$; it will be
used in the proof of the ultimate claim in the paper, Proposition~\ref{pro:TheLastCase}.

\begin{prop}
\label{pro:SolidTorusSurgery2}
Let $X$ be a compact orientable connected manifold so that $\partial {X}$ consists of tori.
Denote the components of $\partial X$ by $T,T_{1},\dots,T_{n}$.  
Fix $\mathcal{T}$ a non empty subset of $\{T_{1},\dots,T_{n}\}$.
For a multislope $\alpha$ on $T_{1},\dots,T_{n}$, we will denote
the link formed by the cores of the solid tori attached to $\mathcal{T}$ as $\mathcal{L}$.  
Let $\mathcal{A}$ be a set of multislopes of $\partial X$ so that every $\alpha \in \mathcal{A}$
satisfies the following conditions:
	\begin{enumerate}
	\item $X(\alpha) \cong D^{2} \times S^{1}$.
	\item $X(\alpha)  \setminus \mbox{int}N(\mathcal{L})$ is irreducible.
	\item No component of $\mathcal L$ is a core of $X(\alpha)$.
	\item $\alpha|_{T} = \nos$.
	\end{enumerate}
	
Then for each $T \in \mathcal{T}$, there exists a bounded set $B_{T}$ of 
the slopes of $T$, so that for any $\alpha \in \mathcal A$ there exists $T \in \mathcal{T}$ so that
$$\alpha|_{T} \in B_{T}.$$
\end{prop}

\begin{proof}
We will induct on $|T(X)|$.  Parts of the proof are similar to the proof of Theorem~\ref{thm:SlopesOnSolidTorus}
and we will only sketch them here.

\bigskip\noindent
{\bf Assume that $X$ is hyperbolic.}  Since $X(\alpha) \cong D^{2} \times S^{1}$ is not hyperbolic, any
$\alpha \in \mathcal A$ factors through a \mnh\ filling.  Let $\alpha' \pf \alpha$ be a \mnh\ filling.
If, for some $T \in \mathcal{T}$, $\alpha'|_{T} \neq \nos$, we add $\alpha'|_{T}$ to $B_{T}$.
We assume as we may that  for all $T \in \mathcal{T}$, $\alpha'|_{T} = \nos$.  Then 
$\mathcal T \subset \partial X(\alpha')$.  We see that the filling $\alpha|_{\partial X(\alpha')}$ 
induced by $\alpha$ on $X(\alpha')$ fulfills conditions (1)$\sim$(4) of the proposition; since 
$|T(X(\alpha'))| < |T(X)|$, the proposition follows from the inductive hypothesis.

We assume from now on that $X$ is not hyperbolic.

\bigskip\noindent
{\bf Assume that $X$ is not prime.}  Let $S$ be a sphere  embedded in $X$ that realizes the 
decomposition $X' \# X''$, where $X'$ is prime and 
$T \subset X'$ (we are not assuming that $X''$ is prime).  
By condition~(2) of the proposition,  $\mathcal T \subset \partial X'$.  Let $\mathcal{A}'$
be the restrictions defined by
$$ \mathcal{A}' = \{ \alpha|_{\partial X'} | \alpha \in \mathcal{A}\}.$$  
It is easy to see that conditions (1)$\sim$(4) of the proposition hold, and since $|T(X')| < |T(X)|$, 
the proposition follows from the inductive hypothesis.

We assume from now on that $X$ is prime and not hyperbolic.

\bigskip\noindent
{\bf Assume that $X$ is JSJ.}  We first fix the notation that will be used in this case.
Let $X_{0}$ be the component of the torus decomposition of $X$ that contains
$T$.  We will denote the components of $\partial X_{0}$ as $T$, $F_{1},\dots,F_{k}$ and the components of 
$\mbox{cl}(X \setminus X_{0})$ as $X_{1},\dots,X_{k}$, numbered so that $F_{j} \subset \partial X_{j}$.  
To avoid the situation $X_{j} = \emptyset$, if $F_{j} \subset \partial X$ we push it slightly into the 
interior so that $X_{j} \cong T^{2} \times [0,1]$ in this case. 
Since $D^{2} \times S^{1}$ contains no non-separating tori, 
we assume as we may that $X_{j} \cap X_{j'} = \emptyset$ for $j \neq j'$.  We will denote  as 
$\mathcal{F} \subset \{F_{1},\dots,F_{k}\}$ the components $F_{j}$ that bound $X_{j}$ for which 
$\mathcal{T} \cap \partial X_{j} \neq \emptyset$.
Given $\alpha \in \mathcal{A}$, we will denote $\alpha|_{\partial X_{j}}$ by $\alpha_{j}$.
(Note that by definition of $\mathcal{F}$, $F_{j} \in \mathcal{F}$ if and only if 
$\mathcal{L} \cap X_{j}(\alpha_{j}) \neq \emptyset$.)
Clearly, $\partial X_{j}(\alpha_{j}) = F_{j}$ and $X_{j}(\alpha_{j})$ is either a solid torus or 
the exterior of a non trivial knot which we will denote as $E(K_{j})$.  
Up-to finite ambiguity, we fix $\mathcal{F}_{\mbox{st}}$ and $\mathcal{F}_{\mbox{k}}$ 
so that $\mathcal{F} = \mathcal{F}_{\mbox{st}} \sqcup \mathcal{F}_{\mbox{k}}$ 
and consider the multislopes $\alpha \in \mathcal{A}$ for which 
$X_{j}(\alpha_{j}) \cong D^2 \times S^{1}$ whenever $F_{j} \in \mathcal{F}_{\mbox{st}}$ and
$X_{j}(\alpha_{j}) \cong E(K_{j})$ whenever $F_{j} \in \mathcal{F}_{\mbox{k}}$.
To avoid overly complicated notation we do not rename $\mathcal{A}$.
There are two cases to consider:

\medskip\noindent
{\bf Case One: some $X_{j}(\alpha_{j})$ has no core.}
Let $\mathcal{A}_{1} \subset \mathcal{A}$ be defined by requiring that 
for some $F_j \in \mathcal{F}_{\mbox{st}}$, no component of $\mathcal{L} \cap X_{j}(\alpha_{j})$ 
is a core of the solid torus $X_{j}(\alpha_{j})$.
The second assumption of the proposition implies that $\mathcal{L} \cap X_{j}(\alpha_{j})$ 
is irreducible.  By Lemma~\ref{lem:TreeOfJSJ}, $|T(X_{j})| < |T(X)|$.
Applying the inductive hypothesis to $X_{j}$ we see that 
for each $T \in \mathcal{T} \cap \partial X_{j}$,
there is a bounded set $B_{T}$, so that $\alpha|_{T} \in B_{T}$ for some such $T$.  
The proposition follows for $\mathcal{A}_{1}$.

\medskip\noindent
{\bf Case Two: every $X_{j}(\alpha_{j})$ has core.}  
Let $\mathcal{A}_{2} = \mathcal{A} \setminus \mathcal{A}_{1}$.
Then for every $F_j \in \mathcal{F}_{\mbox{st}}$ and every $\alpha \in \mathcal{A}_{2}$,
the core of the solid torus attached to one of the components of $\mathcal{T} \cap \partial X_{j}$
is a core of $X_{j}(\alpha_{j})$;  we will denote it as $L_{j}$ (there may be more than
one such component; we pick one).  
By Lemma~\ref{lem:KnotExteriorsAreDisjoint},
 for every $j$ for which $F_{j} \in \mathcal{F}_{\mbox{k}}$,
there exists an embedded ball $D_{j} \subset X_{0}(\alpha_{0})$, so that $X_{j}(\alpha_{j}) \subset D_{j}$
and $D_{j} \cap D_{j'} = \emptyset$ for $j \neq j'$.  
Thus the second assumption of the proposition implies that $\mathcal{F}_{\mbox{st}} \neq \emptyset$.
Any $\alpha \in \mathcal{A}_{2}$ induces a multislope on $\partial X_{0} = T,F_1,\dots,F_{k}$,
which we will denote as $\alpha_{0}$, that consists of following slopes:
\begin{enumerate}
\item the meridian of $X_{j}(\alpha_{j})$ (on components $F_{j}$ 
that bound $X_{j}(\alpha_{j}) \cong D^{2} \times S^{1}$),
\item a slope that intersects the meridian of $X_{j}(\alpha_{j})$ exactly once 
(on components $F_{j}$ that bound 
$X_{j}(\alpha_{j}) \cong E(K_{j})$, a non trivial knot exterior),
\item $\nos$ (on $T$).
\end{enumerate}
For $F_j \in \mathcal{F}_{\mbox{k}}$, the core of the solid tori attached to $F_{j}$ is an unknot
in $D_{j}$; thus the cores of the solid tori attached to $\cup_{F_j \in \mathcal{F}_{\mbox{\tiny k}}} F_{j}$
form a (possibly empty) unlink, which we will denote as $\mathcal{U}$.  
The cores of the solid tori attached to $\cup_{F_j \in \mathcal{F}_{\mbox{\tiny st}}} F_{j}$
form a (non empty) link which we will denote as $\mathcal{L}_{0}$.

We claim that $\mathcal{L}_{0}$ is 
irreducible in the complement of $\mathcal{U}$.  
Assume, for a contradiction, that $X(\alpha) \setminus \mbox{int}N(\mathcal{L}_{0} \cup \mathcal{U})$ 
is reducible and let $S$ be a reducing sphere.  Since $S$ is disjoint form the cores
of the solid tori attached to $F_{j}$ (for every $F_{j} \in \mathcal{F}$)
we may isotope $S$ out of them.  It is now easy to see that $S \subset X(\alpha)$
and is disjoint from $\mathcal{L}$.  Since $S$ is a reducing sphere for 
$\mathcal{L}_{0} \cup \mathcal{U}$, there are components of
$\mathcal{L}_{0} \cup \mathcal{U}$ in the ball bound by $S$ in $X_{0}(\alpha_{0})$.  Clearly,
there are components of $\mathcal{L}$ in the ball bound by $S$ in $X(\alpha)$.  Thus $S$ is a reducing
sphere for $\mathcal{L}$, contradicting the second assumption of the proposition.

We assume as we may by Lemma~\ref{lem:TwistingToGetIrreducible} that 
the slopes on $\mathcal{U}$ were chosen so that $\mathcal{L}_0$ is irreducible.  
We claim that no component of $\mathcal{L}_{0}$ is the core of
$X_{0}(\alpha_{0})$.  Assume for a contradiction
that this is not the case and fix  $F_j \in \mathcal{F}_{\mbox{st}}$
 for which the core of the solid torus attached to $F_{j}$
is a core of $X_{0}(\alpha_{0})$.  
Recall that $L_{j}$ is a core of a solid torus attached to a component of
$\mathcal{T} \cap \partial X_{j}$, and is the core of $X_{j}(\alpha_{j})$.  Thus
$L_{j}$ is a component of $\mathcal{L}$, and is the core of $X(\alpha)$;
this is impossible as it violates the third assumption of the proposition.

Thus $X_{0}$, $\mathcal{L}_{0}$, and the multislopes induced by $\mathcal{A}_{2}$ satisfy the assumptions of the proposition.
By Lemma~\ref{lem:TreeOfJSJ}, $|T(X_{0})| < T(X)|$.  Hence by induction,
for every $F_{j} \in \mathcal{F}_{\mbox{st}}$, there is a bounded set of slopes of $F_{j}$,
which we will denote as $B_{F_{j}}$, so that for every multislope $\alpha_{0}$ (induced by some 
$\alpha \in \mathcal{A}_{2}$), there is $F_j \in \mathcal{F}_{\mbox{st}}$, 
for which $\alpha_{0}|_{F_{j}} \in B_{F_{j}}$.  The slope $\alpha|_{\partial N(L_{j})}$ 
is the projection of $\alpha_{0}|_{F_{j}}$ by the product structure
on $X_{j}(\alpha_{j}) \setminus N(L_{j}) \cong T^{2} \times [0,1]$. 
By the $T^{2} \times [0,1]$ Cosmetic Surgery 
Theorem~(\ref{thm:CosmeticSurgeryOnT2XI}) the projections of $B_{F_{j}}$
under all possible fillings of $X_{j}$ that yield $T^{2} \times [0,1]$
is a bounded set of slopes of $\partial N(L_{j})$, which we will denote as $B_{j}$.
Thus, for every $\alpha \in \mathcal{A}_{2}$, there exists 
$F_j \in \mathcal{F}_{\mbox{st}}$, for which $\alpha|_{\partial N(L_{j})} \in B_{j}$, proving the proposition
in case two.

We assume from now on that $X$ is irreducible and not hyperbolic, and not JSJ.

\bigskip\noindent
{\bf Assume that $X$ is a Seifert fibered space.}
Fix a Seifert fiberation on $X$.  We consider three cases, depending
on the intersection of $\alpha|_{T_{i}}$ with slopes defined by the 
Seifert fiber on $T_{i}$:
\begin{enumerate}
\item If (for some $T_{i} \in \mathcal{T}$) $\alpha|_{T_{i}}$ is the fiber in the 
Seifert fiberation then the intersection number 
of $\alpha|_{T_{i}}$ with the fiber is zero.
\item If (for some $T_{i'} \not \in \mathcal{T}$)  $\alpha|_{T_{i'}}$ is the fiber in the 
Seifert fiberation, then for every $T_{i} \in \mathcal{T}$ the disk obtained by gluing 
a vertical annulus connecting $T_{i'}$ to $T_{i}$ with
a meridian disk of the solid torus attached to $T_{i'}$
is a compressing disk for $T_{i}$ and its boundary is a regular fiber.  Since 
$X(\alpha) \cong D^{2} \times S^{1}$ contains no non separating spheres
or lens space summands, $\alpha|_{T_{i}}$ intersects the regular fiber exactly once.
\item If $\alpha_{T_{i}}$ is not the fiber for any $1 \leq i \leq n$, 
then the fiberation on $X$ extends to a fiberation of $X(\alpha)$, 
which is a fiberation over
$D^{2}$ with at most one exceptional fiber.  The exceptional fiber, if exists, is the core of $X(\alpha)$.  
Thus by assumption~(3) of the proposition every component of $\mathcal L$ is a regular fiber, implying that
for every $T_{i} \in \mathcal{T}$, $\alpha|_{T_{i}}$
intersects the fiber in the Seifert fiberation of $X$ exactly once.

\end{enumerate}

We conclude that for every $\alpha \in \mathcal{A}$ there exists $T \in \mathcal{T}$ so that
$\alpha|_{T}$ intersects the fiber at most once.
The proposition follows from Proposition~\ref{prop:PropertiesOfBoundedSets}~(4).

\bigskip\noindent
This completes the proof of Proposition~\ref{pro:SolidTorusSurgery2}.
\end{proof}

\section{Hyperbolic cosmetic surgery: slopes}
\label{section:HyperSurgSlopes}

\bigskip\noindent
In this section we ask ``how much can a hyperbolic manifold get twisted by performing
cosmetic surgery''?   (Recall that in this paper by {\it hyperbolic manifold} we mean a
connected compact manifold whose interior
admits a complete finite volume hyperbolic metric.)
Consider the following problem: let $M$ be a hyperbolic manifold so that $\partial M$
is a single torus, $B$ a bounded set of slopes of $\partial M$, and  $L \subset M$ a link.  If $\alpha$
is a multislope of cosmetic surgery (that is, $L(\alpha) \cong M$), then an identification
of $L(\alpha)$ with $M$ induces a bijection on the slopes of $\partial M$.  Our goal is to show that the
union of the images of $B$ under all such bijections is bounded.  The theorem below
is stated in terms of fillings (with $X$ corresponding to $M \setminus N(L)$) and is
slightly more general as it allows for more boundary components.

\begin{thm}
\label{thm:CosmeticSurgeryOnM}
Let $M$ be an orientable hyperbolic manifold, $T_{M}$ a component of $\partial M$,
$X$ a compact orientable connected manifold so that $\partial X$ consists of tori that we will denote
as $T,T_{1},\dots,T_{n}$, and $B$ a bounded set of slopes of $T$.  
Let $\mathcal{X} = \{(\alpha,f_{\alpha})\}$ be a set of pairs
so that every $(\alpha,f_{\alpha}) \in \mathcal{X}$ satisfies the following conditions:
	\begin{enumerate}
	\item $\alpha$ is a multislope of $X$.
	\item $f_{\alpha}:X(\alpha) \to M$ is a diffeomorphism.
	\item $f_{\alpha}$ maps $T$ to $T_{M}$. 
	\end{enumerate}
For every $x = (\alpha,f_{\alpha}) \in \mathcal{X}$, we will denote the image of $B$ under
the bijection induced by $f_{\alpha}$ from the slopes of $T$ to those of $T_{M}$ as $B_{x}$.

Then $\cup_{x \in \mathcal{X}} B_{x}$ is a bounded set of slopes of $T_{M}$.
\end{thm}

\begin{proof}
We induct on $|T(X)|$.   

\bigskip\noindent
{\bf Assume that $X$ is Seifert fibered or sol.} Then $X$ admits no hyperbolic filling.

We assume from now on that $X$ is not Seifert fibered or sol.

\bigskip\noindent
{\bf Assume that $X$ is not prime.}  Let $S \subset X$ be a decomposing sphere that realizes
the decomposition $X = X' \cup_{S} X''$, where here $X'$ is a prime manifold containing $T$ (we do not
assume that $X''$ is prime).  Then for every $(\alpha,f_{\alpha}) \in \mathcal{A}$
we have
$$X(\alpha) = X'(\alpha|_{\partial X'}) \cup_{S} X''(\alpha|_{\partial X''}).$$
Since $(\alpha) \cong M$ is hyperbolic, $S \subset X(\alpha)$ bounds a ball.
Condition~(3) of the theorem implies that $\alpha|_{T} = \nos$; hence $X''(\alpha|_{\partial X''})$
is a ball.  Thus $f_{\alpha}$ induces a diffeomorphism that we will denote as
$$f_{\alpha|_{\partial X'}}:X'(\alpha|_{\partial X'}) \to M.$$  
We will denote the set of pairs $\{(\alpha|_{\partial X'},f_{\alpha|_{\partial X'}})\}$
induced by pairs $(\alpha,f_{\alpha}) \in \mathcal{X}$ as $\mathcal{X}'$.
For  $(\alpha|_{\partial X'},f_{\alpha|_{\partial X'}}) = x' \in \mathcal{X}'$,
we will denote the image of $B$ under the bijection induced by $f_{\alpha|_{\partial X'}}$
from the slopes of $T$ to those of $T_{M}$ as $B_{x'}$.  
Since $X'$ corresponds to a direct descendant of the root of $T(X)$,
$|T(X')| < |T(X)|$.  By induction
$$\cup_{x' \in \mathcal{X}'} B_{x'}$$
is bounded set of slopes of $T_{M}$.  By construction, every $x' \in \mathcal{X}'$ is induced by 
$x \in \mathcal{X}$ for which $f_{\alpha|_{\partial X'}}|_{T} = f_{\alpha}|_{T}$.  Thus $B_{x'} = B_{x}$
and therefore 
$$\cup_{x \in \mathcal{X}} B_{x} = \cup_{x' \in \mathcal{X}'} B_{x'}$$ 
is a bounded set of slopes of $T_{M}$.  The theorem follows in this case.

We assume from now on that $X$ is prime and not Seifert fibered or sol.

\bigskip\noindent
{\bf Assume that $X$ is JSJ.}  By Lemma~\ref{prop:ObtainedByFilling}, 
each $\alpha \in \mathcal{A}$ induces a filling on one of the  
components of the torus decomposition of $X$ that yields $M$.  Up-to finite ambiguity we
fix one component of the torus decomposition of $X$, that we will denote as $X_{0}$,
and consider only multislopes $\alpha$ that induce a filling on $X_{0}$.
We will use the following notation:
\begin{enumerate}
\item The multislope induced by $\alpha$ on $\partial X_{0}$ will be denoted as $\alpha_{0}$. 
\item  The closure of the component of $X \setminus X_{0}$ that contains $T$ will be denoted as $X_{1}$
(to avoid the situation $X_{1} = \emptyset$, if $T \subset \partial X_{0}$ we push it slightly into the interior 
so that $X_{1} \cong T^{2} \times [0,1]$ in this case).  
\item The torus $X_{0} \cap X_{1}$ will be denoted as $F$.  
\end{enumerate}
Let $F_{M}$ be a boundary parallel torus in $M$; it follows from the construction in
Lemma~\ref{prop:ObtainedByFilling} that after isotopy of $f_{\alpha}$ if necessary
we may assume that $f_{\alpha}(F) = F_{M}$.  We will denote the components of
$M$ cut open along $F_{M}$ as $M_{0}$ and $M_{1}$, where $M_{0}$ is 
the component not containing $T_{M}$.
By the construction in Lemma~\ref{prop:ObtainedByFilling},  $f_{\alpha}$ induces a diffeomorphism
$f_{\alpha_{0}}:X_{0}(\alpha_{0}) \to M_{0}$ that maps $F$ to $F_{M}$.
By Lemma~\ref{lem:TreeOfJSJ}, $|T(X_{0})| < |T(X)|$.  Therefore by induction
the union of the images of $B$ under the bijections induced by $f_{\alpha_{0}}$
is a bounded set of slopes of  $F_{M}$ that we will denote as $B_{F_{M}}$.

Every slope of  $\cup_{x \in \mathcal{X}} B_{x}$ is obtained from a slope in $B_{F_{M}}$
by projecting using the product structure on $M_{1} \cong T^{2} \times [0,1]$.  Since
$M_{1}$ is obtained from $X_{1}$ by filling, by the $T^{2} \times [0,1]$ Cosmetic
Surgery Theorem~(\ref{thm:CosmeticSurgeryOnT2XI}),  $\cup_{x \in \mathcal{X}} B_{x}$
is a bounded set of slopes of $T_{M}$. 
This completes the proof for JSJ manifolds.

We assume from now on that $X$ is prime and not Seifert fibered, sol, or JSJ.

\bigskip\noindent
{\bf Assume $X$ is hyperbolic.}  
Let $\mathcal{X}^{1} \subset \mathcal{X}$ be all pairs $(\alpha,f_{\alpha})$ for which
$\alpha$ is totally hyperbolic.  Fix $\alpha$ a totally
hyperbolic filling and let $(\alpha,f_{j})$ be all the elements of $\mathcal{X}^{1}$
that have $\alpha$ as their multislope, and different diffeomorphisms
$f_{j}:X(\alpha) \to M$  (for $j\in J$ for some index set $J$).  We will denote
the image of $B$ under the bijection induced by $f_{j}$ between the slopes of $T$
and those of $T_{M}$ as $B_{j}$.
Then $f_{j} \circ f_{1}^{-1}:M \to M$ is a diffeomorphism that sends $T_{M}$ to itself
and $f_{j} = (f_{j} \circ f_{1}^{-1}) \circ f_{1}$.   Thus $B_{j} = \phi_{j}(B_{1})$,
where $\phi_{j}$ is a bijection induce by an element of the mapping class group of $M$.
It is straightforward to see that the bijection induced by $\phi_{j}$ is an isometry
of the slopes; hence $\phi_{j}(B_{1})$ is bounded. 
Since the mapping class group of hyperbolic manifolds is finite, we see that 
$\cup_{j \in J}\phi_{j}(B_{1})$ is bounded as well.
By Proposition~\ref{prop:UsingT(X)forHyperbolicFilling}, $X$ admits 
only finitely many totally hyperbolic fillings, and hence 
$$\cup_{x \in \mathcal{X}^{1}} B_{x}$$
is a bounded set of slopes.

Next we consider $\mathcal{X}^{2} = \mathcal{X} \setminus \mathcal{X}^{1}$.
By Proposition~\ref{pro:mnhIsFinite}, $X$ admits 
only finitely many \mnh\ fillings.   We will denote as $\alpha_{1},\dots,\alpha_{k}$
the \mnh\ fillings of $X$ for which $\alpha_{j}|_{T} = \nos$.  For  
$1 \leq j \leq k$, let $\mathcal{X}_{j}$ be the set of all pairs 
pairs $(\alpha,f_{\alpha})$ satisfying the following conditions:
	\begin{enumerate}
	\item $\alpha$ is a multislope of $X(\alpha_{j})$.
	\item $f_{\alpha}:X(\alpha_{j})(\alpha) \to M$ is a diffeomorphism.
	\item $f_{\alpha}$ maps $T$ to $T_{M}$. 
	\end{enumerate}
For $x_{j} =(\alpha,f_{\alpha}) \in \mathcal{X}_{j}$, we will denote the image of $B$ 
under the bijection induced by
$f_{\alpha}$ between the slopes of $T$ and those of $T_{M}$ as $B_{x_{j}}$. 
Since $X(\alpha_{j})$ corresponds to a direct descendant of the root of $T(X)$, 
$|T(X(\alpha_{j}))| <|T(X)|$.  Thus by induction
$$\cup_{x_{j} \in \mathcal{X}_{j}} B_{x_{j}}$$
is a bounded set of slopes of $T_{M}$.

By definition of $\mathcal{X}^{2}$, for every $x = (\alpha,f_{\alpha}) \in \mathcal{X}^{2}$, $\alpha$ is not \mnh.
Hence $\alpha$ factors through some \mnh\ filling $\alpha_{j}$ (for some $1 \leq j \leq k$), 
that is,
$$X(\alpha) = X(\alpha_{j})(\alpha|_{\partial X(\alpha_{j})}).$$
We can view $f_{\alpha}:X(\alpha) \to M$ as a diffeomorphism 
$f_{\alpha}:X(\alpha_{j})(\alpha|_{\partial X(\alpha_{j}}) \to M$;
thus we obtain $(\alpha|_{\partial X(\alpha_{j})}, f_{\alpha})  \in \mathcal{X}_{j}$
for which $B_{(\alpha|_{\partial X(\alpha_{j})}, f_{\alpha})} = B_{x}$.
This shows that
$$\cup_{x \in \mathcal{X}^{2}}B_{x} \subset \cup_{j=1}^{k}(\cup_{x_{j} \in \mathcal{X}_{j}}B_{x_{j}}).$$
Thus $\cup_{x \in \mathcal{X}^{2}}B_{x}$ is contained in a finite union of 
bounded sets, and hence is itself bounded.

The theorem follows in this final case.

\bigskip\noindent
This completes the proof of Theorem~\ref{thm:CosmeticSurgeryOnM}.
\end{proof}

\section{Hyperbolic cosmetic surgery: radius of injectivity}
\label{section:HyperSurgRadInj}

\bigskip\noindent
Let $X$ be a hyperbolic manifold.  A generic filling on $\mathcal{T} \subset \partial X$, with {\it all} the slopes very
long, yields a hyperbolic manifold with at least $|\mathcal{T}|$ short geodesic.  However, if $|\mathcal{T}|>1$,
this requires excluding infinitely many multislopes.  It is easy to construct examples where
$X$ can be filled to give infinitely many manifolds that violate this rule, for example,
every lens space is obtained by filling the Whitehead link exterior.  As another example, 
given any hyperbolic manifold $M$, let $K \subset M \# T^{2} \times [0,1]$ be a simple
knot.  Then the exterior of $K$ is a hyperbolic manifold
that admits infinitely many distinct multislopes $\alpha^{j}$ so that
$X(\alpha^{j}) \cong M$  for every $j$ without the expected three short geodesics.
In this section we show that although the set of multislopes yielding manifolds without
a short geodesic may be infinite, only finitely many manifolds can be obtained.  

\bigskip\noindent
\begin{thm}
\label{thm:HyperbolicFillingShortGeos}
Let $X$ be a compact connected oriented manifold so that $\partial X$ consists of tori.  Fix $\epsilon>0$.  Then 
all but finitely many hyperbolic manifolds that are obtained by filling  $X$ admit a geodesic of length 
less than $\epsilon$.

If in addition $X$ is hyperbolic, then there are only finitely many totally hyperbolic fillings $\alpha$ on $X$ so 
that $X(\alpha)$ does not admit a geodesic of length less than $\epsilon$.
\end{thm}

\begin{proof}
We induct on $|T(X)|$.

\bigskip\noindent
{\bf Assume that $X$ is Seifert fibered or sol.}  Then no filling of $X$ yields a hyperbolic manifold.

We assume from now on that $X$ is not Seifert fibered or sol.

\bigskip\noindent
{\bf Assume that $X$ is reducible.}  Let $X_{1},\dots,X_{n}$ be the factors of the
prime decomposition of $X$.  Then in any filling of $X$ that gives a hyperbolic manifold (say $M$),
exactly one $X_{i}$ fills to give $M$, and every other $X_{i'}$ fills to a ball.  Thus every hyperbolic 
manifold obtained by filling $X$ is obtained by filling $X_{i}$ for some $i$.
Up-to finite ambiguity we fix a factor $X_{i}$.  Since $X_{i}$ corresponds to a direct descendant
of $X$, $|T(X_{i})| < |T(X)|$.  By induction there are only finitely
many hyperbolic manifolds obtained by filling  $X_{i}$ that do not admit a geodesic of 
length less than $\epsilon$.  The proposition follows in this case.

We assume from now on that $X$ is prime and not Seifert fibered or a solv manifold.

\bigskip\noindent
{\bf Assume that $X$ is JSJ.}  Let $X_{1},\dots,X_{n}$ be the components of the
prime decomposition of $X$. By Proposition~\ref{prop:ObtainedByFilling} any hyperbolic
manifold that is obtained from $X$ by filling is obtained by filling some $X_{i}$.
Up-to finite ambiguity we fix a component of the torus decomposition of $X$ which
we will denote as $X_{i}$.  By Lemma~\ref{lem:TreeOfJSJ}, $|T(X_{i})| < |T(X)|$.  By induction there are only finitely
many hyperbolic manifolds obtained by filling  $X_{i}$ that do not admit a geodesic of 
length less than $\epsilon$.  The proposition follows in this case.

We assume from now on that $X$ is prime, not Seifert fibered or sol, and not 
JSJ.

\bigskip\noindent
{\bf Assume that $X$ is hyperbolic.}  Let $\mathcal{A}$ be an infinite set of multislopes
of $\partial X$ so that $X(\alpha)$ is hyperbolic and does not contain
a geodesic shorter than $\epsilon$ for every $\alpha \in \mathcal{A}$.  
(Note that if no such set exists there is nothing to prove.)
We will first establish conclusion~(2) of the theorem by showing that some multislope
$\alpha \in \mathcal{A}$ is not totally hyperbolic.
We will denote the components of $\partial X$ as $T_{1},\dots,T_{n}$.
After subsequencing and reordering if necessary we assume as we may that for some $0 \leq k \leq n+1$ we have:
\begin{enumerate}
\item For every $1 \leq i \leq k$ and every $j \neq j'$, $\ \alpha_{i}^{j} \neq \alpha_{i}^{j'}$ 
and $\alpha_{i}^{j}\neq\nos$.
\item For every $k+1 \leq i \leq n+1$ and every $j, j'$, $\ \alpha_{i}^{j} = \alpha_{i}^{j'}$.
\end{enumerate}

To avoid overly complicated notation we do not rename $\mathcal{A}$.  
Let $\alpha_{0}$ be the restriction $\alpha|_{T_{k+1},\dots,T_{n}}$
for some $\alpha \in \mathcal{A}$ (by construction $\alpha_{0}$ is independent of choice),
$\widehat X = X(\alpha_{0})$ (so $\partial \widehat{X} = T_{1},\dots,T_{k}$),
and $\widehat{\mathcal{A}} = \{ \alpha|_{\partial \widehat{X}} \ | \ \alpha \in \mathcal{A}\}$. 
We claim that $\widehat{X}$ is not hyperbolic; assume for a contradiction that it is.
By truncating the cusps of $\widehat X$ we obtain a Euclidean metric on every $T_{i}$ ($1 \leq i \leq k$).
Since $\widehat{\mathcal{A}}$ is infinite and the values
$\{\widehat\alpha|_{T_{i}}\ | \ \widehat\alpha \in {\widehat{\mathcal A}}\ \}$ are distinct, for any $l$ there is
a multislope $\widehat\alpha \in\widehat{\mathcal{A}}$ so that $\widehat\alpha|_{T_{i}}$
is longer than $l$ for all $i$.  By Thurston's Dehn
Surgery Theorem, for large enough $l$, $\widehat{X}(\widehat\alpha)$ is hyperbolic and
the cores of the attached solid tori are geodesics of length less
than $\epsilon$, contradicting our assumptions.  Thus $\widehat X$ is not 
hyperbolic.  Since $\mathcal{A}$ is infinite, $k\geq 1$.  By condition~(1) above,
for every $\alpha \in \mathcal{A}$, $\alpha_{0}$ is a strict partial filling of $\alpha$.
This shows that $\alpha$ is not totally hyperbolic, establishing the second conclusion 
of the theorem.

Let $\mathcal{A}$ be the set of all multislopes
of $\partial X$ so that $X(\alpha)$ is hyperbolic and does not contain
a geodesic shorter than $\epsilon$.  
We will denote the set of totally hyperbolic fillings in $\mathcal{A}$ as 
$\mathcal{A}_{-}$ 
and $\mathcal{A} \setminus \mathcal{A}_{-}$ as $\mathcal{A}_{+}$.
Every $\alpha \in \mathcal{A}_{+}$ admits
a \mnh\ partial filling, and the \mnh\ fillings of $X$ correspond to the direct descendants of the root
of $T(X)$; up-to finite ambiguity we fix a direct descendant of  the root of $X$ that
we will denote as $X_{i}$.  Then $|T(X_{i})| < |T(X)|$.  By induction there are only finitely
many hyperbolic manifolds obtained by filling  $X_{i}$ that do not admit a geodesic of 
length less than $\epsilon$.  The theorem follows from this and finiteness of $\mathcal{A}_{-}$
that was established above.

This completes the proof of Theorem~\ref{thm:HyperbolicFillingShortGeos}.
\end{proof}

\section{Cosmetic surgery on $S^{3}$}
\label{section:CosmeticSurgeryOnS3}

\bigskip\noindent
We now to turn to one of the more interesting applications of $T(X)$, concerning
cosmetic surgery on $S^{3}$.  Recall that a {\it cosmetic surgery} is $S^{3}$ 
is a surgery on a link $L \subset S^{3}$ with multislope $\alpha$
so that $L(\alpha) \cong S^{3}$.  Note that following examples:
	\begin{enumerate}
	\item Let $L = K_{1} \cup K_{2}$ be the Whitehead link.  Then infinitely many slopes on $K_{1}$ can be completed to
	a cosmetic surgery, namely: $1/m$ can be completed to the cosmetic surgery given by $(1/m,1/0)$,
	where here and in the examples below we are using the standard meridian--longitude.  That
	is not a real problem: $\{1/m\}$ is a bounded set.
	\item Worse is the Hopf link $H = K_{1} \cup K_{2}$.  It is easy to see that $H(p/q,r/s) \cong S^{3}$
	if and only if $ps-rq=\pm1$.
	Thus {\it every} slope on $K_{1}$ can be completed to a cosmetic surgery. 
	\item Kawauchi~\cite{kawauchi} constructed a two component link $L = K_{1} \cup K_{2} \subset S^{3}$ that admits a non-trivial
	cosmetic surgery.  By Teragaito \cite{teragaito} we may assume that $L$, $K_{1}$ and 
	$K_{2}$ are all hyperbolic (this was also announced by Kawauchi~\cite{kawauchi2}).
	For a detailed discussion see the introduction to~\cite{teragaito}.
	Let $L' \subset S^{3}$ be the core of the attached solid tori
	after this surgery, and let $H' \subset S^{3}$ be the Hopf link.  By isotopy of $H'$ (where we allow $H'$ to intersect $L'$)
	we place $H'$ in a ``very complicated'' position relative to $L'$.   There is a surgery on $L'$ which ``undoes'' the surgery
	gets back to $S^{3}$.  Denote the image of $H'$ under this surgery by $H = K_{3} \cup K_{4}$.  
	For any slope $\alpha_{3}$ there exists infinitely many slopes $\alpha_{4}$ so that 
	$L \cup H(\alpha_{1},\alpha_{2},\alpha_{3},\alpha_{4}) \cong S^{3}$.
	However, we expect that $H$ is no longer the Hopf link; in fact, it is quite likely that the components of $H$ 
	are no longer unknotted, as the disks bound by the components of $H'$ are likely to be destroyed by the surgery on $L'$, 
	and new disks are unlikely to appear.  
	We do not prove these claims, but in light of this discussion we expect the following to be true:
	there exists a four component link in $S^{3}$ (such as $L \cup H = K_{1} \cup K_{2} \cup K_{3} \cup K_{4}$ 
	above) that contains no Hopf sublink (perhaps even no 
	unknotted components), yet every slope on $K_{3}$ can be completed to a cosmetic surgery.
	The moral is this: it is our aim to prove that {\it not} any slope can be competed to a cosmetic surgery,
	but one must beware to Hopf links, including those that are invisible in the original link but manifest
	themselves after surgery.  
	\end{enumerate}

\bigskip\noindent
We are now ready to state:	

\begin{thm}
\label{thm:cosmeticSurgeryOnS3}
Let $L \subset S^{3}$ be a link and denote its components by $K_{1},\dots,K_{n}$.  Let $\mathcal{A}$ be a set
of multislopes of $L$ so that every $\alpha \in \mathcal{A}$ fulfills the following two conditions:
	\begin{enumerate}
	\item $L(\alpha) \cong S^{3}$.
	\item For every $\alpha' \pf \alpha$ with $\alpha'|_{T_{1}} = \nos$, $L(\alpha')  \not\cong T^{2} \times [0,1]$.
	\end{enumerate}
Then the restrictions $\mathcal{A}_{1} = \{\alpha|_{T_{1}} \ |\ \alpha \in \mathcal{A} \}$ form a bounded set.
\end{thm}

\begin{rmks}
\begin{enumerate}
\item There exist $\alpha' \pf \alpha$ with $\alpha'|_{T_{1}} = \nos$ and $L(\alpha')  \cong T^{2} \times [0,1]$
if and only if the cores of the solid tori attached to $\partial N(K_{1})$ and 
$\partial N(K_{i})$ form a Hopf link (for aome $2 \leq i \leq n$).
The cores of the solid tori attached along a multislope in $\mathcal{A}$ may, in fact, contain a Hopf sublink $H$; 
our assumption only requires that the core of the solid torus attached to 
$\partial N(K_{1})$ is not a component of $H$.
\item If there exist $\alpha' \pf \alpha$ with $\alpha'|_{T_{1}} = \nos$ and $L(\alpha')  \cong T^{2} \times [0,1]$
then obviously 
$\mathcal{A}_{1}$ may contain of all the slopes of $T_{1}$.
\end{enumerate}
\end{rmks}

\begin{proof}[Proof of Theorem~\ref{thm:cosmeticSurgeryOnS3}]
We will denote $\mbox{cl}(S^{3} \setminus N(L))$ as $X$ and $\partial N(K_{i})$ as $T_{i}$.
Although the theorem
was phrased in terms of surgery, we will prove the equivalent statement for fillings of $X$.
We induct on $|T(X)|$.   

\bigskip\noindent
{\bf Assume that $X$ is not prime.}  Let $X = X' \# X''$, where here $X'$ is the factor of the prime 
decomposition of $X$ that contains $T_{1}$.
By renumbering the components of $\partial X$ if necessary we assume as we may that $\partial X_{1} = T_{1},\dots,T_{k}$,
for some $1 \leq k \leq n$.  For any multislope $\alpha \in \mathcal{A}$ we have
$$X(\alpha) \cong X'(\alpha|_{\partial X'}) \# X''(\alpha|_{\partial X''}).$$
Thus $X'(\alpha|_{\partial X'}) \cong S^{3} \cong X''(\alpha|_{\partial X''})$.   
If, for some $2 \leq i \leq k$, 
$$X'(\nos,\alpha_{2},\dots,\alpha_{i-1},\nos,\alpha_{i+1},\dots,\alpha_{k}) \cong T^{2} \times [0,1],$$
then 
\begin{eqnarray*}
L(\nos,\alpha_{2},\dots,\alpha_{i-1},\nos,\alpha_{i+1},\dots,\alpha_{n}) &\cong& 
X'(\nos,\alpha_{2},\dots,\alpha_{i-1},\nos,\alpha_{i+1},\dots,\alpha_{k}) \# X''(\alpha|_{\partial X''}) \\
&\cong& T^{2} \times [0,1] \# S^{3} \\
&\cong& T^{2} \times [0,1].
\end{eqnarray*}
This contradicts the second assumption of the theorm.  
Thus $X'$ and $\mathcal{A}' = \{\alpha|_{\partial X'} | \alpha \in \mathcal{A}\}$
fulfill the assumptions of the theorem.  Since $X'$ corresponds to a direct descendant of the root of $T(X)$,
$|T(X')| < |T(X)|$.  By induction,  $\mathcal{A}_{1}' = \{\alpha'|_{T_{1}} \ |\ \alpha' \in \mathcal{A}'  \}$ is bounded.
It is easy to see that $\mathcal{A}_{1} = \mathcal{A}_{1}'$; the theorem follows in this case.

We assume from now on that $X$ is prime.

\bigskip\noindent
{\bf Assume that $X$ is Seifert fibered a sol manifold.}  Clearly we may ignore sol manifolds.
If $n=1$ then $L$ is a knot and the result is well known; 
assume from now on $n\ge2$.  We fix a Seifert fiberation on $X$.  Then the fibers on $T_{1}$
define a slope which we denote by $\alpha_{1}^{f}$.  
For convenience we will denote $\alpha|_{T_{i}}$ as $\alpha_{i}$.
Define $\mathcal{A}_{f},\ \mathcal{A}_{0}, \ \mathcal{A}_{1} \subset \mathcal{A}$ by:
	\begin{enumerate}
	\item  $\alpha \in \mathcal{A}_{f}$ if $\alpha_{{1}} = \alpha_{1}^{f}$.
	\item $\alpha \in \mathcal{A}_{0}$ if for some $2 \leq i \leq n$, $\alpha_{{i}}$ is the fiber on $T_{i}$.
	\item $\alpha \in \mathcal{A}_{1}$ if $\alpha \not\in \mathcal{A}_{f} \cup \mathcal{A}_{0}$.
	\end{enumerate}
Clearly,  $\mathcal{A} = \mathcal{A}_{f} \cup \mathcal{A}_{0} \cup \mathcal{A}_{1}$.	
	
If $\alpha \in \mathcal{A}_{f}$ then $\alpha_{1} \in B_{f}$, where $B_{f}$ is the bounded set of slopes of $T_{1}$ 
defined by $B_{f} = \{\alpha_{1}^{f}\}.$

For $\alpha \in \mathcal{A}_{0}$, let $D$ be the 
disk obtained by attaching a vertical annulus connecting $T_{i}$ and $T_{1}$ to the meridian disk of the solid 
torus attached to $T_{i}$ (where here $i$ is as in the definition of $\mathcal{A}_{0}$).  
Thus we see that $D$ is a compressing disk for $T_{1}$ and the slope defined by $\partial D$
is $\alpha_{1}^{f}$.  Since $X(\alpha_{1},\dots,\alpha_{n}) \cong S^{3}$,
$\Delta(\alpha_{1},\alpha_{1}^{f}) = 1$ (recall that $\Delta$ denotes the geometric intersection
number).  Thus $\alpha_{1} \in B_{0}$, where $B_{0}$ is the bounded set of slopes of $T_{1}$ defined by:
$$B_{0} = \{\alpha | \ \Delta(\alpha,\alpha_{1}^{f}) = 1\}.$$

If $\alpha \in \mathcal{A}_{1}$ then 
the fiberation of $X$ extends to a fiberation of $X(\nos,\alpha_{2},\dots,\alpha_{n})$,
and the fiberation of $X(\nos,\alpha_{2},\dots,\alpha_{n})$ extends to a fibration of
$X(\alpha) \cong S^{3}$.  
Thus $X(\nos,\alpha_{2},\dots,\alpha_{n})$ is a Seifert fibered
space over $D^{2}$ with at most two exceptional fibers and the cores of the solid tori
attached to $T_{1},\dots,T_{n}$ are fibers.

Assume first that  $X(\nos,\alpha_{2},\dots,\alpha_{n})$ is a Seifert fibered
space over $D^{2}$ with exactly two exceptional fibers.  Then $\alpha_{1} \in B_{0}$.

Next assume that  $X(\nos,\alpha_{2},\dots,\alpha_{n})$ is a Seifert fibered
space over $D^{2}$ with at most one exceptional fiber, and the exceptional fiber
(if exists) is the core of the solid torus attached to $T_{i}$; 
by renumbering $T_{2},\dots,T_{n}$ if necessary we may assume that
$i=2$.  Then $X(\nos,\nos,\alpha_{3},\dots,\alpha_{n}) \cong T^{2} \times [0,1]$, 
contradicting our assumption.

Thus we have reduced the proof to the case where $X(\nos,\alpha_{2},\dots,\alpha_{n})$ is a Seifert fibered
space over $D^{2}$ with exactly one exceptional fiber, and the exceptional fiber
is not the core of a solid torus attached to $T_{i}$ ($2 \leq i \leq n$).
Then the exceptional fiber is contained in $X$ and its multiplicity, which we will denote as $d$,
does not depend on $\alpha \in \mathcal{A}_{1}$.  
Since $X(\nos,\alpha_{2},\dots,\alpha_{n})$ is a Seifert fibered space over $D^{2}$ with exactly one exceptional fiber, 
$X(\nos,\alpha_{2},\dots,\alpha_{n}) \cong D^{2} \times S^{1}$; 
we will denoting the slope defined by the boundary of its meridian disk as $\alpha'$.  Then 
$\Delta(\alpha_{1}^{f},\alpha') = d$.  Since $X(\alpha) \cong S^{3}$, we also have that
$\Delta(\alpha',\alpha_{1})=1$.
Thus $\alpha_{1}\in B_{1}$,  
where $B_{1}$ is the set of slopes of $T_{1}$ defined by:
$$B_{1} = \{\alpha\ | \ (\exists \alpha')\  \Delta(\alpha_{1}^{f},\alpha') = d \mbox{ and }\Delta(\alpha',\alpha) = 1\}.$$
By applying Proposition~\ref{prop:PropertiesOfBoundedSets}~(4) twice we see that $B_{1}$ is bounded.

Since $\mathcal{A} = \mathcal{A}_{f} \cup \mathcal{A}_{0} \cup \mathcal{A}_{1}$, for any
$\alpha \in \mathcal{A}$, $\alpha_{1} \in B_{f} \cup B_{0} \cup B_{1}$.  This completes the proof 
for Seifert fibered and sol manifolds.

We assume from now on that $X$ is prime and not Seifert fibered or a  slov manifold.

\bigskip\noindent
{\bf Assume that $X$ is JSJ.}  Let $X_{0}$ be the component of the torus decomposition of $X$
that contains $T_{1}$ and denote the remaining components of $\partial X$ as $\{F_{j}\}_{j=1}^{k}$,
see Figure~\ref{fig:thm10JSJ}.
\begin{figure}
\includegraphics[width=4in]{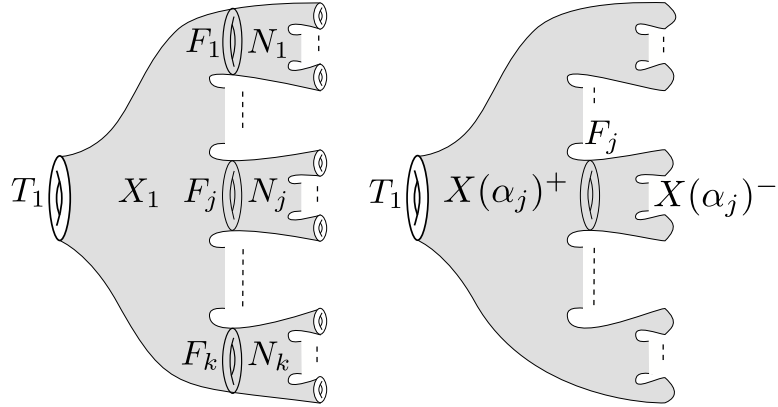}
\caption{Notation used when $X$ is JSJ}
\label{fig:thm10JSJ}
\end{figure}
Given $\alpha \in \mathcal{A}$, we will denote the components of $X(\alpha|_{T_{2},\dots,T_{n}})$
cut open along $F_{j}$ as follows (see the right side of Figure~\ref{fig:thm10JSJ}):
\begin{enumerate}
\item The component whose boundary is $T_{1} \cup F_{j}$ will be denoted as $X(\alpha)_{j}^{+}$.
\item The component whose boundary is $F_{j}$ will be denoted as $X(\alpha)_{j}^{-}$.
\end{enumerate}
Since $S^{3}$ does not admit a non separating torus, we assume as we may that
$X(\alpha)_{j}^{-} \cap X(\alpha)_{j'}^{-} = \emptyset$ for $j \neq j'$.

There are two cases to consider:

\bigskip\noindent
{\bf Case One.}  Let $\mathcal{A}_{1}$ be the multislopes $\alpha \in \mathcal{A}$ for which 
$X(\alpha)_{j}^{+} \not\cong T^{2} \times [0,1]$ for all $j$.  
Then $\alpha\in \mathcal{A}_{1}$ induces a multislope on $\partial X_{0}$,
which we will denote as $\alpha_{0}$, defined as follows:
	\begin{enumerate}
	\item $\alpha_{0}|_{T_{1}} = \alpha|_{T_{1}}$.
	\item If $X(\alpha)_{j}^{-} \cong D^{2} \times S^{1}$ then $\alpha_{0}|_{F_{j}}$ is the slope 
	of the meridian of the solid torus $X(\alpha)_{j}^{-}$.
	\item If $X(\alpha)_{j}^{-} \neq D^{2} \times S^{1}$ then $X(\alpha)_{j}^{-} \cong E(K_{j})$ 
	for some non trivial knot $K_{j} \subset S^{3}$.
	We take $\alpha_{0}|_{F_{j}}$ to be any slope that intersects the meridian of $E(K_{j})$ exactly once.
	\end{enumerate}
	
By Lemma~\ref{lem:KnotExteriorsAreDisjoint}, we may assume that the components  $X(\alpha)_{j}^{-}$
in Case~(3) are contained in disjointly embedded balls; hence removing $X(\alpha)_{j}^{-}$ and attaching a solid
torus along $\alpha_{0}|_{F_{j}}$ does not change 
$X(\alpha)$ or $X(\alpha)_{j'}^{+}$ (for $1 \leq j' \leq k$), recall
Figure~\ref{fig:unknotting}.  Thus $X_{0}$ and the induced slopes 
$\{\alpha_{0}\ | \ \alpha \in \mathcal{A}_{1}\}$ satisfy the conditions of the Theorem
(condition~(1) follows from the corresponding assumption for $X$, and condition~(2)
follows from the defining assumption of case one).               
By Lemma~\ref{lem:TreeOfJSJ},  $|T(X_{0})| < |T(X)|$.  
By induction, $\{\alpha_{0}|_{T_{1}} \ | \ \alpha \in \mathcal{A}_{1}\}$ is a bounded set which
we will denote as $B_{1}$. By construction $\alpha_{0}|_{T_{1}} = \alpha|_{T_{1}}$. 

Hence $\{\alpha|_{T_{1}} \ | \ \alpha \in \mathcal{A}_{1} \} = B_{1}$ is bounded.

\bigskip\noindent
{\bf Case Two.}  Fix $1 \leq j \leq k$.  Let $\mathcal{A}_{2,j} \subset \mathcal{A}$ be the multislopes $\alpha \in \mathcal{A}$ 
for which $X(\alpha)_{j}^{+} \cong T^{2} \times [0,1]$.  The definitions immediately imply that the following two conditions
hold:
\begin{enumerate}
\item For any $\alpha \in \mathcal{A}_{2,j}$, $X(\alpha)_{j}^{+}(\alpha|_{T_{1}}) \cong D^{2} \times S^{1}$.
\item  $\mathcal{A} = \mathcal{A}_{1} \cup (\cup_{j=1}^{k} \mathcal{A}_{b,j})$.
\end{enumerate}
We will denote the component of $X$ cut open along $F_{j}$ that does not contain $X_{1}$ as $N_{j}$
and $\partial X \cap N_{j}$ as $\mathcal{T}_{j}$; to avoid the situation $N_{j} = \emptyset$,
if $F_{j} \subset \partial X$ we push it slightly into the interior so that $N_{j} \cong T^{2} \times [0,1]$
in this case.  Note that $\partial N_{j} = F_{j} \cup \mathcal{T}_{j}$.

Every $\alpha \in \mathcal{A}_{2,j}$ induces the multislope on $\partial N_{j}$, 
which we will denote as  $\alpha_{j}$, defined by 
the slope of the meridian of the solid torus $X(\alpha)_{j}^{+}(\alpha|_{T_{1}})$
on $F_{j}$ and the restriction $\alpha|_{\mathcal{T}}$ on $\mathcal{T}_{j}$.
We show that the following two conditions hold:
	\begin{enumerate}
	\item $N_{j}(\alpha_j) \cong S^{3}$: by construction, $N_{j}(\alpha_{j}) \cong X(\alpha)$
	and by assumption, $X(\alpha) \cong S^{3}$.
	\item Let $\alpha_{j}' \pf \alpha_{j}$ be a partial filling for which $\alpha_{j}'|_{F_{j}} = \nos$.
	Then $N_{j}(\alpha_{j}') \not\cong D^{2} \times S^{1}$: assume for a 
	contradiction that $N_{j}(\alpha_{j}') \cong D^{2} \times S^{1}$. 
	Let $T_{i}$ be the component of $\partial N_{j} \cap X$ for which $\alpha_{j}'|_{T_{i}} = \nos$.
	Let $\alpha' \pf \alpha$ be the partial filling giving by setting $\alpha'|_{T_{1}}$ and
	$\alpha'|_{T_{i}}$ to $\nos$.  Then 
	\begin{eqnarray*}
	X(\alpha') &\cong& X(\alpha)_{j}^{+} \cup_{F_{j}} N_{j}(\alpha_{j}') \\
	&\cong&  T^{2} \times [0,1] \cup_{F_{j}} T^{2} \times [0,1] \\ 
	&\cong& T^{2} \times [0,1],
	\end{eqnarray*}
	violating assumption~(2) of the theorem.
	\end{enumerate}
Thus the assumptions of the theorem are satisfied by $N_{j}$ and
$\{\alpha_{j}   \ | \ \alpha \in \mathcal{A}_{2,j} \}$.  By Lemma~\ref{lem:TreeOfJSJ},
$|T(N_{j})| < |T(X)|$.  By induction,
$\{\alpha_j|_{F_{j}}\}$ is bounded.  For each $\alpha \in \mathcal{A}_{2,j}$, 
$\alpha|_{T_{1}}$ is the image of $\alpha_{j}|_{F_{j}}$ under the projection induced by the product structure
$X(\alpha)_{j}^{+} \cong T^{2} \times [0,1]$.  By the $T^{2} \times [0,1]$ Cosmetic Surgery 
Theorem~(\ref{thm:CosmeticSurgeryOnT2XI}),  the union of the images of $\{\alpha_j|_{F_{j}}\}$
under the projections given by all possible filling of $\partial X(\alpha)_{j}^{+} \setminus (T_{1} \cup F_{j})$
for which $X(\alpha)_{j}^{+} \cong T^{2} \times [0,1]$ is a bounded set of slopes of $T_{1}$
that we will denote as $B_{2,j}$.  We conclude that $\alpha|_{T_{1}} \in B_{2,j}$.
This completes case two.

\bigskip\noindent
Since $\mathcal{A} = \mathcal{A}_{1} \cup (\cup_{j} \mathcal{A}_{2,j})$ we have
$$\{\alpha|_{T_{1}} \ | \ \alpha \in \mathcal{A} \} \subset B_{1} \cup (\cup_{j}B_{2,j}).$$   
The theorem follows for JSJ manifolds, and we assume from now on that $X$ is prime, not Seifert fibered,
sol, or JSJ.  Thus $X$ is hyperbolic.

\bigskip\noindent
{\bf Assume that $X$ is hyperbolic.}  Since $S^{3}$ is not hyperbolic any $\alpha \in \mathcal{A}$
admits a \mnh\ partial filling.  Up-to finite ambiguity we fix a \mnh\ filling which we
will denote as $\alpha'$.
If $\alpha'|_{T_{1}} \neq \nos$ then for any $\alpha$ with $\alpha' \pf \alpha$,
$\alpha|_{T_{1}} = \alpha'|_{T_{1}}$ is in the finite (and hence bounded) set $\{\alpha'|_{T_{1}}\}$.
Otherwise, any $\alpha$ with $\alpha' \pf \alpha$ factors through $X(\alpha')$:
$$X \stackrel{\alpha'}\to X(\alpha') \stackrel{\alpha|_{\partial X(\alpha')}}\longrightarrow X(\alpha')(\alpha|_{\partial X(\alpha')}).$$
It is straightforward to see that we can apply induction to $X(\alpha')$ and 
$\{\alpha|_{\partial X(\alpha')} \ | \ \alpha \in \mathcal{A} \}$, showing that
$\{\alpha|_{\partial X(\alpha')}|_{T_{1}} \ | \ \alpha \in \mathcal{A} \}$ is bounded.
Thus for any $\alpha$ with $\alpha' \pf \alpha$,
$\alpha|_{T_{1}} = \alpha'|_{T_{1}}$ is in this bounded set.

This completes the proof in this final case.
\end{proof}

\section{Proof of Theorem~\ref{thm:main}}
\label{sec:SetUpOfProof}

\bigskip\noindent
In this section we prove Theorem~\ref{thm:main} assuming the results of the next three sections.
We decided to present the proof before Sections~\ref{sec:FillingsThatDontFactorThroughM},
\ref{sec:FillingsOfEthatDontFactorThroughTXI}, and~\ref{sec:BothFillingsFactor}
in order to help the reader understand the motivation behind the exact statements proved in those
sections.

\bigskip\noindent
As in the statement of the theorem, let $M$ be a hyperbolic manifold so that
$\partial M$ is a single torus that we will denote as $T$and $V>0$ a fixed number.  
Consider $\beta$, a slope on $T$, so that $\lv[M(\beta)] < V$.  
By the Structure Theorem of~\cite{rieckyamashita}, there exist
finitely many covers $\phi:X \to E$, with $X$ and $E$ hyperbolic, so that the following 
diagram commutes:

\bigskip

\begin{center}
\begin{picture}(200,60)(0,0)
\label{diagram}
  \put(  0,  0){\makebox(0,0){$E$}}
  \put(  0,50){\makebox(0,0){$X$}}
  \put(  0, 40){\vector(0,-1){30}}   
  \put( 10,  0){\vector(1,0){73}}
  \put(100,  0){\makebox(0,0){$(S^3,L)$}}
  \put(100, 40){\vector(0,-1){30}}
  \put(108,28){\makebox(0,0)}
  \put(100,50){\makebox(0,0){$M(\beta)$}}
  \put( 10,50){\vector(1,0){75}}
\end{picture}
\end{center}

\bigskip\noindent
Here, the horizontal arrows represent inclusions induced by fillings and the vertical arrows represent covering projections;
$\phi:X \to E$ is an unbranched covering projection between hyperbolic manifolds and $\hat\phi:M \to S^{3}$ is a branched 
cover realizing $\lv[M(\beta)]$.  Up-to finite ambiguity we fix one cover $\phi:X \to E$.  We will denote the components
of $\partial X$ as $T_{1},\dots,T_{n}$.

\bigskip\noindent
{\bf Case One.}   We first consider fillings $X \to M(\beta)$ that do not factor though $M$, that is, slopes $\beta$ so that
for some multislope $\alpha$ of $\partial X$ 
the following two conditions hold:
	\begin{enumerate}
	\item $X(\alpha) \cong M(\beta)$.
	\item There is no $\alpha' \pf \alpha$ so that $X(\alpha') \cong M$.
	\end{enumerate}
In Section~\ref{sec:FillingsThatDontFactorThroughM} we prove Theorem~\ref{thm:fillingsThatDontFactor},
showing that the set of slopes $\beta$ 
that arise in this way, which we will denote as $B_{1}$, is bounded.  
We remark that this is a general fact about fillings $X \to M(\beta)$
and does not use the covers $\phi:X \to E$ and $\hat{\phi}:M(\beta) \to S^{3}$.

\bigskip\noindent
{\bf Case Two.}  
We next consider fillings $X(\alpha) \cong M(\beta)$ that do factor though $M$, that is, fillings that admit a partial
filling $M$.   Up to finite ambiguity, we may assume that the component of $\partial X$ 
that corresponds to $\partial M$ is $T_{1}$.  Thus we are considering multislopes $\alpha$
so that the diffeomorphism $X(\alpha) \to M(\beta)$ induces, by restriction, a diffeomorphism 
$$X(\alpha|_{T_{2},\dots,T_{n}}) \to M.$$

Denote the components of $\partial E$ by $T'_{1},\dots,T'_{m}$.  By renumbering  $T'_{1},\dots,T'_{m}$
if necessary we may assume
that $T'_{1} = \phi(T_{1})$.  The diagram above implies that
$\alpha$ induces a multislope of $\partial E$, which we will denote as $\alpha^{E} = (\alpha_{1}^{E},\dots,\alpha^{E}_{m})$.   
In Case Two we only consider fillings on $E$ that do not factor through $T^{2} \times [0,1]$; more precisely:
$$(\forall \alpha' \pf \alpha^{E}\mbox{ with } \alpha'|_{T_{1}'} = \nos) \hspace{10pt} E(\alpha') \not\cong T^{2} \times [0,1].$$
In that case, the strategy is as follows: applying the $S^{3}$ Cosmetic Surgery Theorem~(\ref{thm:cosmeticSurgeryOnS3}) we see that the
possibilities for $\alpha_{1}^{E}$ are bounded; the covering projection $\phi:X \to E$ induces a bilipschitz bijection between the 
slopes of $T_{1}'$ and those of $T_{1}$ (Lemma~\ref{lem:CorrespondingSlopesFeray}); hence the possibilities for slopes on $T_{1}$ are bounded.
The argument is worked out in detail in 
Section~\ref{sec:FillingsOfEthatDontFactorThroughTXI}.

\medskip\noindent
{\bf Case Three.}  The last and most exciting case is when the filling of $X$ factors through $M$ and the filling of $E$
factors through $T^{2} \times [0,1]$.  The proof in this case is given in Section~\ref{sec:BothFillingsFactor}.
Again, we conclude that $\{\alpha|_{T_{1}}\}$ is bounded.

\bigskip
\bigskip\noindent
Assuming the results of the following sections, we deduce Theorem~\ref{thm:main} as follows:

\begin{proof}[Proof of Theorem~\ref{thm:main}]
Let ${\mathcal{A}}$ be the set of all multislopes of $\partial X$ so that
for each $\alpha \in {\mathcal{A}}$ there is a slope $\beta$ of $\partial M$ so that
$X(\alpha) \cong M(\beta)$.

Since $M$ is hyperbolic there is a finite set of slopes of $\partial M$, which we will denote
as $B_{F}'$, so that for any $\beta \not \in B_{F}'$, $M(\beta)$ is hyperbolic and
the core of the attached solid torus is its unique shortest geodesic (Lemma~\ref{lem:HyperbolicDehnSurgery}).  
Let $B_{F}$ be the set of slopes of $T$ defined as
$$B_{F} = \{ \beta \ | \ (\exists \beta' \in B_{F}')\ M(\beta) \cong M(\beta')\}.$$
Since $M$ is hyperbolic and $|\partial M|=1$, no manifold is obtained by filling infinitely 
many distinct slopes of $\partial M$; hence $B_{F}$ is finite.  
We will only consider multislopes $\alpha \in \mathcal{A}$ for which 
X$(\alpha) \cong M(\beta)$ for $\beta \not\in B_{F}$.  
To void overly complicated notation we will not rename $\mathcal{A}$.

We now consider Cases Two and Three.  We will denote as ${\mathcal{A}}_{2,3}$
the multislopes of ${\mathcal{A}}$ in these cases, so that every $\alpha \in \mathcal{A}_{2,3}$
admits a partial filling $\alpha' \pf \alpha$ so that $X(\alpha') \cong M$.   Up-to finite ambiguity we may assume
that $\alpha' = \alpha|_{T_{2},\dots,T_{n}}$.
 We will denote the set of restrictions $\{\alpha|_{T_{1}}\ | \ \alpha \in {\mathcal{A}}_{2,3}\}$ as $B_{M}'$.
By Propositions~\ref{pro::FillingsOfEthatDontFactorThroughTXI}
and~\ref{pro:TheLastCase} and Lemma~\ref{lem:BoundedSetsOnMathcalT} 
(see also Remark~\ref{rmk:LastCase}), $B_{M}'$ is bounded.  

Let $\beta$ be a slope
of $\partial M$ so that $M(\beta) \cong X(\alpha)$ for some $\alpha \in {\mathcal{A}}_{2,3}$.  
We will consider $X(\alpha)$ 
as $X(\alpha|_{T_{2},\dots,T_{n}})(\alpha|_{T_{1}})$,
the manifold obtained by filling $X(\alpha|_{T_{2},\dots,T_{n}})$ along slope $\alpha|_{T_{1}}$.
Then we see that $X(\alpha|_{T_{2},\dots,T_{n}}) \cong M$ and
$$X(\alpha|_{T_{2},\dots,T_{n}})(\alpha|_{T_{1}}) \cong M(\beta).$$
Let $f:X(\alpha|_{T_{2},\dots,T_{n}}) \to M$ be a diffeomorphism.  Denote the image of 
$\alpha|_{T_{1}}$ under $f$ as $\alpha^{M}$ and 
the image of $B_{M}'$ under $f_{\alpha|_{T_{2},\dots,T_{n}}}$ as 
$B_{M,f}$; note that $\alpha^{M} \in B_{M,f}$.  Let $i$ be the isometric
involution on the slopes of $T$ given by Lemma~\ref{lem:MCGisometryOfSlopes}; we will denote 
$B_{M,f} \cup i(B_{M,f})$ 
as $B_{M}$.  Clearly $B_{M}$ is bounded and $\alpha^{M} \in B_{M}$.

Since $M(\alpha^{M}) \cong X(\alpha|_{T_{2},\dots,T_{n}})(\alpha|_{T_{1}})$
and $M(\beta)$ are diffeomorphic, by Mostow's Rigidity there is an isometry
$f:M(\alpha^{M}) \to M(\beta)$.   Since $\beta \not\in B_{F}$, the cores of the attached solid tori are the shortest geodesics 
of $M(\alpha^{M})$ and $M(\beta)$.  
Thus $f$ carries the core of the solid torus attached to $M(\alpha^{M})$ 
to the core of the solid torus attached to $M(\beta)$ and hence by restriction $f$
induces an isometry $M \to M$ that maps $\alpha^{M}$ to $\beta$.
By Lemma~\ref{lem:MCGisometryOfSlopes}, $\alpha^{M} = \beta$
or $\alpha^{M} = i(\beta)$.  Since $\alpha^{M} \in B_M$ and $B_{M} = i(B_{M})$,
$\beta \in B_{M}$.

Recall that we denoted the set of slopes of $\partial M$ that is realized in Case One 
as $B_{1}$; in Theorem~\ref{thm:fillingsThatDontFactor}
we show that $B_{1}$ is bounded.  Combining all the possibilities we see that
$$\beta \in B_{M}  \cup B_{1}\cup B_{F}.$$  
Theorem~\ref{thm:main} follows.
\end{proof}

\section{Case One: fillings of $X$ that do not factor through $M$}
\label{sec:FillingsThatDontFactorThroughM}

\bigskip\noindent
In this section we consider two manifolds, denoted as $X$ and $M$, where $M$ is a one cusped hyperbolic
manifolds.  If $\alpha'$ is a multislope of $X$ so that $X(\alpha') \cong M$, then (trivially) for any slope $\beta$
of $\partial M$ there is a multislope $\alpha$ so that $\alpha' \pf \alpha$ and $X(\alpha) \cong M(\beta)$.
In Theorem~\ref{thm:fillingsThatDontFactor} we show that if one considers only
multislopes $\alpha$ that do not admit 
$\alpha' \pf \alpha$ with $X(\alpha') \cong M$, 
then the set of slopes of $\partial M$ that give rise to manifolds that are also obtained
by filling $X$ is bounded.  This is purely a result about filling and
is independent of the covers considered in this paper.   The precise statement is:

\begin{thm}
\label{thm:fillingsThatDontFactor}
Let $X$ be a compact orientable connected manifold so that $\partial X$ consists of tori
and $M$ a hyperbolic manifold with $\partial M$  a single torus.  
Let ${\mathcal{A}}$ be a set of multislopes of $\partial X$ 
and ${\mathcal{B}}$ a set of slopes of $\partial M$ 
fulfilling the following two conditions:
	\begin{enumerate}
	\item For every $\beta \in {\mathcal{B}}$ there is a multislope 
	$\alpha \in {\mathcal{A}}$ so that $X(\alpha) \cong M(\beta)$.
	\item For every $\alpha \in \mathcal{A}$ and every $\alpha' \pf \alpha$, $X(\alpha) \not\cong M$.  
	\end{enumerate}
Then $\mathcal{B}$ is bounded.
\end{thm}

\begin{proof}
By Thurston's Dehn Surgery Theorem (see Lemma~\ref{lem:HyperbolicDehnSurgery}) we may fix an $\epsilon>0$
and a finite set of slopes of $T$, which we will
denote as $B_{f}'$, so that for every $\beta \not\in B_{f}'$ the following 
three conditions hold:
	\begin{enumerate}
	\item $M(\beta)$ is hyperbolic.
	\item The core of the attached solid torus, which we will denote as $\gamma$, is a geodesic and $l(\gamma) < \epsilon$.
	\item Any geodesic $\delta \subset M(\beta)$ with $l(\delta) < \epsilon$ is a power of $\gamma$.
	\end{enumerate}
Since no manifold is obtained by filling along infinitely many distinct slopes of $M$, the set 
$B_{f} = \{ \beta \ | \ (\exists \beta' \in B_{f}') \ M(\beta) \cong M(\beta') \}$ is finite.
For the remainder of the proof we
only consider $\beta \not \in B_{f}$.  Accordingly, we remove the multislopes $\alpha \in \mathcal{A}$
for which $X(\alpha) \cong M(\beta)$ for $\beta \in B_{f}$.  To avoid overly complicated notation we
do not rename $\mathcal{A}$ and $\mathcal{B}$.

\bigskip
\noindent
We induct on $|T(X)|$.

\bigskip\noindent
{\bf Assume that $X$ is Seifert fibered or sol.}  Then $X$ cannot be filled to give a hyperbolic manifold.

We assume from now on that $X$ is not Seifert fibered or solv.

\bigskip\noindent
{\bf Assume that $X$ is not prime.} Let $X_{1},\dots,X_{n}$ be the factors of the prime
decomposition of $X$.  For every $\alpha \in \mathcal{A}$ 
we will denote the restriction $\alpha|_{\partial X_{i}}$ as $\alpha_{i}$.
By considering the image of decomposing spheres for $X$ in $X(\alpha)$ we see that
for every $\alpha \in \mathcal{A}$: 
\begin{equation}
\label{eq:FillingsThatDontFactor1}
X(\alpha) = X_{1}(\alpha_{1}) \# \cdots \#  X_{n}(\alpha_{n}).
\end{equation}
For $1 \leq i \leq n$ we will denote as $\mathcal{A}_{i} \subset \mathcal{A}$ the multislopes $\alpha$ for which
$X_{i}(\alpha_{i}) \cong M(\beta)$ (for some $\beta \in \mathcal{B}$)
and for $i' \neq i$, $X_{i'}(\alpha_{i'}) \cong S^{3}$.
Since $\beta \not \in B_{f}$, $X(\alpha)$ is hyperbolic and hence prime.  
Thus by Equation~(\ref{eq:FillingsThatDontFactor1}) 
$$\mathcal{A} = \cup_{i=1}^{n} \mathcal{A}_{i}.$$
Fix $1 \leq i \leq n$.
Suppose, for a contradiction, that for some $\alpha \in \mathcal{A}_{i}$ there is
$\alpha_{i}' \pf \alpha_{i}$ so that $X_{i}(\alpha_{i}') \cong M$.
Then there is a unique component $T$ of $\partial X_{i}$ for which $\alpha_{i}'|_{T} = \nos$.
Let $\alpha' \pf \alpha$ be the partial filling defined by setting $\alpha'|_{T} = \nos$ and
$\alpha'|_{T'} = \alpha|_{T'}$ for any other component $T'$ of $\partial X$.  Then by
Equation~(\ref{eq:FillingsThatDontFactor1}),
\begin{eqnarray*}
X(\alpha') &\cong& X_{i}(\alpha'|_{\partial X_{i}}) \ \# \ (\#_{i' \neq i}X_{i'}(\alpha'|_{\partial X_{i'}})) \\
 &\cong& X_{i}(\alpha'_{i}) \ \# \ (\#_{i' \neq i} X_{i'}(\alpha_{i'})) \\
  &\cong& M \ \# \ (\#_{i' \neq i}S^{3})  \\
  &\cong& M.
\end{eqnarray*}
(Here we used that $ X_{i'}(\alpha_{i'})) \cong S^{3}$, which holds by the definition of $\mathcal{A}_{i}$.)  
This contradicts the assumptions of the theorem.  Hence $X_{i}$ and 
$\{\alpha_{i} \ | \ \alpha \in \mathcal{A}_{i}\}$ fulfill the assumptions of the theorem.  Since $X_{i}$
corresponds to a direct descendant of the root of $T(X)$, $|T(X_{i})| < |T(X)|$.  We will denote as $B_i$
the set of slopes of $T$ so that for every $\beta \in B_{i}$ there is $\alpha_{i} \in \mathcal{A}_{i}$
with $M(\beta) \cong X_{i}(\alpha_{i})$.  By induction, $B_{i}$ is bounded.
As $i$ was arbitrary, we obtain bounded sets $B_{1},\dots,B_{n}$.
Since $\mathcal{A} = \cup_{i=1}^{n} \mathcal{A}_{i}$,
$$\mathcal{B} \subset B_{f} \cup (\cup_{i=1}^{n} B_{i}).$$
The theorem follows for non prime manifolds. 

We assume from now on that $X$ is prime and not Seifert fibered or sol.

\bigskip\noindent
{\bf Assume that $X$ is hyperbolic.}  
Since there is no lower bound on the length of geodesics in $\{M(\beta) \ | \ \beta \in \mathcal{B} \}$, 
Theorem~\ref{thm:HyperbolicFillingShortGeos} does no apply directly;
however, the proof here is similar to the proof of that theorem where more details can be found.  We start with:

\medskip\noindent
{\bf Claim.}  There are only finitely many totally hyperbolic fillings in $\mathcal{A}$.

\medskip\noindent
When proving the claim we may obviously assume that $\mathcal{A}$ is infinite.
Let $\{ \alpha^{j} \}_{j \in \mathbb N} \subset \mathcal{A}$ be an infinite set.
We will prove the claim by showing that some multislope of   $\{ \alpha^{j} \}_{j \in \mathbb N}$
is not totally hyperbolic.  
We denote the components of $\partial X$ as $T_{1},\dots,T_{n}$ and $\alpha^{j}|_{T_{i}}$ as $\alpha_{i}^{j}$.
After subsequencing and reordering if necessary we assume as we may that for some $0 \leq k \leq n+1$ we have:
\begin{enumerate}
\item For every $1 \leq i \leq k$ and every $j \neq j'$, $\ \alpha_{i}^{j} \neq \alpha_{i}^{j'}$ and $\alpha_{i}^{j} \neq \nos$.
\item For every $k+1 \leq i \leq n+1$ and every $j, j'$, $\ \alpha_{i}^{j} = \alpha_{i}^{j'}$.
\end{enumerate}
Since $\{ \alpha^{j} \}_{j \in \mathbb N}$ is infinite, $k \geq 1$.
Let $\alpha' = (\nos,\dots,\nos,\alpha^{j}_{k+1},\dots,\alpha^{j}_{n})$; by construction
$\alpha'$ does not depend on $j$.  
We claim that $X(\alpha')$ is not hyperbolic.  Assume, for a contradiction, that it is.
Then for any $l>0$ we may choose $j$ so that the slopes $\alpha_{1}^{j},\dots,\alpha^{j}_{m}$
are all longer than $l$ (where here the lengths are measured in the Euclidean metrics induced on the
boundary components after some truncation of the cusps).  By Thurston's Dehn Surgery Theorem if $l$ is sufficiently large then
$X(\alpha')(\alpha_{1}^{j},\dots,\alpha^{j}_{m})$ is hyperbolic and 
admits at least $k$ geodesics that are all shorter than $\epsilon$; since $M(\beta)$
admits only one such geodesic, $k=1$.  Thus $X(\alpha')$ and $M$ are one cusped hyperbolic manifolds
that admit infinitely many diffeomorphic fillings.  A standard application of Mostow's Regidity shows 
that $X(\alpha') \cong M$.  Since $\alpha' \pf \alpha^{j}$ this contradicts the second assumption
of the theorem, showing that $X(\alpha')$ is not hyperbolic.  Hence $\alpha' \pf \alpha^{j}$
is a non hyperbolic partial filling, showing that $\alpha^{j}$ is not totally hyperbolic.
Hence there are only finitely many \mnh\ fillings, as claimed.

\bigskip\noindent
By the claim there are only finitely many $\alpha \in \mathcal{A}$ that are totally hyperbolic.
We remove these multislopes from $\mathcal{A}$ and remove from $\mathcal{B}$ the finite set
of slopes $\beta$ for which $M(\beta) \cong X(\alpha)$ only for totally hyperbolic fillings; 
to avoid overly complicated
notation we do not rename $\mathcal{A}$ and $\mathcal{B}$.  Hence every multislope
of $\mathcal{A}$ admits a \mnh\ partial filling.  Up-to finite ambiguity we fix one
\mnh\ partial filling of $X$ and denote it as $\alpha'$.  By definition of partial filling,
if $\alpha' \pf \alpha$ then $X(\alpha)$ is obtained by filling $X(\alpha')$ along
slopes $\alpha|_{\partial X(\alpha')}$.  By assumption, $\alpha$ does not admit 
a partial filling that yields $M$; hence the same holds for $\alpha|_{\partial X(\alpha')}$.
Since $X(\alpha')$ corresponds to
a direct descendant of the root of $T(X)$, $|T(X(\alpha'))| < |T(X)|$.  By induction
the set of slopes of $T$
$$\{\beta \ | \ M(\beta) \cong X(\alpha')(\alpha|_{\partial X(\alpha')}), \ \alpha \in \mathcal{A}\}$$
is bounded.  The theorem follows.

We assume from now on that $X$ is prime and not Seifert fibered, sol, or hyperbolic.

\bigskip\noindent
{\bf Assume that $X$ is JSJ.}  By Proposition~\ref{prop:ObtainedByFilling} for every
$\beta \in \mathcal{B}$, $M(\beta)$ is obtained by filling some component of the torus 
decomposition of $X$.  Up-to finite ambiguity we fix such component which we will
denote as $X_{1}$ and consider only slopes $\beta$ (and corresponding multislopes $\alpha$) 
so that $M(\beta)$ is obtained by filling $X_{1}$; 
to avoid overly complicated notation we do not rename $\mathcal{A}$
and $\mathcal{B}$.  Recall that in Proposition~\ref{prop:ObtainedByFilling} for every 
$\alpha \in \mathcal{A}$ we constructed a multislope for $X_{1}$, which we will
call the {\it induced} multislope and denote as $\alpha_{1}$, so that $M(\beta) \cong X_{1}(\alpha_{1})$.
We will denote the set of induced multislopes thus obtained
as $\mathcal{A}_{1}$; thus for every $\beta \in \mathcal{B}$
there is $\alpha_{1} \in \mathcal{A}_{1}$ so that $M(\beta) \cong X_{1}(\alpha_{1})$.

Let $\mathcal{A}_{1}^{+}$ be the multislopes $\alpha_{1} \in \mathcal{A}_{1}$ for which there is no 
$\alpha_{1}' \pf \alpha_{1}$ so that $X_{1}(\alpha_{1}') \cong M$.  Then $X_{1}$ and
$\mathcal{A}_{1}^{+}$  fulfill the assumptions of the theorem.   By Lemma~\ref{lem:TreeOfJSJ}, 
$|T(X_{1})| < |T(X)|$.  By induction, the set
$$\{\beta \ | \ (\exists \alpha_{1} \in \mathcal{A}_{1}^{+}) M(\beta) \cong X_{1}(\alpha_{1}) \},$$
which we will denote as $B_{1}^{+}$, is bounded.  We will return to $B_{1}^{+}$ at the end of the proof.

We will denote $\mathcal{A}_{1} \setminus \mathcal{A}_{1}^{+}$ as $\mathcal{A}_{1}^{-}$. 
Then every $\alpha_{1} \in \mathcal{A}_{1}^{-}$ admits 
$\alpha_{1}' \pf \alpha_{1}$ so that $X_{1}(\alpha_{1}') \cong M$.  Up-to finite
ambiguity we fix one boundary component of $\partial X_{1}$, 
which we will denote as $F$, and consider only the 
multislopes $\alpha_{1} \in \mathcal{A}_{1}^{-}$ for which 
$\alpha_{1}'|_{F} = \nos$ (in other words, $F \subset \partial X_{1}$ 
corresponds to $T$ after filling along $\alpha_{1}'$).  
Note that $F \not\subset \partial X$, for otherwise
$\alpha_{1}$ would correspond to a multislope $\alpha \in \mathcal{A}$ that admits
$\alpha' \pf \alpha$ for which $X(\alpha') \cong M$, contradicting the assumptions 
of the theorem.  We will denote the components of $X$ cut open along
$F$ as $X^{+}$ and $X^{-}$, with $X_{0} \subset X^{+}$.

In the remainder of the proof we work directly with the set of multislopes of $\mathcal{A}$ that
induce multislopes of $\mathcal{A}_{1}^{-}$ fulfilling the assumptions above,
which we will denote as $\mathcal{A}^{-} \subset \mathcal{A}$.
Given $\alpha \in \mathcal{A}^{-}$, we will denote the multislope $\alpha|_{\partial X^{+}}$ induced on $\partial X^{+}$
as  $\alpha^{+}$ and the multislope $\alpha|_{\partial X^{-}}$ induced on $\partial X^{-}$ as $\alpha^{-}$.  
By defining assumption of $\mathcal{A}^{-}_{1}$, $X^{+}(\alpha^{+}) \cong M$
and by Proposition~\ref{prop:ObtainedByFilling} either $X^{-}(\alpha^{-}) \cong D^{2} \times S^{1}$
or $X^{-}(\alpha^{-}) \cong E(K)$ for a non-trivial knot $K \subset S^{3}$.  But if the latter occured,
$F$ would be an incompressible torus in $M(\beta)$, which is absurd as $M(\beta)$ is hyperbolic.
Thus $X^{-}(\alpha^{-}) \cong D^{2} \times S^{1}$.  We will denote the slope defined by the 
meridian of $X^{-}(\alpha^{-})$ on $F$ as $\mu_{F}$ and consider $X(\alpha)$ as $X^{+}(\alpha^{+})(\mu_{F})$,
the manifold obtained by filling $X^{+}(\alpha^{+})$ along slope $\mu_{F}$.

Let $f:X^{+}(\alpha^{+})(\mu_{F}) \to M(\beta)$ be a diffeomorphism.   Then,
since we assumed that $\beta \not\in B_{f}$ (recall the definition of $B_{f}$ in the beginning
of the proof), the cores of the solid tori attached to $X^{+}(\alpha^{+})$ and $M$
are the unique shortest geodesics in $X^{+}(\alpha^{+})(\mu_{F})$ and $M(\beta)$.
By Mostow's Rigidity we may assume
that $f$ is an isometry, and thus $f$ carries the core of the solid torus attached to $X^{+}(\alpha^{+})$
to the core of the solid torus attached to $M$.  The restriction of
$f$ induces a diffeomorphism which we will denote as $f^{+}:X^{+}(\alpha^{+}) \to M$.   
Note that $f^{+}$ maps $\mu_{F}$ to $\beta$.

Turning our attention to the solid torus $X^{-}$, we claim that for every $\alpha \in \mathcal{A}^{-}$,
the induced multislope $\alpha^{-}$ satisfies the following conditions:
\begin{enumerate}
\item $X^{-}(\alpha^{-}) \cong D^{2} \times S^{1}$: this was established above.
\item For any partial filling $\hat\alpha \pf \alpha^{-}$, $X^{-}(\hat\alpha) \not\cong T^{2} \times [0,1]$:
otherwise, let $\alpha' \pf \alpha$ be the partial filling 
defined by setting $\alpha'|_{\partial X^{-}} = \hat\alpha$
and $\alpha'|_{\partial X^{+}} = \alpha|_{\partial X^{+}}$.  Then we have:
$$X(\alpha') \cong X^{-}(\hat\alpha) \cup_{F} X^{+}(\alpha^{+}|_{\partial X^{+}}) \cong (T^{2} \times [0,1]) \ \cup_{F} \ M \cong M,$$
contradicting the assumption of the theorem.  
\end{enumerate}

Thus $X^{-}$ and the multislopes $\mathcal{A}^{-}_{1}$
fulfill the requirements of the Solid Torus Cosmetic Surgery Theorem (\ref{thm:SlopesOnSolidTorus}), 
showing that the set of meridians of the solid tori $\{X^{-}(\alpha^{-}) \ | \ \alpha^{-} \in \mathcal{A}^{-}_{1} \}$, 
which we will denote as $B^{-}$, is bounded.

By the discussion above, if $X(\alpha) \cong M(\beta)$ then $\beta$ is the image of $\mu_{F} \in B^{-}$
under the diffeomorphism denoted above as $f^{+}:X^{+}(\alpha^{+}) \to M$.  We will denote the image
of $B^{-}$ under $f^{+}$ as $B^{-}_{f^{+}}$. By Lemma~\ref{lem:MCGisometryOfSlopes}
there is an isometric involution $i$ on the slopes of $F$ so that 
if $g^{+}:X^{+}(\alpha^{+}) \to M$ is a diffeomorphism, the image of $B^{-}$ under
$g^{+}$ is either $B^{-}_{f^{+}}$ or $i(B^{-}_{f^{+}})$.
We will denote $B^{-}_{f^{+}} \cup i(B^{-}_{f^{+}})$ as $B_{M}$.
Clearly, $B_{M}$ is bounded.  The theorem follows for JSJ manifolds,
as $\mathcal{B} \subset B_{f} \cup B_{1}^{+} \cup B_{M}$.

\bigskip\noindent
This completes the proof of Theorem~\ref{thm:fillingsThatDontFactor}.
\end{proof}

\section{Case Two: fillings of $E$ that do not factor through $T^{2} \times [0,1]$}
\label{sec:FillingsOfEthatDontFactorThroughTXI}

\bigskip\noindent
In this section we prove the following proposition, that is used in Case Two of the proof of Theorem~\ref{thm:main}:

\begin{prop}
\label{pro::FillingsOfEthatDontFactorThroughTXI}
Let $X$ and $E$ be compact orientable connected manifolds with toral boundary 
and $\phi:X \to E$ a branched cover so that $\phi|_{\partial X}:\partial X \to \partial E$
is a cover. Let $\hat{\mathcal{A}}$ be a set of multislopes of $X$ so that every 
$\hat\alpha \in \hat{\mathcal{A}}$ induces a multislope on $\partial E$ which we will
denote as $\alpha^{E}$.  Fix a component of $\partial X$ (which we will denote we $T_{1}$)
and denote $\phi(T_{1})$ as $T_{1}'$.  Suppose that every $\alpha^{E}$ fulfills the following conditions:
	\begin{enumerate}
	\item $E(\alpha^{E}) \cong S^{3}$.
	\item There does not exist $\alpha' \pf \alpha^{E}$ fulfilling the following two conditions:
		\begin{enumerate}
		\item  $\alpha'|_{T_{1}'} = \nos$.
		\item $E(\alpha') \cong T^{2} \times [0,1]$.
		\end{enumerate}
	\end{enumerate}

Then the set of restrictions $\{\hat\alpha|_{T_{1}} \ | \ \hat\alpha \in \hat{\mathcal{A}}\}$ is bounded.
\end{prop}

\begin{proof}
Let $\mathcal{A}_{E}$ be the set multislopes $\{ \alpha^{E} \}$ above.
By the $S^{3}$ Cosmetic Surgery  Theorem (\ref{thm:cosmeticSurgeryOnS3}), the set of restrictions 
$\{ \alpha^{E}|_{T_{1}'} \ | \ \alpha^{E} \in \mathcal{A}_{E}\}$ is bounded.  The set of restrictions 
$\{\hat\alpha|_{T_{1}} \ | \ \hat\alpha \in \hat{\mathcal{A}}\}$ is
contained in the image of $\{ \alpha^{E}|_{T_{1}'} \ | \ \alpha^{E} \in \mathcal{A}_{E}\}$
under the bilipschitz bijection induced
by $\phi$ (Lemma~\ref{lem:CorrespondingSlopesFeray}).
The proposition follows.
\end{proof}

\section{Case Three: fillings of $X$ that factor through $M$ and fillings of $E$ that factor through $T^{2} \times [0,1]$}
\label{sec:BothFillingsFactor}

\bigskip\noindent
In this section we tackle Case Three, the final case of Theorem~\ref{thm:main}.
Below we denote as $\mathcal{A}_{3}$ the multislopes in $\mathcal{A}$ that correspond to
this case.  After general analysis of the situation, we show that certain conditions must
hold (up-to finite ambiguity).  These conditions are summarized in Lemma~\ref{lem:CoresNotInBall}.
Theorem~\ref{thm:main} then follows from Proposition~\ref{pro:TheLastCase} and
Lemma~\ref{lem:BoundedSetsOnMathcalT}.

We begin by fixing our notation.   
Let $X$, M$(\beta)$, $E$, $L$, $\phi$, and $\hat\phi$ be as in the
diagram in Section~\ref{sec:SetUpOfProof} (see Page~\pageref{diagram}).  We will denote the 
multislope of $X$ that corresponds to the filling $X \to M(\beta)$ as $\alpha$ and the 
multislope of $E$ that corresponds to $E \to S^{3}$ as $\alpha^{E}$.  By construction, $\alpha$
and $\alpha^{E}$ are corresponding multislopes (in the corresponds defined by the restriction of 
$\phi:X \to E$ to the boundary; recall Subsection~\ref{subsection:FareyAndCovering}).
Up-to finite ambiguity we fix two components of $\partial E$ which we will
denote as $T_{1}'$ and $T_{2}'$ 
(note that since $E \to S^{3}$ is assumed to factor through $T^{2} \times [0,1]$, $|\partial E| \geq 2$)
and a component of $\phi^{-1}(T_{1}')$ which we will denote as $T_{1}$.
We will denote the remaining components of  $\partial X$ as $T_{2},\dots,T_{n}$ and the remaining components of
$\partial E$ as $T_{3}',\dots,T_{k}'$.  
In the two preceding sections we have reduced the proof of Theorem~\ref{thm:main}
to multislopes $\alpha \in \mathcal{A}$ fulfilling 
the following conditions:
	\begin{enumerate}
	\item $X(\alpha) \cong M(\beta)$ (for some slope $\beta$ of $\partial M$).
	\item $X(\alpha|_{T_{2},\dots,T_{n}}) \cong M$.
	\item $E(\alpha^{E}|_{T_{3}',\dots,T_{k}'}) \cong T^{2} \times [0,1]$.
	\end{enumerate}
Thus we obtain the following commutative diagram, where here horizontal
arrows represent inclusions induced by fillings, vertical arrows represent
covering projections, $T^{2} \times [0,1]$ represents $E(\alpha^{E}|_{T_{3}',\dots,T_{k}'})$,
$\phi^{-1}(T^{2} \times [0,1])$ is denoted as $X'$, and $\phi|_{X'}$ is denoted as $\phi'$:

\bigskip

\begin{center}
\begin{picture}(200,60)(0,0)
  \put(  0,  0){\makebox(0,0){$E$}}
  \put(  0,50){\makebox(0,0){$X$}}
  \put(  0, 40){\vector(0,-1){30}}   
  \put(160,  0){\makebox(0,0){$S^3,L$}}
  \put(160, 40){\vector(0,-1){30}}
  \put(10,28){\makebox(0,0){$/\phi$}}
  \put(70,28){\makebox(0,0){$/\phi'$}}
  \put(170,28){\makebox(0,0){$/{\hat{\phi}}$}}
  \put(108,28){\makebox(0,0)}
  \put(160,50){\makebox(0,0){$M(\beta)$}}
  \put(115,50){\makebox(0,0){$M$}}
  \put(60,50){\makebox(0,0){$X'$}}
  \put( 7,50){\vector(1,0){45}}
  \put( 67,50){\vector(1,0){40}}
  \put( 123,50){\vector(1,0){23}}
  \put(60, 40){\vector(0,-1){30}}   
  \put(7,0){\vector(1,0){33}}
  \put( 90,0){\vector(1,0){55}}
  \put(65,0){\makebox(0,0){$T^{2} \times [0,1]$}}  
\end{picture}
\end{center}
\bigskip	
We will denote the set of all multislopes fulfilling these conditions as $\mathcal{A}_{3} \subset  \mathcal{A}$.
For $\alpha \in \mathcal{A}_{3}$, we will denote  the restriction $\alpha|_{T_{2},\dots,T_{n}}$ as $\hat\alpha$, 
and the set or restriction $\{\hat\alpha \ | \ \alpha \in \mathcal{A}_{3}\}$ as $\hat{\mathcal{A}}$.

Our goal is to show that the set of slopes induced on $T_{1}$ by restricting the multislopes $\mathcal{A}_{3}$
is bounded.  However, in the course of the proof we sometimes end up with a bounded set of slopes
of a different component of $\partial X$, say $T_{2}$.  It will always be the case that $\phi(T_{2}) =T_{1}'$
or $\phi(T_{2}) = T_{2}'$.  If $\phi(T_{2}) = T_{1}'$, then the covering map $\phi$ induces
a bijection between the slopes of $T_{2}$ and those of $T_{1}'$, and a bijection between
the slopes of $T_{1}'$ and those of $T_1$; both bijections are bilipschitz maps
of the Farey graph (Lemma~\ref{lem:CorrespondingSlopesFeray}).
Composing the two bijections, we readily see a bounded set of slopes on $T_{1}$.
If instead $\phi(T_{2}) = T_{2}'$, let $B_{2}$ be the bounded set of slopes of $T_{2}$ and
let $B_{2}'$ be its image (in the slopes of $T_{2}'$) under the bijection induced by $\phi$;
by Lemma~\ref{lem:CorrespondingSlopesFeray}, $B_{2}'$ is bounded.
Since $E(\alpha^{E}|_{T_{3}',\dots,T_{k}'}) \cong T^{2} \times [0,1]$, we can then project $B_{2}'$ from $T_{2}'$ 
to a set of slopes on $T_{1}'$ using the product structure.  As this set depends on $\alpha^{E}$, we denote it as 
$B_{1,\alpha^{E}}'$.  By the $T^{2} \times [0,1]$ 
Cosmetic Surgery Theorem~(\ref{thm:CosmeticSurgeryOnT2XI}), 
$\cup_{\alpha^{E}} B_{1,\alpha^{E}}'$ is a bounded set of slopes on $T_{1}'$.
Using  Lemma~\ref{lem:CorrespondingSlopesFeray} once more, we obtain 
a bounded set of slopes on $T_{1}$.  

We summarize this:

\begin{lem}
\label{lem:BoundedSetsOnMathcalT}
With the notation above, suppose that for each component $T_{i}$ of $\phi^{-1}(T_{1}' \cup T_{2}')$,
there is a bounded set of slopes of $T_{i}$, that we will denote as $B_{i}$, so that that for each
$\alpha \in \mathcal{A}_{3}$, there is some component $T_{i}$ of $\phi^{-1}(T_{1}' \cup T_{2}')$ 
so that $\alpha|_{T_{i}} \in B_{i}$.
 
Then $\{ \alpha|_{T_{1}} \ | \ \alpha \in \mathcal{A}_{3}\}$  is bounded.
\end{lem}

\begin{rmk}
It is unfortunate that this set up does not lend itself well to induction.  There are many problems, and 
here is perhaps the best example: when considering a JSJ manifold $X$ we apply 
Proposition~\ref{prop:ObtainedByFilling} and conclude that $M(\beta)$ is obtained by filling some component of the torus
decomposition of $X$; however, the cover $\phi':X' \to E'$ is nowhere to be found.  We must therefore
first identify the essential information that is preserved in the inductive step.
This is done in the following two lemmas:
\end{rmk}

\begin{lem}
\label{lem:X'IsIrreducible}
Any prime factor of $X'$ has at least two boundary components.
\end{lem}

\begin{proof}
If $X'$ is prime then the lemma follows from the fact that  $T \times \{0\}$  and 
$T \times \{1\}$ (as components of $\partial E(\alpha^{E}|_{T_{3}',\dots,T_{k}'}) \cong T^{2} \times [0,1]$)
have at least
one preimage each.  Otherwise, let $S \subset X'$ be a decomposing sphere that realizes the decomposition 
$X' = X'' \# X'''$, where $X''$ is a prime factor of $X'$.  We will prove the lemma by showing that $|\partial X''| \geq 2$.
\begin{enumerate}
\item Suppose that $|\partial X''| = 0$.  Equivalently, $X''$ is closed.  Then $X''$ is a prime factor of $M$. 
But $M$ is a hyperbolic manifold and $\partial M \neq \emptyset$,  
and so it has no closed factors.  Thus $|\partial X''| \neq 0$. 
\item Suppose that $|\partial X''|=1$.  Without loss of generality we may assume that 
$\phi'(\partial X'') = T^{2} \times \{0\}$.
We will denote the component of $X'$ cut open along $S$ that
corresponds to $X''$ as $X^{*}$; note that $\partial X^{*}$ has two components: a torus (which we
can naturally identify with $\partial X''$) and a sphere (which we
can naturally identify with $S$).   Since $\pi_{2}(T^{2} \times [0,1])$ is trivial,
$\phi'|_{S}:S \to T^{2} \times [0,1]$ can be extended to a map from the closed
ball $D$, which we will denote as $\phi''$.  After a homotopy of $\phi'$ if necessary,
we assume as we may that for a sufficiently small fixed $\epsilon>0$
the following three conditions hold:
\begin{enumerate}
\item $\phi''(D) \cap (T^{2} \times [0,\epsilon)) = \emptyset$.  
\item $\phi''(D) \cap (T^{2} \times (1-\epsilon,1]) = \emptyset$.  
\item $\phi'(X^{*}) \cap (T^{2} \times (1-\epsilon,1]) = \emptyset$.  
\end{enumerate}

We paste $\phi'':D \to T^{2} \times [0,1]$ and $\phi':X^{*} \to T^{2} \times [0,1]$ to obtain a map
which we will denote as $\psi:X'' = X^{*} \cup_{S} D \to T^{2} \times [0,1]$.
By conditions~(a) and~(b) above the following holds:
$$\psi|_{\psi^{-1}(T^{2} \times [0,\epsilon))} = \phi'|_{\psi^{-1}(T^{2} \times [0,\epsilon))}.$$
Therefore 
$\psi|_{\psi^{-1}(T^{2} \times [0,\epsilon))}:\psi^{-1}(T^{2} \times [0,\epsilon)) \to T^{2} \times [0,\epsilon)$
is a cover and has non-zero degree.  On the other hand, $\psi^{-1}(T^{2} \times (1-\epsilon,1]) = \emptyset$ so
$\psi|_{\psi^{-1}(T^{2} \times (1-\epsilon,1])}:\psi^{-1}(T^{2} \times (1-\epsilon,1]) \to T^{2} \times (1-\epsilon,1]$ 
has degree zero.  This contradiction shows that $|\partial X''|\neq1$.
\end{enumerate}
\end{proof}

In light of this lemma, we define:

\begin{notation}
We will denote the prime factor of $X'$
that contains $T_{1}$ as $X''$ and the components of $\partial X'' \setminus T_{1}$
as $\mathcal{T}$.  Given  $\alpha \in \mathcal{A}_{3}$,
the link in $M$ consisting of the cores of the solid
tori attached to $\mathcal{T}$ will be denoted as $\mathcal{L}_{\alpha}$,
or simply $\mathcal{L}$ when no confusion may arise.
\end{notation}

Next we prove:

\begin{lem}
\label{lem:CoresNotInBall}
$\mathcal{L}$ fulfills:
	\begin{enumerate}
	\item $\mathcal{L} \neq \emptyset$.
	\item $E(\mathcal{L})$ is irreducible.
	\end{enumerate}
\end{lem}

\begin{proof}
By Lemma~\ref{lem:X'IsIrreducible}, $X''$
has at least two boundary components.
Hence $\mathcal{T} \neq \emptyset$, and~(1) follows.

If $X'$ is prime that by construction $\mathcal{T} = \partial X' \setminus T_{1}$; hence
$E(\mathcal{L}) \cong X'$ and~(2) follows.  Otherwise,
recall the definition of $S$ and $X^{*} \subset X'$ from the construction of $\psi$ above.
Let $X^{**} = \mbox{cl}(X' \setminus X^{*})$.  Since $M$ is hyperbolic the $\partial M$ is the 
image of $T_{1} \subset X^{*}$, after filling $X^{**}$ we obtain a ball.  This shows that
$E(\mathcal{L}) \cong X^{*}$; by constrcution $X^{*}$ is a factor of the prime decomposition
of $X'$ and hence is itself prime. 

\end{proof}

\bigskip\noindent

Thus, for every $\alpha \in \mathcal{A}_{3}$, there is $I \subset \{1,\dots,n\}$,
so that the following three conditions are satisfied:
\begin{enumerate}
\item $1 \not\in I \neq \emptyset$.
\item $\{T_{i}\}_{i \in I}$ (which we will denote as $\mathcal{T}$) is a union of  components of  $\phi^{-1}(T_{1}' \cup T_{2}')$.
\item The cores of the solid tori attached to $\mathcal{T}$form a link (which we will denote as 
$\mathcal{L} \subset M$) with an irreducible exterior.  
\end{enumerate}
Up-to finite ambiguity we fix $I$ as above and consider only $\alpha \in \mathcal{A}_{3}$ 
that fulfill these conditions for the fixed index set $I$.
To avoid overly complicated notation we do not rename $\mathcal{A}_{3}$ or $\hat{\mathcal{A}}$.

\bigskip\noindent
We are now ready to state and prove the main proposition of this section:

\begin{prop}
\label{pro:TheLastCase}
Let $M$ be a hyperbolic manifold, $\partial M$ a single torus that we will denote as $T$,
$X$ a compact connected orientable manifold, $\partial X$ tori that we will denote as 
$\{T_{i}\}_{i=1}^{n}$, and
$\mathcal{A}_{3}$ a set of multislopes on $\partial X$.
Fix a non empty index set $I \subset \{2,\dots,n\}$; we will denoted $\{T_{i}\}_{i \in I}$ as $\mathcal{T}$.
Denote the link formed by the core of the solid tori attached to $\mathcal{T}$ when filling along $\alpha \in \mathcal{A}$
as $\mathcal{L}_{\alpha} \subset M$ (or simply $\mathcal{L}$ when no confusion may arise).
Assume that any $\alpha \in \mathcal{A}_{3}$ fulfills the following conditions:
	\begin{enumerate}
	\item $\alpha|_{T_{1}} = \nos$.
	\item $X(\alpha) \cong M$.
	\item $E(\mathcal{L}_{\alpha})$ is irreducible.
	\end{enumerate}
Then for each $i \in I$, there is a bounded set $B_{i}$ of slopes of $T_{i}$, 
so that for each $\alpha \in \mathcal{A}_{3}$,
there is an $i \in I$, so that $\alpha|_{T_{i}} \in B_{i}$.
\end{prop}

\begin{rmk}
\label{rmk:LastCase}
By the discussion above (in particular~\ref{lem:CoresNotInBall}), 
$X$, $T_{1}$, $\mathcal{T}$, and the multislopes of $\mathcal{A}$ considered in Case Three
of Theorem~\ref{thm:main} satisfy the conditions of the Proposition~\ref{pro:TheLastCase} for some $I$.
By Lemma~\ref{lem:BoundedSetsOnMathcalT},
proving this proposition will complete the proof of Theorem~\ref{thm:main}.
\end{rmk}

\begin{proof}[Proof of Proposition~\ref{pro:TheLastCase}]
We induct on $|T(X)|$.

\bigskip\noindent 
{\bf Assume that $X$ is Seifert fibered or sol.}  Then no filling of $X$ gives $M$.
We assume from now on that $X$ is not Seifert fibered or sol.

\bigskip\noindent 
{\bf Assume that $X$ is hyperbolic.} 
Fix $\epsilon>0$ smaller than the length of the shortest geodesic in $M$.
By Theorem~\ref{thm:HyperbolicFillingShortGeos} there are only finitely many
totally hyperbolic fillings on $X$ yielding $M$.  Given ${i} \in I$, we will denote the
set of slopes obtained by restricting totally hyperbolic multislopes of $\mathcal{A}$
to $T_{i}$ as $B_{i}^{1}$, whenever the restriction does not equal $\nos$.  
Then $B_{i}^{1}$ is finite (and hence bounded).  
We will denote as $\mathcal{A}' \subset \mathcal{A}_{3}$ the set
$$\mathcal{A}' = \{\alpha \in \mathcal{A}_{3} \ | \ (\forall i \in I) \ \alpha|_{T_{i}} \not\in B_{i}^{1}\}.$$ 
From this point on we assume as we may that $\alpha \in \mathcal{A}'$. 
Then $\alpha$ is not totally hyperbolic and hence admits a \mnh\ partial
filling.  For each $i \in I$, let $B_{i}^{2}$ be the set of restrictions $\{ \alpha_{\mbox{\tiny min}}|_{T_{i}}\}$,
where $\alpha_{\mbox{\tiny min}}$ ranges over all \mnh\ fillings 
for which $\alpha_{\mbox{\tiny min}}|_{T_{i}} \neq \infty$.
By Proposition~\ref{pro:mnhIsFinite} $X$ admits only finitely many \mnh\ fillings
and so $B_{i}^{2}$ is finite (and hence bounded).
We will denote as $\mathcal{A}'' \subset \mathcal{A}'$ the set 
$$\mathcal{A}'' = \{\alpha \in \mathcal{A}_{3} \ | \ (\forall i \in I) \ \alpha|_{T_{i}} \not\in (B_{i}^{1} \cup B_{i}^{2})\}.$$
From this point on we assume as we may that $\alpha \in \mathcal{A}''$.  By deifinition,
any $\alpha \in \mathcal{A}''$ admits a \mnh\ partial filling $\alpha_{\mbox{\tiny min}}$ so that
$\alpha_{\mbox{\tiny min}}|_{T_{i}} = \infty$ for every ${i} \in I$.  

Fix a \mnh\ filling $\alpha_{\mbox{\tiny min}}$ for which
$\alpha_{\mbox{\tiny min}}|_{T_{i}} = \infty$ for every ${i} \in I$.  
We will denote $X(\alpha_{\mbox{\tiny min}})$ as $X_{1}$ and the 
the restrictions to $X_{1}$ of multislopes in $\mathcal{A}''$ that admit a partial filling 
$\alpha_{\mbox{\tiny min}}$ as $\mathcal{A}_{\alpha_{\mbox{\tiny min}}}$, that is:
$$\mathcal{A}_{\alpha_{\mbox{\tiny min}}} = 
\{\alpha|_{\partial X_{1}} \ | \ \alpha \in \mathcal{A}'' \mbox{ and } \alpha_{\mbox{\tiny min}} \pf \alpha   \}.$$
We claim that the following conditions are satisfied:
\begin{enumerate}
\item $|T(X_{1})| < |T(X)|$: this holds since $X_{1}$ corresponds to a direct descendant of the root of $T(X)$.
\item $\mathcal{T} \subset \partial X_{1}$: this follows from the definition of $\mathcal{A}''$,
where we required that $\alpha|_{T_{i}} = \nos$ (for all $i \in I$).
\item For any $\alpha_{1} \in \mathcal{A}_{\alpha_{\mbox{\tiny min}}}$, $\alpha_{1}|_{T_{1}} = \nos$:
it follows from the definitions that  $\alpha_{1}|_{T_{1}} = \alpha|_{T_{1}}$;
hence by the first assumption of the proposition, $\alpha|_{T_{1}} = \nos$.
\item For any $\alpha_{1} \in  \mathcal{A}_{\alpha_{\mbox{\tiny min}}}$, $X_{1}(\alpha_{1}) \cong M$:
by definition, $X_{1}(\alpha_{1}) = X(\alpha_{\mbox{\tiny min}})(\alpha_{1}) = X(\alpha)$.
By the second assumption of the proposition, $X(\alpha) \cong M$.
\item For any $\alpha_{1} \in  \mathcal{A}_{\alpha_{\mbox{\tiny min}}}$, 
$E(\mathcal{L}_{\alpha_{i}})$ is irreducible, where here 
$\mathcal{L}_{\alpha_{1}}$ denotes the link formed by the cores of the solid tori attached to 
$\mathcal{T} \subset \partial X_{1}$: it is straightforward to see that 
$E(\mathcal{L}_{\alpha_{1}}) = E(\mathcal{L})$.
By the third assumption of the proposition, $E(\mathcal{L})$ is irreducible.
\end{enumerate}
By~(2)--(5),  $X_{1}$, $T_{1}$, $\mathcal{T}$, and $\mathcal{A}_{\alpha_{\mbox{\tiny min}}}$ fulfill 
the assumptions of the proposition.  By~(1) we may apply induction, showing that 
for every $i \in I$, there is a bounded set of slopes of $T_{i}$, that we will denote as $B_{i,\alpha_{\mbox{\tiny min}}}^{3}$, 
so that for each $\alpha_{1} \in \mathcal{A}_{\alpha_{\mbox{\tiny min}}}$,  there is an ${i} \in I$, 
so that $\alpha_{1}|_{T_{i}} \in B_{i,\alpha_{\mbox{\tiny min}}}^{3}$.  

With the notation of the preceding paragraph, every $\alpha \in \mathcal{A}''$ admits a 
\mnh\ filling $\alpha_{\mbox{\tiny min}}$
and $\alpha|_{T_{i}} = \alpha_{1}|_{T_{i}}$ for every $i \in I$.  Hence for some $i \in I$, 
$\alpha|_{T_{i}} \in B_{i,\alpha_{\mbox{\tiny min}}}^{3}$.
By Proposition~\ref{pro:mnhIsFinite}, $X$ admits only finitely many \mnh\ fillings.  Hence the set
$$B_{i}^{3} = \bigcup_{\alpha_{\mbox{\tiny min}}} B_{i,\alpha_{\mbox{\tiny min}}}^{3}$$ 
is bounded.  The proposition follows by setting 
$$B_{i} = B_{i}^{1} \cup B_{i}^{2}\cup B_{i}^{3}.$$

We assume from now on that $X$ is not Seifert fibered, sol, or hyperbolic.

\bigskip\noindent 
{\bf Assume that $X$ is not prime.} Let $X_{1}$ be the factor of the prime decomposition
of $X$ that contains $T_{1}$; say $X = X_{1} \# X_{1}'$ (we are not assuming that
$X_{1}'$ is prime).  Then any $\alpha \in \mathcal{A}$ induces the multislopes
$\alpha_{1} = \alpha|_{\partial X_{1}}$ and $\alpha_{1}' = \alpha|_{\partial X'_{1}}$
on $\partial X_{1}$ and $\partial X_{1}'$, respectively.
Since $X_{1}(\alpha_{1}) \# X_{1}'(\alpha_{1}') = X(\alpha) \cong M$ and $M$ is hyperbolic, 
the following conditions hold:
	\begin{enumerate}
	\item  Since $T_{1} \subset \partial X_{1}$, $X_{1}(\alpha_{1}) \cong M$ (with $\alpha_{1}|_{T_{1}} = \nos$).
	\item  $X_{1}'(\alpha_{1}') \cong S^{3}$.
	\end{enumerate}
By construction the reducing sphere that gives the decomposition $X = X_{1} \# X_{1}'$ is disjoint from $\mathcal{L}$.	
Since $E(\mathcal{L})$ is irreducible, no component of it is contained in $X_{1}'(\alpha_{1}')$;
equivalently, $\mathcal{T} \subset \partial X_{1}$.
It is easy to see that $X_{1}$, $T_{1}$, $\mathcal{T}$, and $\{\alpha_{1} \ | \ \alpha \in \mathcal{A} \}$ fulfill the 
assumptions of the proposition.  
Since $X_{1}$ corresponds to a direct descendant of the root of $T(X)$, $|T(X_{1})| < |T(X)|$.
By induction, for each ${i} \in I$,
there is a bounded set of slopes $B_i$ of $T_{i}$, so that for each $\alpha \in \mathcal{A}$,
there is an ${i} \in I$ with $\alpha_{1}|_{T_{1}} \in B_{i}$.  Since for all $i \in I$, $\alpha|_{T_{i}} = \alpha_{1}|_{T_{i}}$,
the proposition follows in this case.

We assume from now on that $X$ is prime and not Seifert fibered, sol, or hyperbolic.

\bigskip\noindent 
{\bf Assume that $X$ is JSJ.}  Let $X_{0}$ be the component of the torus decomposition of $X$ that contains $T_{1}$.
Denote the components of $\partial X_{0} \setminus T_{1}$ as $\{F_{j}\}_{j=1}^{k}$.  Denote the component of 
$\mbox{cl}(X \setminus X_{0})$ containing $F_{j}$ as $X_{j}$, see Figure~\ref{fig:F1}.  
\begin{figure}
\includegraphics[height=2.5in]{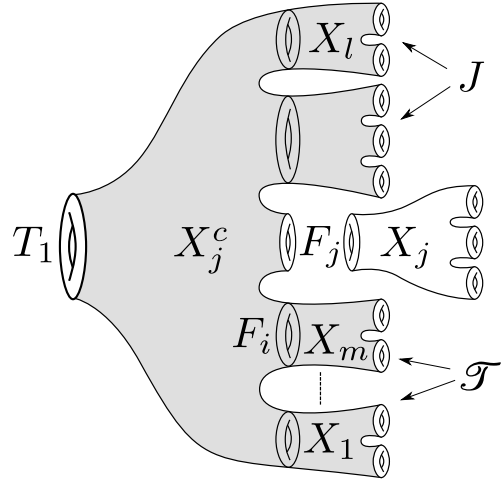}
\caption{The notation when $X$ is JSJ (the shaded region is $X_{j}^{c}$)}
\label{fig:F1}
\end{figure}  
To avoid the situation where $X_{j} = \emptyset$, if $F_{j} \subset \partial X$ we push it slightly into the interior of $X$ 
so that $X_{j} \cong T^{2} \times [0,1]$ in that case.  Since $M$ is hyperbolic
it admits no non separating tori; thus we assume as we may that  $X_{j} \neq X_{j'}$
for $j \neq j'$, for otherwise no filling of $X$ yields $M$.  By renumbering if necessary,
we assume as we may that $\partial X_{j}$ contains a component of $\mathcal{T}$ exactly when $j \leq m$,
for an appropriately chosen $m$.

For $\alpha \in \mathcal{A}$ we will denote the restriction $\alpha|_{\partial X_{j}}$ as $\alpha_{j}$ 
(by definition $\alpha_{j}|_{F_{j}} = \nos$).  Since $X(\alpha) \cong M$ is hyperbolic, every
torus in $X(\alpha)$ is either boundary parallel ({\bf A} below), 
or bounds a solid torus ({\bf B} and {\bf C} below), or bounds a non trivial knot exterior in a ball ({\bf D} below).  
Thus for every $1 \leq j \leq k$ exactly one of the following holds:
	\begin{description}
	\item[Case~{\bf A}] $X_{j}(\alpha_{j}) \cong M$ and 
	$\mbox{cl}(X(\alpha) \setminus X_{j}(\alpha_{j})) \cong T^{2} \times [0,1]$.
	\item[Case~{\bf B}] $X_{j}(\alpha_{j}) \cong D^{2} \times S^{1}$ and no component of $\mathcal{L} \cap X_{j}(\alpha_{j})$
	is a core of $X_{j}(\alpha_{j})$.
	\item[Case~{\bf C}] $X_{j}(\alpha_{j}) \cong D^{2} \times S^{1}$ 
	and some component of $\mathcal{L} \cap X_{j}(\alpha_{j})$, which we will denote as $K_{j}$, 
	is a core of $X_{j}(\alpha_{j})$.  
	\item[Case~{\bf B}] $X_{j}(\alpha_{j}) \cong E(K_{j})$ for a non-trivial knot $K_{j} \subset S^{3}$
	and $X_{j}(\alpha_{j}) \subset D_{j}$
	for some ball $D_{j} \subset X(\alpha)$.
	\end{description}

\bigskip

\noindent We first consider the following:

\bigskip\noindent
{\bf Case~{\bf A} happens for some $j \leq m$.}
Fix $j$ ($1 \leq j \leq m$) and let $\mathcal{A}_{j}' \subset \mathcal{A}$ be
$$\mathcal{A}_{j}' = \{\alpha \in \mathcal{A} \ | \ X_{j}(\alpha_{j}) \cong M \mbox\}.$$
Note that for any $\alpha \in \mathcal{A}$, if $F_{j}$ is as in Case~{\bf A} above then
$\alpha \in \mathcal{A}_{j}'$.  We will denote the set of restrictions 
$\{\alpha_{j} \ | \ \alpha \in \mathcal{A}_{j}' \}$ as $\mathcal{A}_{j}$.
Let $\mathcal{T}_{j} = \mathcal{T} \cap X_{j}$ and $\mathcal{L}_{j} = \mathcal{L} \cap X_{j}(\alpha_{j})$;
equivalently, $\mathcal{L}_{j}$ is the link formed by the cores of the solid tori attached to $\mathcal{T}_{j}$.
Since $j \leq m$, $\mathcal{T}_{j} \neq \emptyset$. 
If, for some $\alpha_j \in \mathcal{A}_{j}$, $\mathcal{L}_{j}$ were reducible then any reducing sphere for 
$X_{j}(\alpha_{j}) \setminus \mbox{int}N(\mathcal{L}_{j})$
would be a reducing sphere for $\mathcal{L}$; this contradicts the assumptions of the proposition.

Hence $X_{j}$, $F_{j}$, $\mathcal{T}_{j}$ and $\mathcal{A}_{j}$ fulfill the assumptions of the proposition.  
By Lemma~\ref{lem:TreeOfJSJ}, $|T(X_{j})|<|T(X)|$.
By induction on each component $T$ of $\mathcal{T}_{j}$,
there is a bounded set of slopes that we will denote as $B_{T}$, 
so that for each $\alpha_{j} \in \mathcal{A}_{j}$, 
there is a component $T$ of $\mathcal{T}_{j}$,
so that $\alpha_{j}|_{T_{i}} \in B_{T}$.  It follows that for every $\alpha \in \mathcal{A}_{j}'$,  
$\alpha|_{T} \in B_{T}$ for some $T$; the proposition follows in this case.  

We may consider from now on only multislopes from 
$\mathcal{A} \setminus (\cup_{j \leq m} \mathcal{A}_{j}')$ (and in particular, we
assume as we may that Case~{\bf A} above
does not happen for $j \leq m$).  To avoid overly complicated notation we do not rename $\mathcal{A}$.

\bigskip\noindent
Next we consider the following:

\bigskip\noindent
{\bf Case~{\bf A} happens for some $j \geq m+1$.}   Fix $j$ ($m+1 \leq j \leq k$).
We will denote $\mbox{cl}(X \setminus X_{j})$ as $X_{j}^{c}$
and the restriction $\alpha|_{\partial X_{j}^{c}}$ as $\alpha_{j}^{c}$.  
Since $j \geq m+1$, $\mathcal{T} \subset \partial X_{j}^{c}$.
Let $\mathcal{A}_{j}' \subset \mathcal{A}$ be the set:
$$\mathcal{A}_{j}' = \{\alpha \in \mathcal{A} \ | \ X_{j}^{c}(\alpha^{c}_{j}) \cong T^{2} \times [0,1]\}.$$
Note that for any $\alpha \in \mathcal{A}$, $F_{j}$ is as in Case~{\bf A} above if and only if
$\alpha \in \mathcal{A}_{j}'$.  
Given $\alpha \in \mathcal{A}'_{j}$ we will denote the restriction $\alpha|_{X_{j}^{c}}$
as $\alpha^{c}_{j}$ (by definition $\alpha_{j}^{c}|_{F_{j}}$ and $\alpha_{j}^{c}|_{T_{1}}$ are $\nos$).
We will denote the set of restrictions $\{\alpha_{j}^{c} \ | \ \alpha \in \mathcal{A}_{j}' \}$
as $\mathcal{A}_{j}^{c}$.  

Thus we have a manifold $X_{j}^{c}$ so that $\mathcal{T} \subset \partial X_{j}^{c}$ and 
multislopes $\mathcal{A}_{j}^{c}$ so that for every $\alpha_{j}^{c} \in \mathcal{A}_{j}^{c}$,
$X_{j}^{c}(\alpha_{j}^{c}) \cong T^{2} \times [0,1]$ 
and  $\partial X_{j}^{c}(\alpha_{j}^{c}) = T_{1} \cup F_{j}$.
Up-to finite ambiguity we fix $J \subset \{1,\dots,k\}$, 
and denote as $\mathcal{A}_J^{c} \subset \mathcal{A}_{j}^{c}$ the multislopes  
for which $X_{j'}(\alpha_{j'}) \cong E(K_{j'}) \subset D_{j'}$ for a non-trivial knot 
$K_{j'} \subset S^{3}$ and a ball $D_{j'} \subset X(\alpha)$ 
if and only if $j' \in J$.  
By Lemma~\ref{lem:KnotExteriorsAreDisjoint} we may assume that the balls 
$\{D_{j'}\}_{j' \in J}$ are disjointly embedded.  

\begin{figure}[h]
\includegraphics[height=2.5in]{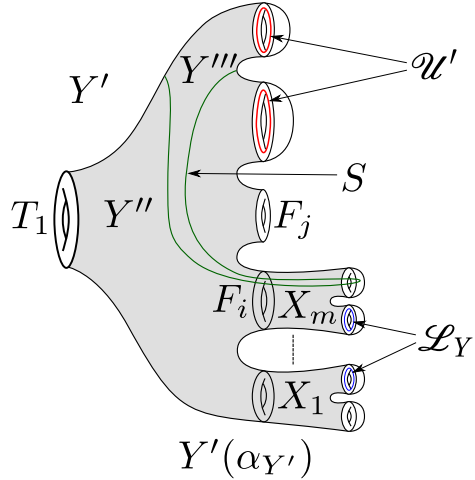}
\caption{$Y'(\alpha_{Y'})$}
\label{fig:F2}
\end{figure}
Let $Y'$ be the closure of $X_{j}^{c} \setminus (\cup_{j' \in J} E(K_{j'}))$ (see Figure~\ref{fig:F2};
in that figure $\{F_{j'}\}_{j' \in J}$ are the two tori on the top right).
By restriction, any $\alpha_{j}^{c} \in \mathcal{A}_{J}^{c}$ induces
a slope on each component of 
$\partial Y' \setminus (T_{1}  \cup F_{j} \cup (\cup_{j' \in J} F_{j'}))$.
For every $j' \in J$, we pick a slope on $F_{j'}$ that intersects the meridian of $E(K_{j'})$ exactly once.
(There are infinitely many way to do this; we will exploit this flexibility soon 
when appealing to Lemma~\ref{lem:TwistingToGetIrreducible}.)
Denote the multislope obtained as $\alpha_{Y'}$, and note that $Y'(\alpha_{Y'}) \cong T^{2} \times [0,1]$
and $\partial Y'(\alpha_{Y'}) = T_{1} \cup F_{j}$.  We will denote $\mathcal{T} \cap \partial Y'$
as $\mathcal{T}_{Y}$, the link formed by the cores of the solid tori attached to $\mathcal{T}_{Y}$
as $\mathcal{L}_{Y}$, and the link formed by the cores of the solid tori attached to $\cup_{j' \in J}F_{j'}$
as $\mathcal{U}'$ (in Figure~\ref{fig:F2}, $\mathcal{U}'$ is red).  
Since the components of $\mathcal{U}'$ are unknots contained in the disjointly
embedded balls $D_{j'}$, $\mathcal{U}'$ is an unlink.
Finally assume, for a contradiction, that $\mathcal{T}_{Y} = \emptyset$.
Then $\mathcal{L} \subset \cup_{j' \in J} D_{j'}$;
this contradicts the assumption that $E(\mathcal{L})$ is irreducible,
showing that $\mathcal{T}_{Y} \neq \emptyset$.

If $Y'(\alpha_{Y'}) \setminus \mbox{int}N(\mathcal{L}_{Y} \cup \mathcal{U}')$ is irreducible we denote
$Y'$ as $Y$, $\mathcal{U}'$ as $\mathcal{U}$, and $\alpha_{Y'}$ as $\alpha_{Y}$.  
Otherwise, let $S$ be a reducing sphere realizing the decomposition of 
$Y'(\alpha_{Y'}) \setminus \mbox{int}N(\mathcal{L}_{Y} \cup \mathcal{U}')$ as $Y'' \# Y'''$, where $Y''$ is irreducible and
$T_{1} \subset Y''$ (in Figure~\ref{fig:F2}, $S$ is green).  
Note that $Y'(\alpha_{Y'}) \setminus \mbox{int}N(\mathcal{L}_{Y} \cup \mathcal{U}')$ is obtained
from $Y'$ by filling $\partial Y' \setminus (T_{1} \cup F_{j} \cup \mathcal{T}_{Y} \cup (\cup_{j' \in J} F_{j'}))$
(in Figure~\ref{fig:F2} the components that are {\it not} filled are $T_{1}$, $F_{j}$, the reds, 
and the blues).  
Since $\mathcal{T} \subset \mathcal{T}_{Y} \cup (\cup_{j' \in J} \partial X_{j'})$,
$\mathcal{L} \cap S = \emptyset$.
We consider $S$ as a sphere in 
$$Y'(\alpha_{Y'}) \cup_{F_{j}} X_{j}(\alpha_{j}) \cong T^{2} \times [0,1] \ \cup_{F_{j}} \  M \cong M.$$ 
Since $M$ is hyperbolic, $S$ bounds a ball in $ Y'(\alpha_{Y'}) \cup_{F_{j}} X_{j}(\alpha_{j})$ 
which we will denote as $D$.  
Since $T_{1} \subset Y''$, $D = Y'''(\alpha_{Y'}|_{Y'''})$.
Clearly, $\mathcal{T}_{Y} \cap D = \emptyset$, for otherwise $S$ would be a reducing sphere for 
$E(\mathcal{L})$.  We will denote as $Y$ the manifold obtained from $Y'$ by filling 
$(\cup_{j' \in J} F_{j'}) \cap D$ along the multislope induced by $\alpha_{Y'}$.  
We will denote the multislope $\alpha_{Y'} | _{\partial Y}$ as $\alpha_{Y}$
and $\mathcal{U}' \setminus (\mathcal{U}' \cap D)$ as $\mathcal{U}$.
By construction $\mathcal{T}_{Y} \subset \partial Y$, and therefore 
we may consider  $\mathcal{L}_{Y}$ as the link in $Y(\alpha_{Y})$
formed by the cores of the solid tori attached to $\mathcal{T}_{Y}$.   Then the following two conditions hold:
	\begin{enumerate}
	\item $\mathcal{U}$ is an unlink.
	\item $Y(\alpha|_{Y}) \setminus \mbox{int}N(\mathcal{L}_{Y} \cup \mathcal{U}) \cong Y''$ and hence
	is irreducible.
	\end{enumerate}
We choose the slopes of $\mathcal{U}$, 
as we know we may by Lemma~\ref{lem:TwistingToGetIrreducible}, 
so that $E(\mathcal{L}_{Y})$ is irreducible.  

\begin{lem}
\label{lem:no-core}
For every $\alpha_{Y} \in \mathcal{A}_{Y}$, there is a slope $\alpha'$ of
$F_{j}$, so that $Y(\alpha_{Y})(\alpha')$ satisfies the following condition:	
	\begin{enumerate}
	\item $Y(\alpha_{Y})(\alpha') \cong D^{2} \times S^{1}$.
	\item $\mathcal{L}_{Y} \subset Y(\alpha_{Y})(\alpha')$ is irreducible. 
	\item No component of $\mathcal{L}_{Y}$ is a core of $Y(\alpha_{Y})(\alpha')$ .
	\end{enumerate}
\end{lem}

\begin{proof}[Proof of Lemma~\ref{lem:no-core}]
Fix $\alpha_Y \in \mathcal{A}_{Y}$.

By construction, $Y(\alpha_{Y}) \cong T^{2} \times [0,1]$;
hence for any slope $\alpha'$,
$Y(\alpha_Y)(\alpha') \cong D^{2} \times S^{1}$.
Thus~(1) is satisfied by any slope $\alpha'$ of $F_{j}$.

Since the exterior of $\mathcal{L}_{Y}$ as a link in $Y(\alpha_{Y})$ is
irreducible, by Hatcher, there is a finite set of slopes of $F_{j}$, which we will denote as $B_{f}$,
so that for any slope $\alpha' \not\in B_{f}$, the exterior of $\mathcal{L}_{Y}$
as a link in $Y(\alpha_{Y})(\alpha')$ is irreducible.

For~(3) we fix a component $K$ of $\mathcal{L}_{Y}$.  Let $[K]$ denote the homology
class represented by $K$ in  $H_{1}(Y(\alpha_{Y});\mathbb Z)$ ($[K]$ is only defined up to sign).  
We consider two possibilities:
	\begin{description}
	\item[{$[K]$} is not primitive] Then by Lemma~\ref{lem:CoresOfSolidTori}, $K$ is not a core of 
	of $Y(\alpha_{Y})(\alpha')$ for any slope $\alpha'$; we set $B_{K} = \emptyset$.
	\item[{$[K]$} is primitive] 	By Lemma~\ref{lem:CoresOfSolidTori}, if $K$ is a core
	of $Y(\alpha_{Y})(\alpha')$ then $[K]$ and $[\alpha']$ generate
	$H_{1}(Y(\alpha_{Y}))$.  We will denote as $B_{K}$ the set of slopes
	of $F_{j}$ that correspond to homology classes fulfilling this condition.  It is easy to see that
	$B_{K}$ has diameter $2$ in the Farey graph of the slopes of $F_{j}$.
	\end{description}
Since $B_{f} \cup (\cup_{K} B_{K})$ is a finite union of bounded sets it is itself bounded;
hence its complement is not empty.  The Lemma~\ref{lem:no-core} follows by picking 
$\alpha' \not\in B_{f} \cup (\cup_{K} B_{K})$.
\end{proof}

For each $\alpha_{Y} \in \mathcal{A}_{Y}$ we pick a slope $\alpha'$ of $F_{j}$
satisfying the conditions of Lemma\ref{lem:no-core}.
By Proposition~\ref{pro:SolidTorusSurgery2}, on every component $T$
of $\mathcal{T}_{Y}$, there exists a bounded set $B_{T}$, so for that every $\alpha_{Y} \in \mathcal{A}_{Y}$,
there is a component $T$ of $\mathcal{T}_{Y}$, so that $\alpha_{Y}|_{T} \in B_{T}$.
Given any $\alpha \in \mathcal{A}$, we construct 
$\alpha_{j}^{c} \in \mathcal{A}_{J}^{c}$ and $\alpha_{Y}$ as above.
The proposition follows in this case since $\alpha|_{T} = \alpha_{j}^{c}|_{T} = \alpha_{Y}|_{T} \in B_{T}$.

We may consider from now on only multislopes from 
$\mathcal{A} \setminus (\cup_{j \geq m+1} \mathcal{A}_{j}')$.
Thus from now on we will only consider multislopes for which Case~{\bf A} 
does not happen for any $j$.
To avoid overly complicated notation
we do not rename $\mathcal{A}$.

\bigskip\noindent
Next we consider the following:

\bigskip\noindent
{\bf Case~{\bf B} happens for some $j \leq m$.}  Fix $1 \leq j \leq m$.
We will denote as $\mathcal{A}_{j}' \subset \mathcal{A}$  
the multislopes $\alpha \in \mathcal{A}$ for which
$X_{j}(\alpha_{j}) \cong D^{2} \times S^{1}$ and no component
of $\mathcal{L} \cap X_{j}(\alpha_{j})$ is a core.  
Note that Case~{\bf B} occurs if and only if $\alpha \in \mathcal{A}_{j}$.

We will denote the set of restrictions $\{\alpha_{j} \ | \ \alpha \in \mathcal{A}_{j}' \}$ as $\mathcal{A}_{j}$
and $\mathcal{T} \cap X_{j}$ as  $\mathcal{T}_{j}$.
Since $j \leq m$, $\mathcal{T}_{j} \neq \emptyset$.  
Given $\alpha_{j} \in \mathcal{A}_{j}$,
we will denote the link formed by the cores of the solid tori attached
to $\mathcal{T}_{j}$ as $\mathcal{L}_{j}$.  It is easy to see that
if $X_{j}(\alpha_{j}) \setminus \mbox{int}N(\mathcal{L}_{j}$) were reducible 
then $E(\mathcal{L})$ would be reducible, 
contradicting the third assumption of the proposition.  By the assumption of Case~{\bf B}, no
component of $\mathcal{L}_{j}$ is a core of the solid torus $X_{j}(\alpha_{j})$.
Therefore by Proposition~\ref{pro:SolidTorusSurgery2}, for each component $T$ of $\mathcal{T}_{j}$,
there is a bounded set of slopes of $T$ which we will denote as $B_{T}$, so that for each $\alpha_{j} \in \mathcal{A}_{j}$, 
there is a component $T$ of $\mathcal{T}_{j}$ so that $\alpha_{j}|_{T} \in B_{T}$.  
The proposition follows in this case since for any $\alpha \in \mathcal{A}_{j}'$, 
$\alpha|_{T} = \alpha_{j}|_{T}$.

We may consider from now on only multislopes from 
$\mathcal{A} \setminus (\cup_{j \leq m} \mathcal{A}_{j}')$.
To avoid overly complicated notation
we do not rename $\mathcal{A}$.

\bigskip\noindent
We have reduced the proof to the following:

\bigskip\noindent
{\bf Cases~{\bf A} never happens and Case~{\bf B} never happens for $j \leq m$.}  
Consider $J_1,J_2 \subset \{1,\dots,k\}$ fulfilling the following conditions:
	\begin{enumerate}
	\item $\emptyset \neq J_{1} \subset \{1,\dots,m\}$.
	\item $J_{1} \cap J_{2} = \emptyset$.
	\item $\{1,\dots,m\} \subset J_{1} \cup J_{2}$.
	\end{enumerate}
Let $\mathcal{A}_{J_{1},J_{2}} \subset \mathcal{A}$ be the multislopes $\alpha$ fulfilling the following conditions:
	\begin{enumerate}
	\item For every $j \in J_{1}$, $X_{j}(\alpha_{j}) \cong D^{2} \times S^{1}$.  
	\item For any $1 \leq j \leq k$, $X_{j}(\alpha_{j}) \cong E(K_{j})$ (for a non-trivial knot $K_{j} \subset S^{3}$)
	if and only if $j \in J_{2}$.
	\end{enumerate}
Using the fact the Case~{\bf A} does not happen (that is, every $F_{j}$ bounds a solid torus or
a knot exterior contained in a ball) and irreducibility of $\mathcal{L}$ and 
Lemma~\ref{lem:KnotExteriorsAreDisjoint} (which together imply that $J_{1} \neq \emptyset$),
it is easy to see that for any $\alpha \in \mathcal{A}$ there is a 
choice of $J_{1},J_{2}$ as above for which $\alpha \in \mathcal{A}_{J_{1},J_{2}}$; thus 
   $$\mathcal{A} = \bigcup_{J_{1},J_{2}} \mathcal{A}_{J_{1},J_{2}}.$$
Up-to finite ambiguity we fix $J_1,J_2 \subset \{1,\dots,k\}$ fulfilling the conditions above
and consider only multislopes from $\mathcal{A}_{J_{1},J_{2}}$.

Fix a multislope $\alpha \in \mathcal{A}_{J_{1},J_{2}}$.

By Lemma~\ref{lem:KnotExteriorsAreDisjoint}
we may fix disjointly embedded balls $\{ D_{j}\}_{j \in J_{2}}$ so that $E(K_{j}) \subset D_{j}$.
Since Case~{\bf B} does not happen for $j \leq m$, for every $j \in J_1$, at least one component of 
$\mathcal{L}$ is a core of $X_{j}(\alpha_{j})$; we choose one and denote it as $K_{j}$.   
We will denote $\cup_{j \in J_{1}} K_{j}$ as $\mathcal{L}_{1}$ and $\cup_{j \in J_{1}} T_{j}$ 
as $\mathcal{T}_{1}$.

The multislope $\alpha$ induces a multislope on $\partial X_{0}$
as follows: for$j \not\in J_{2}$, the slope induced on $F_{j}$ is the meridian of solid torus
$X_{j}(\alpha_{j})$.
For $j \in J_{2}$, the slope induced on $F_{j}$ is any slope that intersects
the meridian of $E(K_{j})$ once (we will exploit this flexibility soon, when appealing
to Lemma~\ref{lem:makingL1irreducible}).  We will denote the multislope induced by $\alpha$
on $\partial X_{0}$  as $\alpha_{0}$.  By construction
$X_{0}(\alpha_{0}) \cong M$.  For $j \in J_{1}$, we will denote the core
of the solid torus attached to $F_{j}$
as $K_{j}$ and the link formed by the cores of the solid tori attached to $\mathcal{T}_{1}$
as $\mathcal{L}_{1}$; no confusion should arise from this notation, as $K_{j}$
and $\mathcal{L}_{1}$ are isotopic to the knots and link denote that way previosly.
We will denote the link formed by the solid tori attached to $\cup_{j \in J_{2}} F_{j}$
as $\mathcal{U}$.  By construction, the components of $\mathcal{U}$
are unknots embedded in the the disjoint balls $D_{j}$, and hence
$\mathcal{U}$ is an unlink.

In order to apply Lemma~\ref{lem:TwistingToGetIrreducible} we need to know that
$\mathcal{L}_{1}$ is irreducible in the complement of $\mathcal{U}$;
this is not quite the case, but we can obtain this by considering only some of the 
components of $\mathcal{U}$.  To that end we prove:

\begin{lem}
\label{lem:makingL1irreducible}
Suppose $S$ is a reducing sphere for $\mathcal{L}_{1}$ in the complement of
$\mathcal{U}$.  Then $S$ bounds a ball $D \subset X_{0}(\alpha_{0})$
so that $D \cap \mathcal{L}_{1} = \emptyset$ and $D$ contains 
at least one component of $\mathcal{U}$.
\end{lem}

\begin{proof}[Proof of Lemma~\ref{lem:makingL1irreducible}]
Let $S$ be a reducing sphere
for $\mathcal{L}_{1}$ in $X_0(\alpha_{0}) \setminus N(\mathcal{U})$;
equivalently, $S$ is a reducing sphere for 
$X_0(\alpha_{0}) \setminus N(\mathcal{U} \cup \mathcal{L}_{1})$.
Fix $1 \leq j \leq m$;
note that either $j \in J_{1}$ or $j \in J_{2}$.  If $j \in J_{2}$,
then by construction $S$ is disjoint from $F_{j}$.  If $j \in J_{1}$,
then $S \cap K_{j} = \emptyset$ (since $K_{j} \subset \mathcal{L}_{1}$).
Since $K_{j}$ is a core of the solid torus attached to $F_{j}$, we may isotope $S$ out of
this solid torus without intersecting $K_{j}$.  After performing this isotopy 
(if necessary) for each $1 \leq j \leq m$,  the reducing sphere $S$
is disjoint from $F_{j}$ for every $1 \leq j \leq m$.  Since $X_{0}(\alpha_{0}) \cong X(\alpha)$,
we may consider $S$ as
a sphere in $X(\alpha)$, where we see that $\mathcal{L} \cap S = \emptyset$.
Hyperbolicity of $X(\alpha)$ and irreducibility of $E(\mathcal{L})$ imply that
$S$ bounds a ball $D \subset X(\alpha)$ so that $\mathcal{L} \cap D = \emptyset$.
It follows that $X_{j}(\alpha_{j}) \cap D = \emptyset$ for $1 \leq j \leq m$;
therefore $\mathcal{T}_{1} \cap D = \emptyset$.  
Hence $S$ bounds a ball in $X_{0}(\alpha_{0})$ that is disjoint from 
$\mathcal{T}_{1}$ and hence from $\mathcal{L}_{1}$.  
On the other hand, $S$ does not bound a ball in $X_0(\alpha_{0}) \setminus N(\mathcal{U})$;
hence $D$ must contain at least one component of $\mathcal{U}$.

This completes the proof of Lemma~\ref{lem:makingL1irreducible}.
 \end{proof}

Let $D$ be as in Lemma~\ref{lem:makingL1irreducible}.  We remove the components of 
$\mathcal{U} \cap D$ from $\mathcal{U}$; to avoid overly complicated notation we do not rename $\mathcal{U}$.
We repeat this process as long as we can; it terminates since the number of components of $\mathcal{U}$ is reduced.  
When it does, $\mathcal{L}_{1}$ is irreducible in the complement
of the unlink $\mathcal{U}$.  By Lemma~\ref{lem:TwistingToGetIrreducible}
we may change the slopes $\alpha_{0}|_{\mathcal{U}}$
so that the exterior of $\mathcal{L}_{1}$ is irreducible.
To avoid overly complicated notation we do not rename $\alpha_{0}$

We will denote as $\mathcal{A}_{0}$ the multislopes induced on $\partial X_{0}$ by 
multislopes of $\mathcal{A}_{J_{1},J_{2}}$ via the procedure described above.
We see that $X_{0}$, $T_{1}$, $\mathcal{T}_{1}$, and $\mathcal{A}_{0}$
fulfill the assumptions
of the proposition.  By Lemma~\ref{lem:TreeOfJSJ}, $|T(X_{0})| < |T(X)|$.  By induction, for each component
$F_{j}$ of $\mathcal{T}_{1}$, 
there is a bounded set of slopes of $F_{j}$ which we will denote as
$B_{F_{j}}$, so that for each $\alpha_{0} \in \mathcal{A}_{0}$,
there is some component $F_{j}$ of $\mathcal{T}_{1}$ with $\alpha_{0}|_{F_{j}} \in B_{F_{j}}$.

Given $j\in J_{1}$, let $S \neq F_{j}$ be a component of $\partial X_{j}$. 
Let $\beta$ be a multislope of $\partial X_{j}$ so that $\beta|_{F_{j}}$
and $\beta|_{S}$ are both $\nos$ and $X_{j}(\beta) \cong T^{2} \times [0,1]$.
We will denote as $B_{\beta}$ the projection of $B_{F_{j}}$ to the slopes of $S$ 
induced by the product structure on $X_{j}(\beta)$.  By the 
$T^{2} \times [0,1]$ Cosmetic Surgery Theorem (\ref{thm:CosmeticSurgeryOnT2XI}),
the set $\cup_{\beta} B_{\beta}$ is bounded (where here the union is taken over all possible multislopes
$\beta$ as above; if there is no  such multislope then $\cup_{\beta} B_{\beta} = \emptyset$).  
We will denote $\cup_{\beta} B_{\beta}$ as $B_{S}$.

Given $\alpha \in \mathcal{A}_{J_{1},J_{2}}$, let $\alpha_{0}$ be the induced multislope
on $X_{0}$ as above.  Let $F_{j}$ be the component for which $\alpha_{0}|_{F_{j}} \in B_{F_{j}}$.
Recall that $K_{j}$ is a core of a solid torus attached to $X_{j}$ which is also a core
of $X_{j}(\alpha_{j})$.  Let $S$ be the component of $\partial X_{j}$ that corresponds
to $K_{j}$.  By definition of $\alpha_{0}|_{F_{j}}$, it is the projection of $\alpha|_{S}$
under the natural projection give by the product structure on $X_{j}(\alpha_{j}) \setminus N(K_{j})$.
Thus $\alpha|_{S} \in B_{S}$.  This show that for every $\alpha \in \mathcal{A}_{J_{1},J_{2}}$,
there is a boundary component $S$, so that $\alpha|_{S} \in B_{S}$.  Since $B_{S}$ are bounded,
this completes the proof of Proposition\ref{pro:TheLastCase}.
\end{proof}

\nocite{*}
\bibliographystyle{alpha} %
\bibliography{LVvsVol} %

\begin{thebibliography}{BHW99}

\bibitem[Ber91]{berge}
John Berge.
\newblock The knots in {$D^2\times S^1$} which have nontrivial {D}ehn surgeries
  that yield {$D^2\times S^1$}.
\newblock {\em Topology Appl.}, 38(1):1--19, 1991.

\bibitem[BHW99]{BHW}
Steven~A. Bleiler, Craig~D. Hodgson, and Jeffrey~R. Weeks.
\newblock Cosmetic surgery on knots.
\newblock In {\em Proceedings of the {K}irbyfest ({B}erkeley, {CA}, 1998)},
  volume~2 of {\em Geom. Topol. Monogr.}, pages 23--34 (electronic). Geom.
  Topol. Publ., Coventry, 1999.

\bibitem[CM01]{caomeyerhoff}
Chun Cao and G.~Robert Meyerhoff.
\newblock The orientable cusped hyperbolic {$3$}-manifolds of minimum volume.
\newblock {\em Invent. Math.}, 146(3):451--478, 2001.

\bibitem[Gab89]{gabai}
David Gabai.
\newblock Surgery on knots in solid tori.
\newblock {\em Topology}, 28(1):1--6, 1989.

\bibitem[GL89]{gordonluecke}
C.~McA. Gordon and J.~Luecke.
\newblock Knots are determined by their complements.
\newblock {\em J. Amer. Math. Soc.}, 2(2):371--415, 1989.

\bibitem[Gro82]{GromovBoundedCohomology}
Michael Gromov.
\newblock Volume and bounded cohomology.
\newblock {\em Inst. Hautes \'Etudes Sci. Publ. Math.}, (56):5--99 (1983),
  1982.

\bibitem[Hat82]{hatcher}
A.~E. Hatcher.
\newblock On the boundary curves of incompressible surfaces.
\newblock {\em Pacific J. Math.}, 99(2):373--377, 1982.

\bibitem[Hil74]{hilden}
Hugh~M. Hilden.
\newblock Every closed orientable {$3$}-manifold is a {$3$}-fold branched
  covering space of {$S^{3}$}.
\newblock {\em Bull. Amer. Math. Soc.}, 80:1243--1244, 1974.

\bibitem[HW08]{hardywright}
G.~H. Hardy and E.~M. Wright.
\newblock {\em An introduction to the theory of numbers}.
\newblock Oxford University Press, Oxford, sixth edition, 2008.
\newblock Revised by D. R. Heath-Brown and J. H. Silverman, With a foreword by
  Andrew Wiles.

\bibitem[Jac80]{JacoBook}
William Jaco.
\newblock {\em Lectures on three-manifold topology}, volume~43 of {\em CBMS
  Regional Conference Series in Mathematics}.
\newblock American Mathematical Society, Providence, R.I., 1980.

\bibitem[Joh79]{Johannson}
Klaus Johannson.
\newblock {\em Homotopy equivalences of {$3$}-manifolds with boundaries},
  volume 761 of {\em Lecture Notes in Mathematics}.
\newblock Springer, Berlin, 1979.

\bibitem[JS79]{jacoshalen}
William~H. Jaco and Peter~B. Shalen.
\newblock Seifert fibered spaces in {$3$}-manifolds.
\newblock {\em Mem. Amer. Math. Soc.}, 21(220):viii+192, 1979.

\bibitem[Kaw88]{kawauchi2}
Akio Kawauchi.
\newblock Imitation of {$(3,1)$}-dimensional manifold pairs.
\newblock {\em S\=ugaku}, 40(3):193--204, 1988.
\newblock Translated in Sugaku Expositions {{\bf{2}}} (1989), no. 2, 141--156.

\bibitem[Kaw89]{kawauchi}
Akio Kawauchi.
\newblock Almost identical imitations of {$(3,1)$}-dimensional manifold pairs.
\newblock {\em Osaka J. Math.}, 26(4):743--758, 1989.

\bibitem[KR11]{kobayashirieck}
Tsuyoshi Kobayashi and Yo'av Rieck.
\newblock A linear bound on the tetrahedral number of manifolds of bounded
  volume (after {J}\o rgensen and {T}hurston).
\newblock In {\em Topology and geometry in dimension three}, volume 560 of {\em
  Contemp. Math.}, pages 27--42. Amer. Math. Soc., Providence, RI, 2011.

\bibitem[Mon74]{montesinos}
Jos{\'e}~M. Montesinos.
\newblock A representation of closed orientable {$3$}-manifolds as {$3$}-fold
  branched coverings of {$S^{3}$}.
\newblock {\em Bull. Amer. Math. Soc.}, 80:845--846, 1974.

\bibitem[Mye82]{Myers}
Robert Myers.
\newblock Simple knots in compact, orientable {$3$}-manifolds.
\newblock {\em Trans. Amer. Math. Soc.}, 273(1):75--91, 1982.

\bibitem[Orl72]{OrlikBook}
Peter Orlik.
\newblock {\em Seifert manifolds}.
\newblock Lecture Notes in Mathematics, Vol. 291. Springer-Verlag, Berlin,
  1972.

\bibitem[RR12]{remigiorieck}
J.~{Remigio-Ju{\'a}rez} and Y.~{Rieck}.
\newblock {The Link Volumes of some prism manifolds}.
\newblock {\em ArXiv e-prints}, May 2012.
\newblock arxiv.org/1205.2783.

\bibitem[RY12]{rieckyamashita}
Y.~{Rieck} and Y.~{Yamashita}.
\newblock {The Link Volume of 3-Manifolds}.
\newblock {\em ArXiv e-prints}, May 2012.
\newblock arxiv.org/1205.1274.

\bibitem[Sak96]{sakai}
Takashi Sakai.
\newblock {\em Riemannian geometry}, volume 149 of {\em Translations of
  Mathematical Monographs}.
\newblock American Mathematical Society, Providence, RI, 1996.
\newblock Translated from the 1992 Japanese original by the author.

\bibitem[Sei33]{seifert}
H.~Seifert.
\newblock Topologie {D}reidimensionaler {G}efaserter {R}\"aume.
\newblock {\em Acta Math.}, 60(1):147--238, 1933.

\bibitem[Ser85]{series}
Caroline Series.
\newblock The modular surface and continued fractions.
\newblock {\em J. London Math. Soc. (2)}, 31(1):69--80, 1985.

\bibitem[Som81]{soma}
Teruhiko Soma.
\newblock The {G}romov invariant of links.
\newblock {\em Invent. Math.}, 64(3):445--454, 1981.

\bibitem[Ter02]{teragaito}
Masakazu Teragaito.
\newblock Links with surgery yielding the 3-sphere.
\newblock {\em J. Knot Theory Ramifications}, 11(1):105--108, 2002.

\bibitem[Thu]{thurston}
William~P Thurston.
\newblock The geometry and topology of 3-manifolds.
\newblock available from msri.org.

\end{thebibliography}
\end{document}